\documentclass{amsart}
\usepackage{amsmath,amsthm,amssymb,latexsym,epic,bbm,comment}
\usepackage{graphicx,enumerate,stmaryrd, xcolor}
\usepackage[all,2cell]{xy}
\xyoption{2cell}

\usepackage[active]{srcltx}
\usepackage[parfill]{parskip}
\UseTwocells
\usepackage[
colorlinks=true,
linkcolor=black, 
anchorcolor=black,
citecolor=black, 
urlcolor=black, 
]{hyperref}
\usepackage[e]{esvect}

\newtheorem{theorem}{Theorem}
\newtheorem{lemma}[theorem]{Lemma}
\newtheorem{corollary}[theorem]{Corollary}
\newtheorem{proposition}[theorem]{Proposition}

\theoremstyle{definition}
\newtheorem{defex}[theorem]{Definition/Example}
\newtheorem{example}[theorem]{Example}
\newtheorem{remark}[theorem]{Remark}

\newtheorem{definition}[theorem]{Definition}

\font\sc=rsfs10
\newcommand{\cC}{\sc\mbox{C}\hspace{1.0pt}}

\newcommand{\cN}{\sc\mbox{N}\hspace{1.0pt}}

\newcommand{\cI}{\sc\mbox{I}\hspace{1.0pt}}

\newcommand{\cD}{\sc\mbox{D}\hspace{1.0pt}}

\newcommand{\cU}{\sc\mbox{U}\hspace{1.0pt}}

\newcommand{\cZ}{\sc\mbox{Z}\hspace{1.0pt}}
\newcommand{\cSt}{{\sc\mbox{K}\hspace{1.0pt}}}

\font\scc=rsfs7
\newcommand{\ccC}{\scc\mbox{C}\hspace{1.0pt}}
\newcommand{\ccD}{\scc\mbox{D}\hspace{1.0pt}}

\newcommand{\ccZ}{\scc\mbox{Z}\hspace{1.0pt}}

\newcommand{\quots}[1]{\vv{#1}}

\newcommand{\pdgcat}{\mathfrak{D}_{p}}

\newcommand{\addcat}{\mathfrak{A}_{\Bbbk}^{\mathrm{f}}}
\newcommand{\preaddcat}{\mathfrak{C}_{\Bbbk}}
\newcommand{\cofcat}{\mathfrak{M}_{p}}

\newcommand{\pamod}{\text{-}\mathrm{pamod}}
\newcommand{\pcmod}{\text{-}\mathrm{pcmod}}

\newcommand{\dmod}{\text{-}\mathrm{dmod}}

\newcommand{\AsubC}{\cC_{A}}

\newcommand{\AsubcC}{\ccC_{A}}

\newcommand{\Csub}[1]{\cC_{#1}}
\newcommand{\cCsub}[1]{\ccC_{#1}}

\newcommand{\adj}[4]{#1\colon #2\rightleftarrows #3\colon #4}
\newcommand{\shift}[1]{\langle #1\rangle}

\newcommand{\Hom}{\mathrm{Hom}}

\newcommand{\End}{\mathrm{End}}
\newcommand{\Fun}{\mathrm{Fun}}
\newcommand{\cEnd}{\sc\mbox{E}\mathit{nd}\hspace{1.0pt}}
\newcommand{\lmod}{\text{-}\mathrm{mod}}
\newcommand{\plmod}{\text{-}\mathrm{mod}_H}
\newcommand{\proj}{\text{-}\mathrm{proj}}

\newcommand{\rad}{\operatorname{rad}}

\newcommand{\add}{\operatorname{add}}
\newcommand{\op}{\mathrm{op}} 
\newcommand{\im}{\operatorname{Im}}
\newcommand{\cof}{\text{-}\mathrm{csf}}

\newcommand{\rE}{\mathrm{E}}
\newcommand{\rF}{\mathrm{F}}
\newcommand{\rG}{\mathrm{G}}
\newcommand{\rH}{\mathrm{H}}
\newcommand{\rK}{\mathrm{K}}
\newcommand{\rL}{\mathrm{L}}
\newcommand{\rM}{\mathrm{M}}
\newcommand{\rN}{\mathrm{N}}
\newcommand{\rP}{\mathrm{P}}
\newcommand{\rS}{\mathrm{S}}

\newcommand{\rX}{\mathrm{X}}
\newcommand{\rY}{\mathrm{Y}}
\newcommand{\rZ}{\mathrm{Z}}

\newcommand{\A}{\mathcal{A}}
\newcommand{\B}{\mathcal{B}}
\newcommand{\C}{\mathcal{C}}
\def\D{{\mathcal{D}}}
\newcommand{\E}{\mathcal{E}}
\newcommand{\F}{\mathcal{F}}

\newcommand{\I}{\mathcal{I}}
\newcommand{\J}{\mathcal{J}}
\def\L{{\mathcal{L}}}
\newcommand{\M}{\mathcal{M}}
\def\S{\mathcal{S}}
\newcommand{\X}{\mathcal{X}}

\newcommand{\Z}{\mathcal{Z}}
\def\St{{\mathcal{K}}}

\newcommand{\del}{\partial}
\newcommand{\id}{\mathrm{id}}
\newcommand{\one}{\mathbbm{1}}

\newcommand{\bfA}{\mathbf{A}}
\newcommand{\bfM}{\mathbf{M}}
\newcommand{\bfG}{\mathbf{G}}
\newcommand{\bfP}{\mathbf{P}}
\newcommand{\bfN}{\mathbf{N}}
\newcommand{\bfR}{\mathbf{R}}
\newcommand{\bfC}{\mathbf{C}}
\newcommand{\bfD}{\mathbf{D}}
\newcommand{\bfI}{\mathbf{I}}
\newcommand{\bfJ}{\mathbf{J}}

\newcommand{\bfZ}{\mathbf{Z}}
\newcommand{\bfInd}{\mathbf{Ind}}

\def\th{{\mathtt{h}}}
\newcommand{\ti}{\mathtt{i}}
\newcommand{\tj}{\mathtt{j}}
\newcommand{\tk}{\mathtt{k}}
\newcommand{\tl}{\mathtt{l}}
\newcommand{\tm}{\mathtt{m}}

\newcommand{\tI}{\mathtt{I}}

\numberwithin{equation}{section}
\numberwithin{theorem}{section}

\begin{document}

\title[Cell 2-Representations and Categorification at Prime Roots of Unity]{Cell 2-Representations and Categorification at Prime Roots of Unity}
\author{Robert Laugwitz}
\address{School of Mathematical Sciences,
University of Nottingham, Nottingham, NG7 2RD, UK}
\email{robert.laugwitz@nottingham.ac.uk}

\author{Vanessa Miemietz}
\address{
School of Mathematics, University of East Anglia, Norwich NR4 7TJ, UK}
\email{v.miemietz@uea.ac.uk}
\urladdr{https://www.uea.ac.uk/~byr09xgu/}

\date{\today}

\begin{abstract}
Motivated by recent advances in the categorification of quantum groups at prime roots of unity, we develop a theory of $2$-representations for $2$-categories, enriched with a $p$-differential, which satisfy finiteness conditions analogous to those of finitary or fiat $2$-categories. We construct cell $2$-representations in this setup, and consider a class of $2$-categories stemming from bimodules over a $p$-dg category in detail. This class is of particular importance in the categorification of quantum groups, which allows us to apply our results to cyclotomic quotients of the categorifications of small quantum group of type $\mathfrak{sl}_2$ at prime roots of unity by Elias--Qi [\emph{Advances in Mathematics} {\bf 288} (2016)]. Passing to stable $2$-representations gives a way to construct triangulated $2$-representations, but our main focus is on working with $p$-dg enriched $2$-representations that should be seen as a $p$-dg enhancement of these triangulated ones.
\end{abstract}

\subjclass[2010]{18D05,18D20,17B10}
\keywords{2-representation theory, enriched 2-categories, categorification at roots of unity, Hopfological algebra}

\maketitle

\section*{Contents}

\makeatletter
\@starttoc{toc}
\makeatother

\section{Introduction}\label{s0}

\subsection{Background}

The quantum group $U_q(\mathfrak{g})$ associated to a finite-dimensional Lie algebra depends on a parameter $q$. Depending on whether $q$ is generic or a root of unity, the representation theory has vastly different features. In particular, if $q$ is a root of unity, then $U_q(\mathfrak{g})$ has a finite-dimensional quotient $u_q(\mathfrak{g})$ which has been used to construct invariants of $3$-dimensional manifolds forming a 3-dimensional topological field theory \cite{RT}. These invariants were originally constructed using Chern--Simons theory and have been linked to the Jones polynomial (for $\mathfrak{g}=\mathfrak{sl}_2$) by Witten \cite{W}.

A seminal paper by Crane--Frenkel \cite{CF} suggested to replace such Hopf algebras as $u_q(\mathfrak{g})$ by \emph{categories} in the pursuit of constructing $4$-dimensional topological quantum field theories by algebraic means. From this, the theory of \emph{categorification} emerged. First progress was made by categorifying quantum groups at a generic parameter $q$, obtaining the so called \emph{$2$-Kac--Moody algebras}, in \cite{Ro, KL09, KL,KL11}. A breakthrough towards realizing Crane--Frenkel's vision was the construction of a categorified Jones polynomial for links in $S^3$ in \cite{Kh2}, known as \emph{Khovanov homology}. However, in order to make $4$-dimensional TQFTs a possibility, a categorification of quantum groups at roots of unity is needed. 

An important idea was introduced by Khovanov \cite{Kh}, who observed that working with algebra objects in the category of modules over a finite-dimensional Hopf algebra $H$ gives a way of categorifying algebras over the Grothendieck ring of the stable category of $H$-modules. In the case where $H=\Bbbk[\del]/(\del^p)$ for a field $\Bbbk$ of characteristic $p$, one obtains, when passing to the stable category, the relation $1+q+\ldots+q^{p-1}=0$ of the ring of cyclotomic integers. The relation $\del^p=0$ is reminiscent of the classical  equation $\del^2=0$ of complexes in homological algebra, and there is a close relationship to $N$-complexes which were considered as a new homology theory in \cite{M1,M2}, and in \cite{Ka}, where the possibility of a link to the representation theory of quantum groups was suspected.

The theory proposed in \cite{Kh} was subsequently further developed in \cite{Qi}, where $p$-dg algebras and their modules are formally introduced, in analogy to the theory of dg modules over dg algebras generalizing complexes of modules over an algebra. 
Milestones in realizing the program of categorifying quantum groups at roots of unity were reached in \cite{KQ}, \cite{EQ}, \cite{EQ2} with the categorification of Lusztig's idempotented quantum groups (see \cite{Lu}) $\dot{u}_q(\mathfrak{sl}_2)^+$, $\dot{u}_q(\mathfrak{sl}_2)$, and $\dot{U}_q(\mathfrak{sl}_2)$ for $q$ a \emph{prime} root of unity, obtained by defining a $p$-differential on the $2$-morphism spaces of the corresponding generic $2$-Lie algebra. In \cite{EQ}, categorifications of the graded simple $\dot{u}_q(\mathfrak{sl}_2)$-modules for highest weight $\lambda=0,\ldots, p-1$ are also described.

Nowadays, categorification is often formulated in the language of $2$-categories, since algebras can be viewed as a category with idempotents as objects, which is more natural from a categorical point of view. 
The idea to study representations of categorified quantum groups appears in \cite{Ro}. 
A systematic study of $2$-representations over certain classes of $2$-categories was initiated in \cite{MM1}, with the study of so-called finitary and fiat $2$-categories. These provide $2$-categorical analogues of finite-dimensional algebras and hence provide a first step towards the original vision in \cite{CF} of finding a categorification of finite-dimensional Hopf algebras (called \emph{Hopf categories}) and studying their representation theory. 
 
In \cite{MM5}, \emph{simple transitive} $2$-representations are defined and shown to be a suitable $2$-analogue of simple representations of an algebra, in the sense that they satisfy an appropriate generalisation of the classical Jordan--H\"older theorem. The \emph{cell $2$-representations} originally defined in \cite{MM1}, and inspired by cell representations of a cellular algebra, are examples of such simple transitive $2$-representations. Moreover, for large classes of $2$-categories (including $2$-Lie algebras, see \cite[Section~7.2]{MM5}), they are shown to exhaust all simple transitive $2$-representations (e.g. \cite{MM5,MM6, MMZ2, Zi}), though this is known to not be true in general (see e.g. \cite{Zh, MaMa,KMMZ}). For a survey on finitary $2$-representation theory, see \cite{Maz}.

The present paper initiates a systematic study of $2$-representations of so called $p$-dg $2$-categories, in order to contribute towards attaining some of the goals of the categorification program. This approach adopts the spirit of the series of papers started in \cite{MM1}, but works with enrichments by $p$-differentials in order to connect to the work of \cite{EQ}, and to provide a setup suitable for categorification of algebras over the ring of cyclotomic integers. Further, the construction of $p$-dg $2$-categories and their $p$-dg $2$-representations gives a way to obtain triangulated $2$-categories and $2$-representations compatible with the triangulated structure by passing to stable categories.

Another point of view is that $p$-dg categories are a natural generalization of differential graded categories which are an important tool in contemporary algebraic geometry. By specializing $p=2$, one recovers dg categories in characteristic two, and adjusting signs appropriately, our results remain valid for more general dg $2$-categories, which we expect to emerge as objects of study in their own right in the future.

\subsection{Summary}

This paper studies $p$-dg enriched $2$-categories and introduces cell $2$-representations for a class of such structures satisfying finiteness conditions. Basic properties and the relation to additive cell $2$-representations are investigated. Furthermore, structural results about passing to stable $2$-representations which are compatible with the triangulated structure of the stable $2$-categories are included.

Section \ref{pdgcatsection} introduces all technical results on the level of 1-categories enriched with $p$-differentials --- so called $p$-dg categories --- and the study of their compact semi-free modules, which are the appropriate analogue of free modules in this context, cf. \cite{Qi, EQ}. By \cite{Qi}, the stable category of (compact) cofibrant $p$-dg modules is equivalent as a triangulated category to the (compact) derived category of all $p$-dg modules, and compact cofibrant $p$-dg modules are the idempotent completion of compact semi-free modules. While the desired structure of modules over cyclotomic integers can only be seen by passing to the compact derived category and taking the Grothendieck group, we, for the most part, work on the level of the ($p$-dg enriched) category of compact semi-free modules, whose idempotent completion should be seen as a $p$-dg enhancement of the compact derived category in the spirit of \cite{BK}.

In order to work with small categories, we introduce a combinatorial category equivalent to the $p$-dg enriched category of compact semi-free modules over a $p$-dg category $\C$.
This $p$-dg category $\overline{\C}$ should be interpreted as a $p$-dg analogue of the dg category of (one-sided) twisted complexes of \cite{BK}. For $p$-dg categories of this form, we study different finiteness conditions motivated on one hand by the theory of finitary $\Bbbk$-linear categories (forgetting the differential should yield an idempotent complete (Karoubian) category with finitely many indecomposables and finite-dimensional morphism spaces), and on the other hand by the fantastic filtrations which are crucial in \cite{EQ}. We introduce (relative) tensor products, and a closure under $p$-dg quotients which can be seen as a $p$-dg enrichment of the abelianization of a $p$-dg category of the form $\overline{\C}$ (see Sections \ref{tensorproducts}, \ref{pdgquots}). Moreover, we describe how to pass to the stable category associated to a $p$-dg category of the form $\overline{\C}$ in Section \ref{derivedcats}. 

The next section, Section \ref{pdg2cats}, gathers preliminary results about the kind of $p$-dg $2$-categories for which we will study $p$-dg $2$-representations in Section \ref{pdg2repchapter}. We distinguish \emph{locally finite} and \emph{strongly finitary} $p$-dg $2$-categories $\cC$. These refer to two sets of finiteness assumptions imposed on the structure of the $1$-hom $p$-dg categories $\cC(\ti,\tj)$.
Formal definitions and results on $p$-dg $2$-representations are collected in Section \ref{pdg2repchapter}. In particular, we include two versions of the Yoneda Lemma for $p$-dg $2$-representations (one version for $p$-dg quotient completed $p$-dg categories). 

The core of the paper is formed by Sections \ref{cellchapter} and \ref{CAsec}. In Section \ref{cellchapter}, we introduce cell $2$-representations for $p$-dg $2$-categories which are strongly finitary. That is, when forgetting the $p$-differential, the morphisms categories are idempotent complete and have only finitely many indecomposable $1$-morphisms up to isomorphism. In this context, cell $2$-representations are simple transitive (for a generalization of this concept in an appropriate sense). Forgetting the $p$-differential does, in general, not give a simple transitive additive $2$-representation, but \cite[Theorem~4]{MM6} can be applied to show that it is an inflation by a local algebra of a simple transitive $2$-representation. Moreover, under some technical assumptions, \cite[Theorem~16]{MM3} can be used to prove in Section \ref{pdgendocats} that the endomorphism categories of $p$-dg cell $2$-representations are $p$-dg isomorphic to the enriched category $\Bbbk\plmod$.

A crucial role in \cite{MM3}--\cite{MM6} is played by a certain $2$-category $\cC_A$ constructed from projective bimodules over a finite-dimensional algebra $A$. In \cite{MM3}, it is shown that any fiat $2$-category that is simple in a suitable sense is biequivalent to (a slight variation) of some $\cC_A$. In particular, $2$-Lie algebras are, in some sense, built from many $\cC_{A_\lambda}$ where $\lambda$ runs over positive integral weights and $A_\lambda$ is the product of all cyclotomic quiver Hecke algebras associated to $\lambda$. The fact that all simple transitive $2$-representations of $\cC_A$ are cell $2$-representations (\cite[Theorem 15]{MM5} and \cite[Theorem 5.2]{MMZ2}) implies the same result for  $2$-Lie algebras. This motivates the construction, in Section \ref{CAsec}, of $p$-dg analogues $\cC_\A$ of such $2$-categories associated to a $p$-dg category $\A$, and, under appropriate assumptions, we prove that the cell $2$-representation again is the natural (or defining) representation and thus, when forgetting about the differential, one recovers the cell $2$-representation of the underlying finitary $2$-category.

To conclude the paper, Section \ref{sl2cat} applies some results to the cyclotomic quotient ${\cU_\lambda}$ of the $p$-dg $2$-category $\cU$ introduced by \cite{EQ} to categorify the small quantum group associated to $\mathfrak{sl}_2$. In particular, we analyze the $p$-dg cell $2$-representations of this categorification and show that the subquotient idempotent completion of the restriction of  ${\cU_\lambda}$ to its largest cell is $p$-dg biequivalent to $\cC_{\B}$ where $\B$ is constructed from coinvariant algebras. Thus the $p$-dg cell $2$-representation associated to the cell indexed by $\lambda$ is indeed the natural $2$-representation, which categorifies the simple graded $\dot{u}_q(\mathfrak{sl}_2)$-module of highest weight $\lambda$.  It further follows that the endomorphism categories of these $p$-dg cell $2$-representations are simple, i.e. $p$-dg equivalent to $\Bbbk\plmod$.

\subsection{Acknowledgments}

Part of the work of R.L. on this paper was supported by an EPSRC Doctoral Prize Grant at the University of East Anglia, and by the LMS (Research in Pairs --- Scheme~4). Parts of this article were written during V.M.'s visit to MPI, Bonn, whose hospitality and support is gratefully acknowledged.
The authors thank Ben Elias, Gustavo Jasso, and You Qi for helpful and  stimulating discussions.

\section{\texorpdfstring{$p$}{p}-dg categories}\label{pdgcatsection}

In this section, we collect basic results about $p$-dg categories and their modules and define different finiteness conditions for such categories which will be imposed in later sections of this paper. We also introduce all technical tools on the 1-categorical level, such as a bimodule tensor product action, and the closure under $p$-dg quotients, which will be needed later.

\subsection{The \texorpdfstring{$2$}{2}-category of \texorpdfstring{$p$}{p}-dg categories}\label{pdgcats}

Let $\Bbbk$ be an algebraically closed field of characteristic $p$.
A $\Bbbk$-linear additive category  $\C$ is called a {\bf $p$-dg category} if homomorphism spaces are enriched over the symmetric monoidal closed category $H\lmod$ of finite-dimensional graded $H$-modules for the Hopf algebra $H=\Bbbk[\del]/(\del^p)$ where the degree of $\del$ is chosen to be $2$,  as in \cite{EQ}. We can regard the collection of finite-dimensional graded $H$-modules itself as a $p$-dg category $\Bbbk\lmod_{H}$ whose objects are finite-dimensional graded $H$-modules, and morphisms are all $\Bbbk$-linear maps $f$ equipped with the $H$-action $\del(f) = \del\circ f-f\circ \del$. This means that the homs in $\Bbbk\lmod_{H}$ are the internal homs obtained from the closed monoidal structure \cite[Section 1.6]{K}.

Given a $p$-dg category $\C$, denote by $\Z\C$ the category with the same objects but only those morphisms which have degree $0$ and are annihilated by the action of $\del$.  We refer to the morphisms of $\Z\C$ as {\bf $p$-dg morphisms}. The category $\Z(\Bbbk\lmod_H)$ recovers $H\lmod$.

We say two objects $X$ and $Y$ are  {\bf isomorphic as $p$-dg objects} or  {\bf $p$-dg isomorphic} if there is an isomorphism $f :X \to Y$  of degree zero such that both $f$ and $f^{-1}$ are annihilated by $\del$. That is, $X$ and $Y$ are $p$-dg isomorphic if and only if they are isomorphic in $\Z\C$.

\begin{lemma}\label{pdgisolemma}
An isomorphism $f$ in $\C$ is a $p$-dg isomorphism if and only if it has degree $0$ and is annihilated by $\del$.
\end{lemma}
\begin{proof}
We require that composition in $\C$ is a morphism of graded $H$-modules, i.e. $\partial (gf)=\partial(g)f+g\partial(f)$. Hence, if $f$ is invertible in $\C$ and $\partial f=0$, it follows that
\begin{align}
0=\partial(\id)=\partial(f^{-1})f.
\end{align}
This implies that $\partial (f^{-1})=0$ as $f$ is an isomorphism. Further, the inverse has degree zero precisely if $f$ has degree zero.
\end{proof}

A {\bf $p$-dg functor} between $p$-dg categories is an additive $\Bbbk$-linear functor $M:\C\to \C'$ such that the map $\Hom_{\C}(X,Y) \to \Hom_{\C'}(M(X),M(Y))$ is a degree zero morphism of graded $H$-modules. Morphisms between two $p$-dg functors are natural transformations of ($\Bbbk$-linear) functors. This turns the morphism space between two $p$-dg functors $M$ and $N$ into a graded $H$-module with the action of $\del$ on a $\Bbbk$-linear natural transformation $\psi:M\to N$ given by $\psi_X:M(X)\to N(X)$ simply being the action on $\Hom_{\C'}(M(X),N(X))$, for each object $X$. Indeed, $\partial \psi$ defined this way is a natural transformation, as for a given morphism $f\colon X\to Y$ in the source category of the functors we compute
\begin{align*}
\partial\psi_{Y}\circ Mf&=\partial(\psi_Y\circ Mf)-\psi_Y\circ M(\partial f)\\
&=\partial(Nf\circ \psi_X)-N(\partial f)\circ \psi_X\\
&=Nf\circ \partial \psi_X,
\end{align*}
where the first (and third) equality uses that $M$ (respectively, $N$) is a $p$-dg functor.

We say that a $p$-dg functor $M\colon \C\to \C'$ is part of a {\bf $p$-dg equivalence} if there is a $p$-dg functor $M'\colon \C'\to \C$ such that $M\circ M'$ and $M'\circ M$ are $p$-dg isomorphic to the respective identity functors. We further call a $p$-dg functor $F\colon \C\to \D$  {\bf $p$-dg dense} if for any object $Y$ in $\D$ there exists an object $X$ in $\C$ and a $p$-dg isomorphism $F(X)\cong Y$. Note that a full, faithful, and dense $p$-dg functor is part of an equivalence of categories using the usual proof. However,  assuming $p$-dg density we have a stronger result:

\begin{lemma}\label{pdgequivlemma}
A $p$-dg functor $F\colon \C\to \D$ is full, faithful, and $p$-dg dense if and only if it is (part of) a $p$-dg equivalence.
\end{lemma}
\begin{proof}
Let $F$ be full, faithful, and $p$-dg dense. Note that fully  faithfulness implies that there is a $p$-dg isomorphism
\begin{align*}
\Hom_\C(X,Y) \cong \Hom_\D(FX,FY),
\end{align*}
using Lemma~\ref{pdgisolemma}. Using the same lemma, it suffices to define a $p$-dg functor $G\colon \D\to \C$ and $p$-dg isomorphisms $\eta\colon \one_\D \to GF$, $\varepsilon \colon \one_\C\to FG$. Given an object $Y$ of $\D$ we find a $p$-dg isomorphism $\eta_Y\colon Y\to F(X)$ by $p$-dg density, and set $G(Y):=X$. Given a morphism $g\colon Y \to Y'$ in $\D$ we can use fully faithfulness of $F$ to obtain a unique morphism $f\colon G(Y) \to G(Y')$ such that  $F(f)=\eta_{Y'}\circ g\circ \eta_Y^{-1}$. We define $G(g)=f$. Uniqueness ensures that $G$ indeed is a functor. It is a $p$-dg functor as $\partial (\eta_{Y'}\circ g\circ \eta_Y^{-1})=\eta_{Y'}\circ \partial(g)\circ \eta_Y^{-1}$ and the isomorphism induced by $F$ on morphism spaces commutes with the differential. By construction, $\eta$ is a $p$-dg isomorphism. Given an object $X$ in $\C$, the other equivalence $\varepsilon_C$ can be obtained by applying fully faithfulness to the morphism $\eta_{F(X)}$. Arguing similarly, one sees that $\varepsilon$ defines a natural isomorphism consisting of $p$-dg isomorphisms, hence a $p$-dg isomorphism. This shows that $F$ is part of a $p$-dg equivalence.
The inverse implication is straightforward. 
\end{proof}

We say that an object in a $p$-dg category is {\bf $\Bbbk$-indecomposable} if it has no non-trivial idempotents.
Similarly, an object is  {\bf $p$-dg indecomposable} if its only idempotent endomorphisms which are annihilated by $\del$ are $\id$ and $0$. We say a $p$-dg category $\C$ is {\bf $p$-dg idempotent complete} if it is complete with respect to idempotent $p$-dg morphisms. 
A completion under more idempotents, so-called \emph{subquotient idempotents}, will be discussed in Section \ref{fincat}.

Throughout this article, all $p$-dg categories will be assumed to be small 
and closed under shifts. By being closed under shifts we mean  that for every object $X$ and natural number $n$, there exists an object $X\langle n\rangle$, such that $\Hom_\C(Y,X\langle n\rangle)=\Hom_\C(Y,X)\langle n\rangle$, where, if a morphism $f$ is homogeneous of degree $k$, the morphism $f\langle n\rangle$ is homogeneous of degree $k-n$. 
We will now introduce basic finiteness assumptions on $p$-dg categories.

\begin{definition}\label{pdgfinitary}
We call a $p$-dg category $\C$  {\bf locally finite} if 
morphism spaces between any two $p$-dg indecomposable objects are finite-dimensional.
\end{definition}

The collection of locally finite $p$-dg categories together with $p$-dg functors and their morphisms form a $2$-category, denoted by $\pdgcat$. Note that $\pdgcat$ is enriched over $p$-dg categories, meaning that the category of $p$-dg functors between two $p$-dg categories $\C$ and $\D$ again forms a 
$p$-dg category, which we denote by $\Fun(\C,\D)$.

Forgetting the $H$-module structure and the grading, we denote the underlying $\Bbbk$-linear additive category by $[\C]$. This gives rise to a forgetful $2$-functor $[~~]$ from $\pdgcat$ to the $2$-category with objects locally finite $\Bbbk$-linear categories, $1$-morphisms additive $\Bbbk$-linear functors, and $2$-morphisms natural transformations.

\subsection{Compact semi-free \texorpdfstring{$p$}{p}-dg modules}\label{cofibrantpdgmods}

A (locally finite-dimensional) {\bf $p$-dg module} over a $p$-dg category $\C$ is a $p$-dg functor  $M:\C\to \Bbbk\lmod_H$, i.e the map 
\begin{align}
\Hom_{\C}(X,Y) \to \Hom_{\Bbbk}(M(X),M(Y))\label{lowerstarpdgfunctor}
\end{align}
is a degree zero morphism of graded $H$-modules. Morphisms of $p$-dg modules are $\Bbbk$-linear natural transformations. This makes the category $\C\plmod$ of $p$-dg modules a $p$-dg category.

If $\C$ is locally finite, then to each object $X\in \C$, we can associate the $p$-dg module $P_X =\Hom_\C(X,-)$ over $\C$, which maps an object $Y$ to the graded $H$-module $\Hom_\C(X,Y)$, and a morphism $f\in \Hom_\C(Y,Z)$ to the morphism $$P_X(f)=f_\ast: \Hom_\C(X,Y)\to \Hom_\C(X,Z): g\mapsto f\circ g.$$
In particular, the morphism $\Hom_\C(Y,Z) \to \Hom_\Bbbk(\Hom_\C(X,Y), \Hom_\C(X,Z))$ we obtain commutes with the differentials on the respective homomorphism spaces.

A $p$-dg module $M$ over a locally finite 
$p$-dg category $\C\in \pdgcat$ is called {\bf semi-free} if it admits a filtration by $p$-dg subfunctors  $0=F_0\subset F_1\subset F_2 \subset \ldots \subset M$ such that each subquotient  $F_i/F_{i-1}$ is $p$-dg isomorphic to 
$P_{X_i}$ for some $X_i\in \C$. A semi-free $\C$-module is called {\bf compact} if it has a finite filtration of this form. We denote the full subcategory on compact semi-free modules in the category of $p$-dg $\C$-modules $\C\cof$. Note that this, again, is a $p$-dg category.

We remark that the definition of semi-free modules is related to modules with property (P) for a $p$-dg algebra in \cite[Definition 6.3]{Qi}, and that of compact semi-free modules to finite cell modules in \cite[2.5]{EQ}. For a $p$-dg category, a similar  definition was given in \cite[Section~4.6]{EQ2} where compact semi free modules are referred to as finite cell modules. One crucial difference is that we consider enriched homs, so to obtain the categories of finite cell modules, we need to pass to $\Z(\C\cof)$.

The following definition is a generalisation of the notion of (one-sided) twisted complexes over a dg category, due to Bondal and Kapranov \cite{BK}.  The analogy between compact semi-free $\C$-modules (denoted there by finite cell modules) and one-sided twisted complexes was also pointed out in \cite{QS}.

\begin{definition}\label{barcat}
For a $p$-dg category $\C\in \pdgcat$, we define $\overline{\C}$ as the $p$-dg category whose
\begin{itemize}
\item objects are given by pairs $(\bigoplus_{m=1}^s F_m, \alpha=(\alpha_{k,l})_{k,l})$ where the $F_m$ are objects in $\C$ and $\alpha_{k,l}\in \Hom_\C(F_l,F_k)$, $\alpha_{k,l} = 0$ for all $k\geq l$ such that the matrix $\del \cdot \mathtt{I}_s +\left((\alpha_{kl})_*\right)_{kl}$ acts as a $p$-differential on $\bigoplus_{m=1}^s P^\op_{F_m}$ in  $\C^\op\cof$ (here $\mathtt{I}_s$ is the identity matrix);
\item morphisms are matrices of morphisms between the corresponding objects, with the differential of 
$$\gamma=(\gamma_{n,m})_{n,m} \colon \left(\bigoplus_{m=1}^s F_m, \alpha=(\alpha_{k,l})_{k,l}\right)\longrightarrow \left(\bigoplus_{n=1}^t G_n, \beta=(\beta_{k,l})_{k,l}\right)$$
defined as 
$$ \partial\left((\gamma_{n,m})_{n,m} \right):= (\partial \gamma_{n,m}+(\beta\gamma)_{n,m}-(\gamma\alpha)_{n,m})_{n,m}.$$
\end{itemize}
\end{definition}

Note that the matrix $\del \cdot \tI_s +\alpha_*$ acts as a $p$-differential on the functor $\bigoplus_{m=1}^s P^\op_{F_m}$ --- and hence gives a $p$-dg module --- if and only if $(\del \cdot \tI_s +\alpha_*)^p$ acts by zero. This can in fact be checked  (as we will see in Lemma \ref{matcatequiv}) by verifying that 
\begin{align}
(\del \cdot \tI_s +\alpha_*)^p(\delta_{k,l}\id_{F_m})_{k,l}=(\del \cdot \tI_s +\alpha)^{p-1}\alpha=0, \label{matrixdiffeqn}
\end{align}
which gives conditions on the entries of $\alpha$ and their differentials. In particular, this implies that $\del^{p-1}(\alpha_{i,i+1})=0$ for all $i=1, \ldots, s-1$ and that the degree of each $\alpha_{i,j}$ is $2$.
Observe that our grading conventions differ from \cite{BK} and are closer to those of \cite{Se}. This is made possible by the fact that our $p$-dg categories are closed under grading shift.

We remark that although the symbol $\oplus$ is used in the notation of objects in $\overline{\C}$ this does not denote biproducts internal to $\C$ but rather lists of objects.

The following explicit additive structure on $\overline{\C}$ will be used later: Consider two objects $X=(\bigoplus_{m=1}^s F_m, \alpha=(\alpha_{k,l})_{k,l})$ and $Y=(\bigoplus_{n=1}^t G_n, \beta=(\beta_{k,l})_{k,l})$ in $\overline{\C}$, then the object
$$X\oplus Y=\left(\bigoplus_{m=1}^s F_m\oplus \bigoplus_{n=1}^t G_n, \begin{pmatrix}\alpha&0\\0&\beta\end{pmatrix}\right)
$$
is an object in $\overline{\C}$ which satisfies the universal properties of a biproduct. We use the common notation that if $\gamma=(\gamma_{n,m})_{n,m}\colon X\to Y$ and $\eta=(\eta_{n,m'})_{n,m'}\colon X'\to Y$ are morphisms in $\overline{\C}$, then $\begin{pmatrix}
\gamma & \eta
\end{pmatrix}
$ gives a morphism $X\oplus X'\to Y$; and if $\varphi=(\varphi_{n,m})_{n',m}\colon X\to Y'$ is another morphisms in $\overline{\C}$, then $\begin{pmatrix}
\gamma \\ \varphi
\end{pmatrix}$ is a morphism $X\to Y\oplus Y'$.

Note that with $\C$, $\overline{\C}$ again lies in $\pdgcat$.

\begin{lemma}\label{matcatequiv}
The categories $\overline{\C}$ and  $\C^{\op}\cof$ are $p$-dg equivalent for $\C$ in $\pdgcat$.
\end{lemma}
\begin{proof}
We define a $p$-dg functor
$\rP \colon \overline{\C} \to \C^\op\cof$  by mapping an object $X=(\bigoplus_{m=1}^s F_m, (\alpha_{k,l})_{k,l})$ to the $p$-dg $\C$-module given on $\rP X=\bigoplus_{m=1}^s P^\op_{F_m}$ with differential $\del_{\rP X}:=\del \cdot \mathtt{I}_s +\left((\alpha_{k,l})_*\right)_{k,l}$. This is an object of $\C^\op\cof$ using the filtration $F_n:=\bigoplus_{m=1}^n P^\op_{F_m}$ for $n=1,\ldots, s$. Indeed, using the upper-triangular shape of the matrix $(\alpha_{k,l})_{k,l}$ we see that the $F_n$ give a chain of $p$-dg subfunctors of $\rP X$, and the factors are $p$-dg isomorphic to $P^\op_{F_m}$. For a morphism 
$$\gamma =(\gamma_{n,m})_{n,m} \colon X=\left(\bigoplus_{m=1}^s F_m, \alpha=(\alpha_{k,l})_{k,l}\right)\longrightarrow Y=\left(\bigoplus_{n=1}^t G_n, \beta=(\beta_{k,l})_{k,l}\right),$$
we let $\rP \gamma$ be the $\Bbbk$-natural transformation given by $((\gamma_{n,m})_*)_{n,m} \colon \rP X\to \rP Y$. As $(~~)_*$ is covariant with respect to composition, this gives a functor $\rP$ as desired. It is a $p$-dg functor as
\begin{align*}
\partial \rP(\gamma)&=\partial \left((\gamma_{n,m})_*\right)_{n,m}\\
&=\left(\partial\cdot\mathtt{I}_t +\left((\beta_{k,l})_*\right)_{k,l}\right)\circ\gamma - \gamma \circ\left(\partial \cdot\mathtt{I}_s +\left((\alpha_{k,l})_*\right)_{k,l}\right)\\
&=\del \circ ((\gamma_{n,m})_*)_{n,m} -((\gamma_{n,m})_*)_{n,m}\circ \del + \left(\left((\beta\gamma-\gamma\alpha)_{n,m}\right)_*\right)_{n,m}\\
&=\left((\partial\gamma_{n,m})_*\right)_{n,m}+\left(\left((\beta\gamma-\gamma\alpha)_{n,m}\right)_*\right)_{n,m}\\
&=\rP\left( (\partial\gamma_{n,m})_{n,m}+(\beta\gamma-\gamma\beta)_{n,m}\right)\\
&=\rP\left( \partial\gamma\right),
\end{align*}
where we use that $(~~)_*$ is a $p$-dg functor in the fourth equality, see (\ref{lowerstarpdgfunctor}).
Referring to Lemma~\ref{pdgequivlemma}, it remains to show that $\rP$ is fully faithful and $p$-dg dense.

Let $M$ be a semi-free module over $\C^\op$, with a filtration $0\subset M_1\subset M_2\subset \ldots\subset M_s=M$ such that there are $p$-dg isomorphisms $M_i/{M_{m-1}}\cong P^\op_{F_m}$ for $m=1,\ldots,s$. We can choose an isomorphism (\emph{not} a $p$-dg isomorphism) of $\psi\colon M\stackrel{\sim}{\to} \bigoplus_{m=1}^s P^\op_{F_m}$, and consider the conjugation $\phi\circ \del_M\circ \phi^{-1}$ of the differential $\del=\del_M\colon M\to M$, which gives a map $\del_{m',m}\colon P^\op_{F_m}\to P^\op_{F_{m'}}$ for any $m'\leq m$. By adding $\del_{m,m'}=0$ for $m'>m$, the maps thus obtained give an upper triangular matrix $(\del_{m',m})_{m',m}$. As $M$ is a $p$-dg functor, we require that for any morphism $f\colon A\to B$ in $\C$, $\partial f$ is mapped to $\partial M(f)$. This condition translates under isomorphism to the matrix condition
\begin{align*}
(\partial_{m',m})_{m',m}f^*\mathtt{I}_s-f^*\mathtt{I}_s(\partial_{m',m})_{m',m}=(\partial f)^*\mathtt{I}_s.
\end{align*}
Evaluating componentwise, we find that
\begin{align*}
\partial_{m',m}f^*-f^*\partial_{m',m}&=0, &\forall m\geq m',\\
\partial_{m,m}f^*-f^*\partial_{m,m}&=(\partial f)^*,
\end{align*}
from which we conclude that in fact $\partial_{m,m}=\partial_{P^\op_{F_m}}$, i.e. the given differential on the $m$-th subfactor, and that $\partial_{m',m}\colon P^\op_{F_{m}}\to P^{\op}_{F_{m'}}$ is a natural transformation. Hence $\partial_{m',m}$ is induced by a morphism $\alpha_{m',m}\colon F_{m}\to F_{m'}$. We define a new object $M'$ in $\C^{\op}\cof$ as the direct sum $\bigoplus_{m=1}^sP^{\op}_{F_m}$ together with the differential given by the matrix $\partial I_s+(\alpha_{m',m})_{m',m}$. This object lies in the image of the functor $\rP$. We now observe that the isomorphism $\psi\colon M\to \bigoplus_{m=1}^sP^\op_{F_m}$ previously chosen in fact gives a $p$-dg isomorphism $M\to M'$. Indeed, by construction we have $(\partial_{m',m})_{m',m}=\psi\circ \partial_{M}\circ \psi^{-1}$, which gives $(\partial_{m',m})_{m',m}\circ \psi-\psi\circ \partial_{M}=0$; that is, $\del \psi=0$ when viewed as a morphism $M\to M'$. Hence $\rP$ is $p$-dg dense.

Finally, it is clear that $\rP$ is faithful, and using $p$-density just proved, we can construct a $p$-dg isomorphism $\Hom_{\C^{\op}\cof}(M,N)\cong \Hom_{\C^\op\cof}(M',N')$ using conjugation with $p$-dg isomorphisms, where $M'$, $N'$ are in the image of $\rP$. But a morphism $M'\to N'$ corresponds under the enriched Yoneda lemma to a unique matrix of morphisms, i.e. lies in the image of $\rP$. This shows fullness of $\rP$ and concludes the proof.
\end{proof}

Given a $p$-dg category $\C\in\pdgcat$, denote by $\iota_{\C}\colon \C\to \overline{\C}$ the canonical inclusion, which is a fully faithful $p$-dg functor. Observe that this $p$-dg functor descends to an equivalence $ [\C]\to [\overline{\C}]$.

Note that the functor $\iota$ displays $\C$ as an additive subcategory of $\overline{\C}$ in the sense that the addition of morphisms $\Hom_\C(X,Y)$ is the same as that on $\Hom_\C\left((X,0), (Y,0)\right)$, but it only commutes with biproducts up to $p$-dg isomorphism: the sum $X\oplus Y$ in $\C$ is mapped to $(X\oplus Y,0)$ which is $p$-dg isomorphic to $\left(X\oplus Y, \begin{pmatrix}
0&0\\ 0&0
\end{pmatrix}\right)$. Explicit mutually inverse $p$-dg isomorphisms are given by $\begin{pmatrix}p_X\\ p_Y
\end{pmatrix}$, $\begin{pmatrix}i_X&i_Y
\end{pmatrix}$ where $p_X,p_Y$ denote the projections and $i_X, i_Y$ the injections of $X\oplus Y$ in $\C$.

The following lemma shows that one can extend $p$-dg functors and morphisms of functors to the corresponding categories of semi-free modules.

\begin{lemma}\label{extensionlemma}$~$
Let $\C,\D\in \pdgcat$ be two $p$-dg categories.
\begin{enumerate}[(i)]
\item\label{extensionlemma1} Let $\rF\colon \C\to \D$ be a $p$-dg functor. Then there exists an induced $p$-dg functor $\overline{\rF}\colon \overline{\C}\to\overline{\D}$, such that $\overline{\rF}\circ\iota_\C=\iota_\D\circ\rF$. 
Moreover, the functor $\overline{\rF}$ is the unique $p$-dg functor $\overline{\C}\to \overline{\D}$ descending to the canonical functor $[\overline{\C}]\to [\overline{\D}]$ induced by $\rF$ via direct sums and matrices of morphisms, and the assignment is functorial.

\item\label{extensionlemma2} Given a natural transformation of $p$-dg functors $\lambda\colon \rF\to \rG$, there exists an induced natural transformation $\overline{\lambda}\colon \overline{\rF}\to \overline{\rG}$ such that $\overline{\lambda}_C=\lambda_C$ for any object $C\in \C$. 
Further, $\overline{\del(\lambda)}=\del(\overline{\lambda})$ and $\overline{\lambda}$ is the unique natural transformation such that $[\overline{\lambda}]$ corresponds to $\lambda$ under the equivalence between $[\overline{\C}]$ and $[\C]$. 

The assignment is functorial with respect to horizontal and vertical composition and induces an isomorphism $\Hom_{\Fun(\C,\D)}(\rF,\rG) \cong \Hom_{\Fun(\overline{\C},\overline{\D})}(\overline{\rF},\overline{\rG}) $.

\item\label{extensionlemma3} The $p$-dg functor $\iota_{\overline{\C}}$ gives a $p$-dg equivalence of $\overline{\C}$ and $\overline{\overline{C}}$.\end{enumerate}

The content of this lemma can be summarized by stating that $\overline{(~~)}\colon \pdgcat\to \pdgcat$ is a strict $p$-dg $2$-functor of $p$-dg $2$-categories using the language introduced below in Section~\ref{pdg2cat}.
\end{lemma}
\begin{proof}
To prove part (i), let $X=(\bigoplus_{m=1}^s G_m, \alpha=(\alpha_{k,l})_{k,l})$ be an object of $\overline{\C}$. We define $\overline{\rF}(X):=(\bigoplus_{m=1}^s \rF(G_m), \overline{\rF}(\alpha)=(\rF\alpha_{k,l})_{k,l}$. This is well-defined as the matrix $\partial\cdot \tI_s + ((\rF\alpha_{k,l})_*)_{k,l}$ acts as a $p$-differential on $\bigoplus_{m=1}^s P^\op_{\rF(G_m)}$. This can be seen using (\ref{matrixdiffeqn}). The conditions on $\partial\cdot \tI_s + ((\rF\alpha_{k,l})_*)_{k,l}$ to act as a $p$-differential are given by applying $\rF$ to the corresponding equation in terms of $\alpha$, using that the $p$-dg functor $\rF$ commutes with composition and $\del$.

For a morphism $\gamma \colon X=(\bigoplus_{m=1}^s G_m, \alpha)\to Y=(\bigoplus_{m'=1}^t K_{m'}, \beta)$ in $\overline{\C}$ given by a matrix $\gamma=(\gamma_{k,l})_{k,l}$, we define $\overline{\rF}(\gamma)=(\rF\gamma_{k,l})_{k,l}$. The functor $\overline{\rF}$ thus obtained is indeed a $p$-dg functor as
\begin{align*}
\overline{\rF}(\partial\beta)&=\overline{\rF}((\partial \beta_{k,l})_{k,l}+\beta\gamma - \gamma \alpha)\\
&=((\partial (\overline{\rF}\beta_{k,l}))_{k,l}+\overline{\rF}(\beta)\overline{\rF}(\gamma) - \overline{\rF}(\gamma)\overline{\rF}( \alpha)=\partial(\overline{\rF}\beta).
\end{align*} 

Clearly, by construction, $[\overline{\rF}]$ is the canonical functor $[\overline{\C}]\to [\overline{\D}]$ given by sending $\bigoplus X_i$ to $\bigoplus \rF X_i$, and morphisms $(\gamma_{k,l})_{k,l}$ to $(\rF\gamma_{k,l})_{k,l}$. We claim that $\overline{\rF}$ is the unique $p$-dg functor extending $\rF$ with this property. In fact, consider an object $X=(\bigoplus_{m=1}^s X_m, \alpha)$ of $\overline{\C}$. Assume $\tilde{\rF}$ satisfies the stated property, and denote the image $\tilde{\rF}(X)=(\bigoplus_{m=1}^s, \tilde{\alpha})$. Consider the projection $p_k\colon X\to X_k$. We compute
\begin{align*}
\partial p_k&=\partial(0,\ldots, 0,\id_{X_k},0,\ldots, 0)+(0,\ldots , 0,\id_{X_k},0,\ldots, 0)\alpha \\
&=(0,\ldots, 0,\alpha_{k,k+1},\alpha_{k,k+2},\ldots, \alpha_{k,s}).
\end{align*}
If $\tilde{\rF}$ is a $p$-dg functor, it follows that
\begin{align*}
\tilde{\rF}(\partial p_k)&=(0,\ldots,0,\rF \alpha_{k,k+1},\rF \alpha_{k,k+2},\ldots, \rF \alpha_{k,s})\\
&=(0,\ldots, 0, \tilde{\alpha}_{k,k+1},\tilde{\alpha}_{k,k+2},\ldots, \tilde{\alpha}_{k,s})=\partial (\tilde{\rF}p_k),
\end{align*}
using that $\tilde{\rF}$ restricts to the functor of componentwise application of $[\rF]$ under $[-]$. Hence $(\tilde{\alpha}_{k,l})_{k,l}=(\rF\alpha_{k,l})_{k,l}$, and thus $\tilde{\rF}(X)=\overline{\rF}(X)$. The uniqueness statement implies that $\overline{\rG}\circ \overline{\rF}=\overline{\rG\rF}$, as both these functors restrict to the functor of componentwise application of $[\rG\rF]=[\rG]\circ[\rF]$ under $[-]$.

To prove part (ii), consider a $\Bbbk$-linear natural transformation $\lambda \colon \rF\to \rG$.
We define $\lambda_X$, for $X\in \overline{\C}$ as above, to be the matrix with entries $\lambda_{G_m}$ on the diagonal for $m=1,\ldots, s$. Let $\gamma:(X, \alpha)\to (Y,\beta)$ be a morphism in $\overline{\C}$. 
Then the naturality condition $\lambda_Y \circ \rF\gamma =\rG\gamma \circ \lambda_X$ follows. Further, as $\lambda$ is a natural transformation, we have that $\lambda_X\circ\rF\alpha-\rG\alpha\circ\lambda_X=0$. Hence $\del (\overline{\lambda})=\overline{(\del \lambda)}$. The assignment thus defined is clearly functorial with respect to horizontal composition.
The requirement that  $[\overline{\lambda}]$ corresponds to $\lambda$ under the equivalence between $[\overline{\C}]$ and $[\C]$ forces $\lambda_X$ to be given by a diagonal matrix with entries $\lambda_{G_m}$ as described.

Furthermore,  the equality $\lambda_X \circ \rF\gamma =\rG\gamma \circ \lambda_X$, where $\gamma$ runs over all endomorphisms of $X$ given by  matrices with precisely one identity on the diagonal (say in position $k$) and zeros elsewhere, forces $(\lambda_X)_{k,k}$ to be $\lambda_{G_k}$ for all $k$, while all off-diagonal entries of $\lambda_X$ have to be zero. Thus all possible natural transformations $\lambda\colon \overline{\rF}\to \overline{\rG}$ take the shape of a diagonal matrix when evaluated on each object, and in particular, since $[\lambda]$ is a natural transformation, each diagonal entry has to be the evaluation of a natural transformation on the corresponding indexing object, and hence $\lambda$ is of the form $\overline{\mu}$ for the restriction $\mu$ of $\lambda$ along $\iota_\C$. Together with the observation that for a natural transformation $\lambda \in \Fun(\C,\D)$, the restriction of $\overline{\lambda}$ along $\iota$ recovers $\lambda$, this proves the $p$-dg isomorphism  $\Hom_{\Fun(\C,\D)}(\rF,\rG) \cong \Hom_{\Fun(\overline{\C},\overline{\D})}(\overline{\rF},\overline{\rG}) $.

Finally, to prove part (iii), consider the canonical $p$-dg functor $\C\to \overline{\C}$. Using part (i), we obtain a $p$-dg functor $\overline{\C}\to \overline{\overline{\C}}$. This $p$-dg functor is clearly fully faithful, and it is $p$-dg dense. Indeed, any object in $\overline{\overline{\C}}$ corresponds to an upper triangular matrix with entries morphisms of $\overline{\C}$, meaning it can be viewed as a (larger) upper triangular matrix with entries morphisms of $\C$, but such objects are in the image of $\overline{\C}$. To see this, consider an object $Y=(\bigoplus_{m=1}^s Y_m, \beta)$ in $\overline{\overline{\C}}$. Here, $X_m=(\bigoplus_{k=1}^{t_m}X_k^m, \alpha^m)$ are objects of $\overline{\C}$, and each matrix entry $\beta_{m,m'}$ is a $s_m\times s_{m'}$-matrix of morphisms $\beta_{m,m'}^{k,k'}\colon X^{m'}_{k'}\to X^{m}_{k}$. As $Y$ is an object of $\overline{\overline{\C}}$, we have that $(\del \tI_s+\beta)^p$ acts by zero. We compute
\begin{align*}
(\del \tI_s+\beta)&=\left( \delta_{m,m'} \del_{Y_m}+\beta_{m,m'}\right)_{m,m'}\\
&=\left( \delta_{m,m'} (\delta_{k,k'}\del_{X^m_k}+\alpha^m_{k,k'})_{k,k'} + (\beta_{m,m'}^{k,k'})_{k,k'}  \right)_{m,m'}\\
&=\left( \left(\delta_{m,m'}\delta_{k,k'}\del_{X^m_k}+\alpha^m_{k,k'} + \beta_{m,m'}^{k,k'}\right)_{k,k'}  \right)_{m,m'},
\end{align*} 
to see that the $p$-th power of the latter matrix also acts by zero.
Hence we can define the object $Y'$ of $\overline{\C}$ given by the pair $$\left(\bigoplus_{m=1}^s {\bigoplus_{k=1}^{t_m} X^m_k}, \left( \delta_{m,m'} \alpha^m_{k,k'} + \beta_{m,m'}^{k,k'}\right)_{(m,k),(m',k')}\right).$$
The $p$-dg isomorphisms $(Y_i, 0)\to (\bigoplus_{k=1}^{t_m}{X_k^m}, \alpha^m)$ assemble to a $p$-dg isomorphism $Y\to Y'$ completing the proof of $p$-dg density.
Thus, part (iii) follows using Lemma \ref{pdgequivlemma}.
\end{proof}

In fact, the following lemma shows that a $p$-dg functor of locally finite $p$-dg categories which are of the form $\overline{(-)}$ can be recovered from its restriction to $\C$ in the source category.

\begin{lemma}\label{restrictionlemma}
Let $\C$, $\D$ be categories in $\pdgcat$, and $F\colon \overline{\C}\to \overline{\D}$. Then there is a $p$-dg isomorphism of $p$-dg functors $F\cong (\iota^{-1}_{\overline{\D}})\circ \overline{\left.F\right|_{\C}}$.
\end{lemma}

\begin{proof}
Let $X_1,\dots, X_n\in \C$ and consider $F(X_i, 0)\in \overline{\D}$ for each $i=1,\ldots, n$. By additivity of $F$ and the additive structure of $\overline{\C}$, we have a $p$-dg isomorphism 
$$\bigoplus_{i=1}^n \left.F\right|_{\C}X_i=\bigoplus_{i=1}^n F(X_i, 0) \overset{\sim}{\to}   F\left(\bigoplus_{i=1}^nX_i, 0\right)\in \overline{\D},$$
where the first  zeros are $1\times 1$-matrices and the last zero is an $n\times n$-matrix. By the universal property of biproducts, this isomorphism is natural with respect to morphisms in $\overline{\C}$, and hence
$$\Hom_{\overline{\D}}\left(F\left(\bigoplus_{i=1}^nX_i, 0\right),   F\left(\bigoplus_{i=1}^mY_i, 0\right)   \right) \cong \Hom_{\overline{\D}}\left(\bigoplus_{i=1}^n \left.F\right|_{\C}X_i,\bigoplus_{i=1}^m \left.F\right|_{\C}Y_i \right)$$
for $Y_i\in \C$.

We wish to compare the two functors $\iota_{\overline{\D}}\circ F$ and  $\overline{\left.F\right|_{\C}}$.
Since, for more general objects  $X:=(\bigoplus_{i=1}^nX_i, \alpha), Y:=(\bigoplus_{i=1}^mY_i, \beta)\in \overline{\C}$, we have 
$$\Hom_{[\overline{\C}]} \left(X,Y \right)\cong \Hom_{[\overline{\C}]} \left(\left(\bigoplus_{i=1}^nX_i, 0\right),\left(\bigoplus_{i=1}^m Y_i, 0\right) \right)$$
we obtain an isomorphism 
\begin{equation*}\begin{split}
\Hom_{[\overline{\D}]} \left(F(X),F(Y) \right)&\cong \Hom_{[\overline{\D}]} \left(F\left(\bigoplus_{i=1}^nX_i, 0\right),F\left(\bigoplus_{i=1}^mY_i, 0\right) \right) \\
 &\cong\Hom_{[\overline{\D}]}\left(\bigoplus_{i=1}^n \left.F\right|_{\C}X_i,\bigoplus_{i=1}^m \left.F\right|_{\C}Y_i \right).\end{split}\end{equation*}

Recalling the canonical $p$-dg inclusion $\iota_{\overline{\D}}$, the functor $\iota_{\overline{\D}}\circ F$ descends to the same functor $[\overline{\C}]\to [\overline{\overline{\D}}]$ as $\overline{\left.F\right|_{\C}}$. By the uniqueness statement in \ref{extensionlemma}\eqref{extensionlemma1} we see that $\iota_{\overline{\D}}^{-1}\circ \overline{\left.F\right|_{\C}}$ and $F$ are $p$-dg isomorphic, choosing $\iota_{\overline{\D}}^{-1}$ a quasi-inverse to $\iota_{\overline{\D}}$ by \ref{extensionlemma}\eqref{extensionlemma3}. \end{proof}

\subsection{Subquotient completion}\label{fincat}

Let $\C \in \pdgcat$. We need  to construct a  closure $\C^{\bullet}$ of $\C$ under subquotients. 
This uses considerations from \cite[Section 4.2]{EQ2}. From there, we adapt the following definition:
 
\begin{definition}
Let $\C$ be a $p$-dg category in $\pdgcat$. We say that an idempotent $e\colon X\to X$ in $\C$ is a {\bf subquotient idempotent} if there exists another idempotent $w\colon X\to X$ such that
\begin{enumerate}
\item[(i)] $we=ew=e$,
\item[(ii)] $w\del(w)=0$, and
\item[(iii)] $(w-e)\del(w-e)=0$.
\end{enumerate}
\end{definition}

Note that the definition is symmetric in the sense that an equivalent set of axioms is is to require the existence of an idempotent $v$ on $X$ such that $ve=ev=e$, $\del(v)v=0$, and $\del(v-e)(v-e)=0$.  Note further that for any idempotent $e$ we have $\del(e)=\del(e)e+e\del(e)$, which implies that $(1-e)\partial(e)=\partial(e)e$, and hence $e\partial(e)e=0$. The following lemma is adapted from \cite[Lemma~4.2]{EQ2}.

\begin{lemma}
Given two subquotient idempotents $e\colon X\to X$, and $f\colon Y\to Y$, the subspace
$f\Hom_\C(X,Y)e$ obtains the $p$-differential given by
\begin{align}\label{restricteddiff}
\del^{\bullet}(fge):=f\del(fge)e, &&\forall g\colon X\to Y.
\end{align}
\end{lemma}
\begin{proof}
We give a direct proof for completeness. First show that if $e$ is a subquotient idempotent, then $\del(e)\del(e)e=e\del(e)\del(e)=0$.
Indeed, one computes
\begin{align*}
\del(e)\del(e)e&=\del(ew)\del(we)e\\
&=\del(e)w\del(w)e+\del(e)w\del(e)e+e\del(w)\del(w)e+e\del(w)w\del(e)e\\
&=\del(e)w\del(e)e+ew\del(w)\del(w)e+ew\del(w)w\del(e)e\\
&=\del(e)w\del(e)e=\del(e)e\del(e)e=0,
\end{align*}
where we use that $(w-e)\del(w-e)=0$ implies $w\del(e)=e\del(e)$. The equality $e\del(e)\del(e)=0$ follows similarly.

Now we compute, for $g\in f\Hom_\C(X,Y)e$, that
\begin{align*}
(\del^{\bullet})^2(fge)=\del^{\bullet}(f\del(fge)e)&=f\del(f\del(fge)e)e\\
&=f\del(f)\del(fge)e+f\del^2(fge)e+f\del(fge)\del(e)e\\
&=f\del(f)\del(f)ge+f\del^2(fge)e+fg\del(e)\del(e)e\\
&=f\del^2(fge)e.
\end{align*}
Repeated application shows that $(\del^{\bullet})^p(fge)=0$.
\end{proof}


We now define the {\bf subquotient idempotent completion}  $\C^{\bullet}$ of $\C$ in $\pdgcat$ as having objects $X_e$ for any subquotient idempotent $e\in \End_\C(X)$, together with morphisms $i_e\colon X_e\to X$, $p_e\colon X\to X_e$ such that, if $f\in \End_\C(Y)$ is another subquotient idempotent 
\begin{align}\label{completionhoms}
&\adj{i_f\circ (-)\circ p_e}{\Hom_{\C^{\bullet}}(X_e,Y_f)}{\left(f\Hom_{\C}(X,Y)e, \del^{\bullet}\right)}{p_f\circ(-)\circ i_e},
\end{align}
are mutually inverse $p$-dg isomorphisms. 
This, in particular, implies that
\begin{align}
i_ep_e&=e, &p_ei_e&=\id_{X_e},\label{completioneq1}\\
\del(i_e)&=\del(e)i_e,& \del(p_e)&=p_e\del(e).\label{completioneq2}
\end{align}
In fact, the conditions \eqref{completioneq1}, \eqref{completioneq2} ensure that the morphisms in \eqref{completionhoms} are mutually inverse $p$-dg morphisms, and that the differential on the category $\C^{\bullet}$ satisfies the Leibniz rule.

\begin{proposition}\label{idempotentcompletion}
The category $\C^{\bullet}$ defined above is a $p$-dg category in $\pdgcat$ and there exists a $p$-dg embedding $\C\hookrightarrow \C^{\bullet}$. Furthermore, $(\C^{\bullet})^{\bullet}$ is $p$-dg equivalent to $\C^{\bullet}$.
\end{proposition}
\begin{proof}
Let $f\colon X_{e_1}\to X_{e_2}$, $g\colon  X_{e_2}\to X_{e_3}$ be morphisms in $\C^{\bullet}$, where $f=p_{e_2}e_2he_1i_{e_1}$ and $g=p_{e_3}e_3ke_2i_{e_2}$. Then
\begin{align*}
gf=p_{e_3}e_3ke_2i_{e_2}p_{e_2}e_2he_1i_{e_1}=p_{e_3}e_3ke_2he_1i_{e_1},
\end{align*}
and 
\begin{align*}
\del(gf)&=p_{e_3}\del(ke_2h)i_{e_1},\\
&=p_{e_3}\del(k)e_2hi_{e_1}+p_{e_3}ke_2\del(h)i_{e_1}=\del(g)f+g\del(f),
\end{align*}
using that $k\del(e_2)h=ke_2\del(e_2)e_2h=0$ as $e_2\del(e_2)e_2=0$.

It follows that $\C^{\bullet}$ is a $p$-dg category. It is clear that mapping $X\mapsto X_1$ exhibits $\C$ as a $p$-dg full subcategory of $\C^{\bullet}$. Starting with $\C^{\bullet}$, the embedding $\C^{\bullet}\hookrightarrow (\C^{\bullet})^{\bullet}$ is $p$-dg dense. Indeed, any subquotient idempotent $e$ in $\C^{\bullet}$ is a subquotient idempotent $e=e_1ee_1$ for a subquotient idempotent $e_1$ in $\C$, which forces that $e$ is in fact a subquotient idempotent in $\C$, using the same $w$ such that $ew=w=ew$, and hence $X_e$ in $(\C^{\bullet})^{\bullet}$ is in the $p$-dg essential image of $\C^{\bullet}$.

Note that the subquotient idempotent completion $\C^{\bullet}$ again has finite biproducts. The biproduct $X_e\oplus Y_f$ in $\C^{\bullet}$ can be obtained using the splitting object $(X\oplus Y)_{e\oplus f}$ of the subquotient idempotent $e\oplus f$ on $X\oplus Y$ in $\C$. Moreover, $\C^{\bullet}$ is locally finite as the morphism spaces are subspaces of the finite-dimensional morphism spaces in $\C$.
\end{proof}

We say a $p$-dg category is {\bf subquotient idempotent complete} if the canonical embedding $\C\hookrightarrow \C^{\bullet}$ is part of a $p$-dg equivalence.

\begin{lemma}
A $p$-dg category $\C$ in $\pdgcat$ is subquotient idempotent complete if and only if it is complete under all left and right submodule idempotents, that is, idempotents $e$ such that $\del(e)=\del(e)e$ (respectively, $\del(e)=e\del(e)$).
\end{lemma}
\begin{proof}
Assume $\C$ is complete under left and right submodule idempotents. Let $e\colon X\to X$ be a subquotient idempotent dominated by $w$. Then $w$ and $w-e$ are left submodule idempotents. Consider the splitting object $X_w$.  
Further, $e$ is a right submodule idempotent on $X_w$. Indeed, using that $w-e$ is a left submodule idempotent, we find
$$\overline{\del}(e)=w\del(e)w=w\del(e)=e\del(e)=ew\del(e)w=e\overline{\del}(e),$$
where the second equality uses  
$$\del(e)=\del(ew)=\del(e)w+e\del(w)=\del(e)w.$$
Hence, under the assumptions $e$ also splits, giving an object $X_e$ with the properties as in Equation (\ref{completionhoms}).
\end{proof}

%

\begin{lemma}\label{idemcompfin}
If $\C$ is locally finite
, then $\C^{\bullet}$ is again locally finite.
\end{lemma}

\proof
Notice that morphism spaces between two $p$-dg indecomposable objects in 
$\C^{\bullet}$ are subspaces of morphism spaces between two $p$-dg indecomposable objects in $\C$, and are hence finite-dimensional; thus $\C^{\bullet}$ is locally finite. 
\endproof


We can also define the {\bf $p$-dg idempotent completion $\C^\circ$ of $\C$} by adding objects $X_e$ in the same way for $p$-dg idempotents $e$ only. This produces a $p$-dg idempotent complete $p$-dg category.


\begin{lemma}\label{idempotentcomposition}
Let $F\colon \C\to \D$ be a $p$-dg functor of locally finite $p$-dg categories. Then there exists an induced $p$-dg functor $F^{\bullet}\colon \C^{\bullet}\to \D^{\bullet}$ (respectively, $F^\circ\colon \C^\circ\to \D^\circ$). This assignment  is functorial with respect to composition, i.e. $(GF)^{\bullet}=G^{\bullet}F^{\bullet}$ (respectively, $(GF)^\circ=G^\circ F^\circ$).
\end{lemma}
\begin{proof}
If $e$ is a subquotient idempotent, then so is $F(e)$. Thus, we may define a functor $F^{\bullet}$ by mapping $X_e$ to $F(X)_{F(e)}$, and extending to morphisms accordingly. Compatibility with composition of functors follows directly.
\end{proof}

\subsection{Fantastic filtrations and strong finitarity}\label{idempfant}

Let $\C$ be a $p$-dg category in $\pdgcat$ and $\A$ a $p$-dg subcategory of $\C$. We say that $X\in \C$ has a {\bf fantastic filtration} by objects in $\A$ if there exists a filtration $0=F_0\subset F_1\subset \cdots \subset F_m = X$, whose successive subquotients are objects $X_1,\dots, X_m \in \A$, and maps 
$$\adj{u_i}{X}{X_i}{v_i}$$
such that 
$u_iv_j = \delta_{i,j}\id_{X_i}$, $\sum_{j=1}^m v_ju_j = \id_{X}$, $\del(u_i)v_i = 0$ and $\im\del(v_i)u_i \subset F_{i-1}$ (cf. \cite[Section 5.1]{EQ}).

Let $\C^{\bullet}$ denote the subquotient idempotent completion  of $\C$, and $\B$  the $p$-dg subcategory consisting of the additive and grading shift closure of $\Bbbk$-indecomposable objects in $\C^{\bullet}$. Notice that not necessarily all idempotents in $\C^{\bullet}$ split, cf. \cite[Remark 4.3]{EQ2}.

In order to obtain sufficiently well-behaved $p$-dg categories, we address the question under what circumstances $\overline{\C}$ and $\overline{\B}$ are $p$-dg equivalent.  Both of these categories naturally live inside $\overline{\C^{\bullet}}$.

\begin{lemma}\label{ABsubs}$~$

\begin{enumerate}[(i)]
\item\label{ABsubs1} As subcategories of $\overline{\C^{\bullet}}$, the $p$-dg category $\overline{\C}$ is $p$-dg essentially contained in  $\overline{\B}$ if and only if every $X\in \C$ has a fantastic filtration by objects in $\B$.
\item\label{ABsubs2}  As subcategories of $\overline{\C^{\bullet}}$, the $p$-dg category $\overline{\B}$ is $p$-dg essentially contained in  $\overline{\C}$ if and only if, for every $Y\in \B$, $P_Y^{\op}$ is semi-free over $\C$.
\end{enumerate}

\end{lemma}
\begin{proof}
Assume that $\overline{\C}$ is $p$-dg essentially contained in $\overline{\B}$ and let $X$ be any object of $\C$. By assumption, $X$ is $p$-dg isomorphic to an object $X'=\left( \bigoplus_{j=1}^m X_j,\alpha \right)$, where $X_j$ are objects in $\B$. We define a filtration by $F_i=\bigoplus_{j=1}^i X_i$ for all $i=0,\ldots,m$. We use the obvious projection maps $u_i\colon X'\to X_i$ and injection maps $v_i\colon X_i \to X'$ (composed with the $p$-dg isomorphism). Then clearly $u_iv_j = \delta_{i,j}\id_{X_i}$ and $\sum_{j=1}^m v_ju_j = \id_{X'}$. We further compute, using the differential in $\overline{\C}$,
\begin{align*}
\del(v_i)=&(\del(u_jv_i))_j+ \alpha v_i-v_i 0=\sum_{j=1}^{i-1}\alpha_{ji}v_i.
\end{align*}
This shows that the image of $\del(v_i)$ (and hence the image of $\del(v_i)u_i$) is contained in $F_{i-1}$. Further, $u_i\del(v_i)=0$ and hence $\del(u_i)v_i=0$ using the Leibniz rule. This shows that the $F_i$ give a fantastic filtration of $X'$ by objects in $\B$, which under $p$-dg isomorphism translates to such a filtration of $X$.

Conversely, assume that every $X$ has a fantastic filtration $0=F_0\subset F_1\subset \ldots \subset F_m=X$ where all subquotients $X_i$ are objects of $\B$. We prove that $X$ lies in the $p$-dg essential image of the subcategory $\overline{\B}$ by induction on the length $m$. The statement is clear for $m=1$. Assume that $F_{m-1}$ is $p$-dg isomorphic to the object $\left(\bigoplus_{j=1}^{m-1}X_j,\alpha \right)$ of $\overline{\B}$. The map $\del(v_m)$  factors through $F_{m-1}$ and we can hence consider the compositions $\alpha_{m,j}\colon u_j\del(v_m)\colon X_m\to X_j$ for $j=1,\ldots,m-1$. This constructs the object $\widetilde{X}:=\left(\bigoplus_{j=1}^{m}X_j,\alpha \right)$, which is $p$-dg isomorphic to $X$.

\eqref{ABsubs2} holds by definition, as $\B$ is additively generated by $\Bbbk$-indecomposables.
\end{proof}

\begin{corollary}\label{ABequiv}
Assume that every $X\in \C$ has a fantastic filtration by objects in $\B$ and for every $Y\in \B$, $P_Y^{\op}$ is semi-free over $\C$. In this case, the categories  $\overline{\C}$ and $\overline{\B}$ are $p$-dg equivalent.
\end{corollary}



\begin{definition}\label{strongfindef}
We call a $p$-dg category $\C$ in $\pdgcat$  {\bf strongly finitary}, if, letting $\B$ denote the full $p$-dg subcategory on the additive and grading shift closure of $\Bbbk$-indecomposable objects in $\C$, $\B$ has only finitely many $\Bbbk$-indecomposable objects up to $p$-dg isomorphism and grading shift, and the natural embedding $\overline{\B}\hookrightarrow \overline{\C}$ is $p$-dg essentially surjective. That is, every object has a finite fantastic filtration by  $\Bbbk$-indecomposable objects in $\C$.
\end{definition}

It follows that if $\C$ is strongly finitary $p$-dg category, then $\overline{\C}$ is also strongly finitary, since every object in $\overline{\C}$ has a fantastic filtration by $\Bbbk$-indecomposables objects which are of the form  $(X,(0))$ where $X$ is $\Bbbk$-indecomposable in $\C$.

\begin{lemma}
If $\C$ is a strongly finitary $p$-dg category, then $[\C]$ is an object in the $2$-category of finitary $\Bbbk$-linear categories $\addcat$ as defined in \cite{MM2}.
\end{lemma}
\begin{proof}
Morphism spaces are finite-dimensional by assumption. Furthermore, every object has a finite fantastic filtration by $\Bbbk$-indecomposable objects. This implies that the objects in $[\C]$ decompose as finite direct sums of indecomposable objects. The number of indecomposable objects up to isomorphism is bounded above by the number of $\Bbbk$-indecomposable objects in $\C$ up to $p$-dg isomorphism and grading shift, which is finite by assumption.
\end{proof}

\begin{definition}\label{barclosure}
Let $\C$ be a strongly finitary $p$-dg category and $\S$ be the full $p$-dg subcategory on a subset of objects of $\overline{\C}$. We define the {\bf bar closure} $\widehat{\S}$ of $\S$ to be $\overline{\S^\dagger}\subset \overline{\C}$, where $\S^\dagger$ denotes the closure under $\Bbbk$-isomorphism and grading shift of all $\Bbbk$-indecomposable objects appearing in any fantastic filtration of objects in $\S$.
\end{definition}

Notice that $[\widehat{\S}]$ is equivalent to $\add([\S])$. 
If $\C$ is strongly finitary, and $\S$ a full subcategory on a set of objects of $\C$, then the bar closure $\widehat{\S}$ is again strongly finitary.

For the purpose of giving targets for $p$-dg $2$-representations in Chapter \ref{pdg2repchapter}, we define the $p$-dg $2$-subcategory $\cofcat$ (respectively 
$\cofcat^{\mathrm{sf}}$) of $\pdgcat$ as the $2$-category whose
\begin{itemize}
\item objects are $p$-dg categories $p$-dg equivalent to $\overline{\C}$ for a locally finite (respectively 
 strongly finitary) $p$-dg category $\C$;
\item $1$-morphisms are $p$-dg functors between such categories;
\item $2$-morphisms are all morphisms of such $p$-dg functors.
\end{itemize}

For future use, we record the following observation.

\begin{lemma}\label{sameideal} 
Let $\C$ be a strongly finitary $p$-dg category, $X$ a $\Bbbk$-indecomposable object in $\C$, and $\I$ the two-sided $p$-dg ideal generated by $\id_X$. Then $[\I]$ is equal to the two-sided ideal generated by $\id_X$ in $[\C]$.
\end{lemma}

\proof
By definition, the ideal $\I_{[\C]}$ generated by $\id_X$ in $[\C]$ is obtained by starting with $\id_X$ and closing under compositions and the $\Bbbk$-action. Since hom-spaces in $\C$ and $[\C]$ are the same, $\I_{[\C]}$ is contained in $[\I]$. We need to check that additionally closing under $\del$ does not enlarge the ideal further. Since  $\C$ is strongly finitary, $\I$ consists of objects $Y$ such that $X$ is a $\Bbbk$-indecomposable factor in a fantastic filtriation of $Y$.
Such a factor of $Y$ is determined by a subquotient idempotent $e$,  and we have that $\del(e)$ is in $\I_{[\C]}$. Indeed $\del(e) = \del(e^2) = \del(e)e+e\del(e)$, 
 which is in $\I_{[\C]}$ as desired.\endproof

\subsection{Associating \texorpdfstring{$p$}{p}-dg algebras to \texorpdfstring{$p$}{p}-dg categories}\label{algcat}

\begin{example}\label{algebraex} (cf.\ \cite[Remark 2.26]{EQ})
Assume $\C$ is a locally finite $p$-dg category with a finite set of objects $\mathtt{X}:=\{X_1,\dots, X_r\}$ such that every object in $\C$ has a fantastic filtration by shifts of objects in $\mathtt{X}$. Then taking the endomorphism algebra of $X:=\bigoplus_{i=1}^r X_i$ produces a $p$-dg algebra $A_\C:=\End_\C(X)^\op$ which is finite-dimensional. 

Conversely, let $A$ be a  finite-dimensional $p$-dg algebra. We want to view $A$ as a locally finite $p$-dg category, denoted by $\A$. We choose a decomposition $1=e_{1}+\cdots+ e_{r}$ of the identity into $p$-dg idempotents. We then define $\A$ to be the closure under finite biproducts and grading shifts of objects $Y_i$ corresponding to $e_{1},\ldots, e_{r}$,  and morphism spaces $\A(Y_i,Y_j)=e_i Ae_j$, which are $H$-submodules of $A$. 
\end{example}

%
%

Note that the $p$-dg category $\A$ is typically not subquotient idempotent complete.  If the decomposition $1=e_{1}+\cdots+ e_{r}$ is one of indecomposable $p$-dg idempotents, then 
it is sufficient but not necessary that all idempotents in $A$ are $p$-dg idempotents (i.e. annihilated by $\partial$). Note that all central idempotents in a $p$-dg algebra are necessarily $p$-dg idempotents. Indeed, for any idempotent $e$, $e\partial(e)e=0$. This, if $e$ is central, implies $\partial(e)=0$ in any characteristic.


Associating $p$-dg categories to $p$-dg algebras and vice versa, we have the following lemma.

\begin{lemma}\label{cofoverline}
Let $\C$ be a locally finite $p$-dg category with a finite set of objects $\mathtt{X}:=\{X_1,\dots, X_r\}$ such that every object in $\C$ has a fantastic filtration by shifts of objects in $\mathtt{X}$. Then the category $\overline{\C}$ is $p$-dg equivalent to the category $\overline{\A_{A_\C}}$ where $A_\C$ is defined in Example \ref{algebraex} and we choose the decomposition $1= \id_{X_1}+\cdots \id_{X_r}$.
\end{lemma}

\proof By our assumption on $\C$, setting $\X$ to be the closure under grading shifts and finite biproducts of the full $p$-dg subcategory of $\C$ on objects in $\mathtt{X}$, we have a $p$-dg  equivalence between $\overline{\C}$ and $\overline{\X}$, so if suffices to prove that $\overline{\X}$ and $\overline{\A_{A_\C}}$ are $p$-dg equivalent.

In fact, we will prove that $\X$ and $\A_{A_\C}$ are $p$-dg equivalent.
We define a $p$-dg functor $\X\to \A_{A_\C}$ sending $X_i$ to $Y_i$,  and extend the definition to morphisms using the $p$-dg isomorphisms
$$\X(X_i,X_j)=\C(X_i,X_j)\cong\id_{X_j}\circ \End_\C(X)\circ\id_{X_i}= \A_{A_\C}(Y_i,Y_j).$$
This assignment extends to finite biproducts, and thus provides a fully faithful $p$-dg functor which is $p$-dg dense by construction.
\endproof

In the following,  denote by $A\plmod$ the $p$-dg category of finite-dimensional left $A$-modules in $H\lmod$ enriched via internal homs \cite[2.24]{EQ}. 

\begin{lemma}\label{moduletranslation}
Let $\C$ be a $p$-dg category satisfying the assumptions of Lemma \ref{algebraex}. Then there is a $p$-dg equivalence of the $p$-dg categories $\C^{\op}\plmod$ and $A_\C\plmod$.
\end{lemma}
\begin{proof}
Denote $A=A_\C$. With the same notation as in Lemma \ref{cofoverline},
we first construct a $p$-dg equivalence
 $\Phi\colon A\plmod\to \X^{\op}\plmod$. 
Given an object $V$ in $A\plmod$, consider the assignment $M_{V}(X_i)=e_iV$, where $e_i=\id_{X_i}$ is the $p$-dg idempotent in $A_\C$ corresponding to $X_i$. A morphism $f\colon X_i\to X_j$ corresponds to an element $f=e_ife_j$ in $A_\C$. We assign to it the morphism $M_{V}(f)\colon e_j V\to e_i V$, $e_j(V)\mapsto f(e_j(V))$. This way, $\Phi(V)=M_V$ defines an object in $\X^{\op}\plmod$. Similarly, a morphism $\phi\colon V\to W$ in $A\plmod$ gives a morphism of $p$-dg functors $\Phi(\phi)\colon M_V\to M_W$. Hence, we obtain a $p$-dg functor $A\plmod\to \X^{\op}\plmod$. An inverse functor can be constructed as follows. Given a $p$-dg functor $M\colon \X^{\op}\to \Bbbk\plmod$, define $V=\bigoplus_i M(X_i)$. As there are finitely $X_i$, this gives a finite-dimensional $A$-module. It is readily verified that the provided functors form an equivalence of categories.

Next, we show that $\X^{\op}\plmod$ and $\C^{\op}\plmod$ are $p$-dg equivalent. There is an evident $p$-dg functor $\C^{\op}\plmod\to  \X^{\op}\plmod$ by restriction. This $p$-dg functor is $p$-dg dense as, given a $p$-dg module $M$ over $\X^{\op}$, there exists a $p$-dg functor $\overline{M}\colon \overline{\X}^{\op}\to \overline{\Bbbk\plmod}\simeq \Bbbk\plmod$  by Lemma \ref{extensionlemma}\eqref{extensionlemma1}. This gives a $p$-dg module over $\overline{\X}^{\op}$ such that its restriction to $\C^{\op}$ further restricts to a $p$-dg module over $\X^{\op}$ which is $p$-dg equivalent to $M$. Fullness of the restriction functor follows using Lemma  \ref{extensionlemma}\eqref{extensionlemma2}, and faithfulness follows from Lemma \ref{restrictionlemma}.
\end{proof}

\begin{remark}\label{endostrongfin}
For a strongly finitary $p$-dg category $\C$ we can consider the endomorphism algebra $A$
of a, up to grading shift, complete set of $p$-dg isomorphism classes of $\Bbbk$-indecomposable objects in $\C$. Then $\overline{\C}$ is $p$-dg equivalent to $\overline{\A}$ (cf.\ Example \ref{algebraex}, Lemma \ref{cofoverline}). 
Consider the composition
\[
\phi\colon \overline{\C}\stackrel{\sim}{\rightarrow}\C^{\op}\cof\hookrightarrow \C^{\op}\plmod \stackrel{\sim}{\rightarrow} A\plmod,
\]
where the first equivalence follows from Lemma \ref{matcatequiv} and the last equivalence follows from Lemma \ref{moduletranslation}. The functor $\phi$ is fully faithful, and induces 
an equivalence of $\Bbbk$-linear categories $[\overline{\C}]\simeq  A\proj$. Here, the later category is that of all finitely generated (in our setup, finite-dimensional) projective $A$-modules. 

We further remark that the underlying $\Bbbk$-algebra $A$ is basic if and only if any two $\Bbbk$-indecomposable objects which are $\Bbbk$-isomorphic are also $p$-dg isomorphic.
\end{remark}

\subsection{Tensor products}\label{tensorproducts}
For a symmetric monoidal category like $H\lmod$, enriched finitely cocomplete categories again have a tensor product structure $\boxtimes$ \cite[6.5]{K}. The finite coproducts will again be biproducts in our setup as the $p$-dg categories are, in particular, $\Bbbk$-linear.
Hence the tensor product $\C\boxtimes \D$ for two $p$-dg categories $\C$ and $\D$ has objects generated under finite biproducts by pairs in $\operatorname{Ob}(\C)\times \operatorname{Ob}(\D)$ for which we denote by the corresponding object by $X\otimes Y$. The morphism spaces are matrices of morphisms between objects of the form $X\otimes Y$, $X'\otimes Y'$, on which we use the symmetric monoidal structure of $H\lmod$ to define $\del$. That is,
\begin{align*}
\Hom_{\C\boxtimes\D}(X\otimes Y, X'\otimes Y')&=\Hom_{\C}(X,X')\otimes \Hom_{\D}(Y,Y'),\\
\del(f\otimes g)&=\del(f)\otimes g+f\otimes \del(g).
\end{align*}
One checks that this differential is compatible with compositions and matrix multiplication.

Given $p$-dg functors $F\colon \C\to \C'$ and $G\colon \D\to \D'$, we can induce a $p$-dg functor $F\boxtimes G\colon \C\boxtimes \D\to \C'\boxtimes \D'$. It extends uniquely from requiring $(F\boxtimes G)(X\otimes Y)=F(X)\otimes G(Y)$ to a $p$-dg functor. Moreover, we see that $\boxtimes \colon \pdgcat \times \pdgcat \to \pdgcat$ is a $p$-dg $2$-functor.

\begin{lemma}\label{tensorlemma}Let $\C$ and $\D$ be $p$-dg  categories in $\pdgcat$.
\begin{enumerate}[(i)]
\item\label{tensorlemma1} If $\C$ and $\D$ are locally finite, then so is $\C\boxtimes \D$. 
\item\label{tensorlemma3} The tensor product of two strongly finitary $p$-dg categories is again strongly finitary.
\item\label{tensorlemma4} There exists a fully faithful $p$-dg functor $\overline{\C}\boxtimes \overline{\D}\to\overline{\C\boxtimes \D}$.
\end{enumerate}
\end{lemma}
\begin{proof}
Part (\ref{tensorlemma1}) is clear as the $p$-dg indecomposable objects in the tensor product category are of the form $X\otimes Y$ for $p$-dg indecomposable objects $X$, $Y$, and the morphism spaces are tensor products. Hence the finiteness conditions are inherited by the tensor product categories. To verify part (\ref{tensorlemma3}), note that the tensor products of two objects having a fantastic filtration by $\Bbbk$-indecomposables also has such a filtration, using the tensor products of the $\Bbbk$-indecomposable factors.

To prove part (\ref{tensorlemma4}), we start with the $p$-dg functor $\C\boxtimes \D\to \overline{\C}\boxtimes \overline{\D}$ obtained as the tensor product of the canonical embedding functors into the corresponding  categories $\overline{(-)}$. Using \eqref{extensionlemma1} we obtain a functor $\tau \colon \overline{\C\boxtimes \D}\to \overline{\overline{\C}\boxtimes \overline{\D}}$. We claim that this functor is a $p$-dg equivalence. The functor is clearly fully faithful, and we claim that it is $p$-dg dense. Indeed, an object $X\otimes Y$ in $\overline{\overline{\C}\boxtimes \overline{\D}}$ has the form
\begin{align*}
\left(\bigoplus_{m=1}^s{X_m\otimes Y_m}, \alpha\right), \quad\text{where } X_m=\left(\bigoplus_{k=1}^{t_m}V_{k}, \gamma^m\right), Y_m=\left(\bigoplus_{k=1}^{u_m} W_k, \delta^m\right).
\end{align*}
However, each factor $X_m\otimes Y_m$ of $X\otimes Y$ is $p$-dg isomorphic to the object $\left(\oplus_{k,l}V_k\otimes W_l,\gamma^m\otimes \tI_{u_m}+\tI_{t_m}\otimes \delta^m \right)$ in $\tau(\overline{\C\boxtimes \D})$. This can be seen using distributivity of $\otimes$ with biproducts, and the Leibniz rule for the differentials. Now, arguing similarly as in the proof of Lemma \ref{extensionlemma}, there is a $p$-dg isomorphism of the extension of the objects in  $\tau(\overline{\C\boxtimes \D})$ which are $p$-dg isomorphic to the factors $X_m\otimes Y_m$ to give that $\tau$ is $p$-dg dense, and hence gives a $p$-dg equivalence.
To complete the proof of (\ref{tensorlemma4}), we note that the canonical embedding $\iota\colon \overline{\C}\boxtimes \overline{\D}\to \overline{\overline{\C}\boxtimes \overline{\D}}\simeq \overline{\C\boxtimes \D}$ is clearly fully faithful.
\end{proof}

Let $\C$ be a locally finite $p$-dg category. In the following, we construct an action of $\overline{\C \boxtimes \C^\op}$ on $\overline{\C}$, i.e. a bimodule action
\begin{align}\label{bimoduletensor}
\boxtimes_\C\colon \overline{\C\boxtimes \C^\op} \boxtimes \overline{\C} \longrightarrow\overline{\C}.
\end{align}
Recall that $\Bbbk\plmod$ is the $p$-dg category of finite-dimensional graded $H$-modules with internal homs. All objects in this category are semi-free and the only $\Bbbk$-indecomposable object (up to grading shift) is the trivial $H$-module $\Bbbk$. By Lemma \ref{matcatequiv}, $\Bbbk\plmod$ is hence  $p$-dg equivalent to $\overline{k}$, where $k$ denotes the $p$-dg category with one object (up to grading shift), which has the trivial $H$-module $\Bbbk$ as endomorphism ring. Observe that there is a $p$-dg functor
\begin{align}\label{Hmoduletensor}
\boxtimes\colon \overline{\C} \boxtimes \overline{k} \longrightarrow \overline{\C},
\end{align}
which is obtained by extending the functor $\C\boxtimes k\to \C$, given by mapping $X\otimes \Bbbk\mapsto X$, to $\overline{(-)}$ and pre-composing with the morphism from Lemma  \ref{tensorlemma}\eqref{tensorlemma4}. Under the equivalence of Lemma \ref{matcatequiv} this corresponds to pointwise tensoring with  graded $H$-modules over $\Bbbk$.

For a concrete description, we fix the notation $\dim_t W=\sum_{i}\dim W_it^i$ for a $\mathbb{Z}$-graded $H$-module $W=\bigoplus_{i}W_i$, and write
\begin{align}
\begin{split} X^{n(t)}
&=...\oplus X^{\oplus n_{i+1}}\shift{-i-1}\oplus X^{\oplus n_i}\shift{-i}\oplus X^{\oplus n_{i-1}}\shift{-i+1}\oplus\ldots,\end{split}
\end{align}
for $n(t)=\sum_i n_i t^i \in \mathbb{N}[t]$. Note the convention that the summands are arranged with decreasing value of $i$ (from left to right).
If $V_i$ denotes the $i+1$-dimensional indecomposable $H$-module generated in degree zero, for $i=0,1, \ldots, p-1$, then
\begin{equation}\label{Hmoduletensor2}
 \left(\bigoplus_{m=1}^sX_m, \alpha\right)\otimes V_i\cong \left(\bigoplus_{m=1}^sX_m^{\dim_t V_i}, \alpha\otimes \tI_{i+1}+\id_{\oplus_m X_m}\otimes J_i\right),
\end{equation}
where $J_i$ is the eigenvalue $0$ Jordan block matrix of size $(i+1)\times (i+1)$ with the shifts by $-2$ of the identities on the respective objects as the non-zero morphisms.

Next, we observe that the functor
\begin{align*}
\Hom_{\C}(-,-)\colon &\C^{\op}\boxtimes \C\longrightarrow \Bbbk\plmod\simeq \overline{k},
\end{align*}
determined by $Y\otimes Z \mapsto \Hom_{\C}(Y,Z)$ is a $p$-dg functor. Hence, we obtain a $p$-dg functor
\begin{align*}
\boxtimes_\C=\boxtimes \circ \big(\id_\C\boxtimes \Hom_\C(-,-)\big)\colon &\C\boxtimes \C^\op\boxtimes \C\longrightarrow\overline{\C},\\
&X\otimes Y\otimes Z\longmapsto X\otimes \Hom_\C(Y,Z).
\end{align*}
Lemma \ref{extensionlemma} gives a $p$-dg functor $\overline{\C\boxtimes \C^{\op}\boxtimes\C}\to \overline{\C}$, and after pre-composing with the $p$-dg functor $\overline{\C\boxtimes \C^{\op}}\boxtimes\overline{\C}\to \overline{\C\boxtimes \C^{\op}\boxtimes\C}$ from Lemma \ref{tensorlemma}\eqref{tensorlemma4} we obtain the desired $p$-dg functor $\boxtimes_\C$ as claimed in \eqref{bimoduletensor}, providing a regular action of $\C$-bimodules on $\C$-modules which we will use later.  For objects $X$ in $\overline{\C}$ and $Y$ in $\overline{\C^\op}$, we write $\boxtimes_\C(X\otimes Y)=X\otimes_{\C}Y$.

\begin{remark}
Given a finite-dimensional $p$-dg algebra and a $p$-dg idempotent decomposition $1=e_1+\ldots+e_r$, we can consider the $p$-dg category $\A$ associated to $A$ as in Example \ref{algebraex}. Then then $Ae_i$ are projective $A$-modules under the $p$-dg equivalence from Lemma \ref{moduletranslation}. Moreover, the tensor product $\otimes_\A$ recovers the relative tensor product $\otimes_A$ of bimodules. Indeed, for $p$-dg idempotents $e_1,e_2,e_3$ in $A$, we have
\begin{align*}
(X_1\otimes X_2)\otimes_{\A} X_3 &= X_1\otimes \Hom(X_2,X_3)=X_1\otimes e_2Ae_3.
\end{align*}
This object corresponds under the equivalence from Lemma \ref{moduletranslation} to
$
Ae_1\otimes e_2Ae_3,
$ which is $p$-dg isomorphic to the relative tensor product $Ae_1\otimes_{\Bbbk} e_2 A\otimes_A Ae_3$.
\end{remark}

\subsection{Closure under \texorpdfstring{$p$}{p}-dg quotients}\label{pdgquots}

Given a locally finite $p$-dg category $\C$, we define the {\bf closure under $p$-dg quotients} of $\overline{\C}$ to be the category $\quots{\C}$
\begin{itemize}
\item whose objects are diagrams of the form $X \overset{f}{\longrightarrow} Y$ in $\overline{\C}$ such that $\deg f=0$ and $\del(f)=0$;
\item whose morphisms are pairs $(\phi_0,\phi_1)$ of morphisms in $\overline{\C}$ producing solid commutative diagrams of the form
\begin{equation}\label{vvmorph}\xymatrix{ X \ar^{f}[rr] \ar^{\phi_0}[d]&&Y \ar^{\phi_1}[d]\ar_{\eta}@{-->}[dll]\\ 
X'\ar^{f'}[rr] &&Y',
}
\end{equation}
modulo the subgroup generated by diagrams where there exists a morphism $\eta$, as indicated by the dashed arrow, such that $\phi_1=f'\eta$;
\item the differential of such a pair $(\phi_0,\phi_1)$ is simply the componentwise differential in $\overline{\C}$.
\end{itemize}

Indeed, $\vv{\C}$ is a $p$-dg category as the equation $\phi_1 f=f'\phi_0$ implies that $\del(\phi_1) f=f'\del(\phi_0)$ using the Leibniz rule. Further, if $\phi_1=f'\eta$, then $\del(\phi_1)=f'\del(\eta)$, so the ideal of homotopies $\eta$ is a $p$-dg ideal.

There exists a fully faithful $p$-dg functor
$$
\vv{\iota}\colon \overline{\C} \longrightarrow \vv{\C},
$$
which maps an object $X$ of $\overline{\C}$ to $0\to X$, and is defined on morphisms accordingly. It is easy to observe that a $p$-dg functor $\rF\colon \C \to \D$ induces a $p$-dg functor
\[
\quots{\rF}\colon \quots{\C}\longrightarrow\quots{\D}.
\]
First, we have an induced functor $\overline{\rF}$ using Lemma \ref{extensionlemma}, and for an object $X \overset{f}{\longrightarrow} Y$, $\overline{\rF}(f)$ is annihilated by the differential, and componentwise application of $\overline{\rF}$ to a morphism (given by a pair $(\phi_0,\phi_1)$) provides a $p$-dg functor as desired. Given a natural transformation $\alpha\colon \rF\to \rG$, there is an induced natural transformation $\quots{\alpha}\colon \quots{\rF}\to \quots{\rG}$, defined again componentwise. As these assignments are functorial, we obtain a $p$-dg $2$-functor.

\begin{lemma}\label{vecmod}
Assume $\C$ is a locally finite $p$-dg category with a finite set of objects $\mathtt{X}:=\{X_1,\dots, X_r\}$ such that every object in $\C$ has a fantastic filtration by shifts of objects in $\mathtt{X}$. 
Then $\vv{\C}$ is $p$-dg equivalent to 
$\C^\op\plmod$.
\end{lemma}

We remark that while $\vv{\C}$ naturally looks more like the category of finitely presented rather than finitely generated $\C^\op$-modules, these two categories coincide given our finiteness assumptions, cf. Section \ref{pdgcats}.

\proof 
The equivalence in Lemma \ref{matcatequiv} immediately gives a full $p$-dg functor from $\vv{\C}$ to $\C^\op\plmod$, simply by taking cokernels of the given diagrams. This functor is faithful because a morphism as in \eqref{vvmorph} gives a zero morphism $Y/X\to {Y'}/{X'}$ if an only if the morphism $\phi_1\colon Y\to Y'$ factors through the image of $f'$, i.e. is homotopic to zero. 
The fact that this functor is $p$-dg dense comes from the observation that any object in $\C^\op\plmod$ can be written as the cokernel of a $p$-dg morphism between two semi-free objects. 
Explicitly, using the equivalence $\C^\op\plmod\simeq A\plmod$ for a suitable finite-dimensional $p$-dg algebra $A$ from Lemma \ref{moduletranslation},  we may prove this claim using the language of $p$-dg algebras and their modules. A module $M\in A\plmod$ is given as the cokernel of the $p$-dg morphism $$\delta:=\mu \otimes \id_M - \id_A\otimes a \,:\,A\otimes A\otimes M \to A\otimes M$$ where $\mu$ is multiplication in the algebra and $a$ is the action map, both of which are annihilated by the differential by definition.
\endproof

\begin{lemma}\label{abelianlemma}
Under the assumptions of Lemma \ref{vecmod}, the category $\Z(\vv{\C})$ is abelian. 
\end{lemma}
\begin{proof}
This follows from the equivalence between $\Z(\vv{\C})$ and $\Z(\C^\op\plmod)$ and the fact that the latter is abelian.
\end{proof}

In particular, if $\C$ is strongly finitary, $\vv{\C}\simeq\C^\op\plmod$. 

Despite the analogy to the abelianization \cite{Fr} of an additive (or $\Bbbk$-linear category), the category $\vv{\C}$ is \emph{not} abelian. Under the assumptions of the above lemma, denoting by $A$ the $p$-dg algebra associated to $\C$ as in Example \ref{algebraex}, the abelian category $\Z(\C^\op\plmod)$ gives the category of finitely generated $p$-dg $A$-modules as defined in \cite[Section 2.2]{KQ}. The construction $\vv{\C}$ is a $p$-dg enriched version of this category.

Putting together Lemma \ref{matcatequiv}, Lemma \ref{moduletranslation} and Lemma \ref{vecmod}, we have a commutative diagram of $p$-dg equivalences (respectively, inclusions):
\begin{equation}\label{equivcommute}
\vcenter{\hbox{
\xymatrix{
 \overline{\A} \ar@{-}^{\wr}[d]\ar@{^{(}->}[r]& \vv{\A}\ar@{-}^{\wr}[d]\\
 \A^\op\cof \ar@{-}^{\wr}[d]\ar@{^{(}->}[r]& \A^\op\plmod\ar@{-}^{\wr}[d]\\
 A\cof  \ar@{^{(}->}[r]& A\plmod\\
}}}\end{equation}

\begin{lemma}
For locally finite $\C$, $\vv{\C}$ is again locally finite, and $p$-dg idempotent complete.\end{lemma}

\proof
Assume that $\C$ is locally finite.
Directly, from the definition, we see that morphism spaces between $p$-dg indecomposable objects are finite-dimensional, their elements being pairs of morphisms between finite sums of indecomposable objects in $\C$. Furthermore, $\vv{\C}$ is clearly $p$-dg idempotent complete. Hence $\vv{\C}$ is locally finite. Moreover, all $p$-dg idempotent split due to the existence of cokernels in $\vv{\C}$, see Lemma \ref{abelianlemma}. 
\endproof

\subsection{The compact derived category}\label{derivedcats}

Let $\C$ be a locally finite $p$-dg category.
As mentioned in \cite[2.5]{EQ} and \cite[4.6]{EQ2} it is possible to generalize the constructions of stable and derived categories for $p$-dg algebras to the setup of module categories over $\C$. 
For this, recall the tensor action $\otimes$ of graded $H$-modules  on $\overline{\C}$ defined in \eqref{Hmoduletensor}.

\begin{definition}\label{Stdefinition}
A morphism in $\Z(\overline{\C})$ is {\bf null-homotopic} if it factors through an object of the form $X\otimes V$, for $X$ in $\overline{\C}$ and $V$ in $\Bbbk\plmod$ corresponding to a projective $H$-module.

The {\bf stable category of $\overline{\C}$}, denoted by $\St(\overline{\C})$, is the quotient of $\Z(\overline{\C})$ by the ideal formed by all null-homotopic morphisms.
\end{definition}

Equivalently, $f$ is null homotopic if it factors through an object of the form $X\otimes V$ for $X$ in $\overline{\C}$.

\begin{lemma}
A morphism $f\colon X\to Y$ in $\overline{\C}$ is null homotopic if and only if $f=\del^{p-1}g$ for some morphism $g\colon X\to Y$.
\end{lemma}
\begin{proof}
For modules over $p$-dg algebras, this is proved in \cite[Lemma 5.4]{Qi}, and carries over to categories of the form $\C\cof$ using the tensor product introduced in Section \ref{tensorproducts}. The converse uses  \cite[Lemma~1]{Kh} that every null homotopic morphism $f$ factors through $\id\otimes \del^{p-1} \colon X\to X\otimes H$ as this map has a left inverse. The result then holds, under the equivalence of Lemma \ref{matcatequiv}, to categories of the form $\overline{(-)}$.
\end{proof}

%
%
%
%

\begin{lemma}\label{derivedfunctor}
Given a $p$-dg functor $F\colon \overline{\C}\to \overline{\D}$ in $\pdgcat$, there is an induced functor $$\St(F)\colon \St(\overline{\C})\to \St(\overline{\D}),$$
and passing to stable categories $\St(-)$ is functorial. 
\end{lemma}

\begin{proof}

We have to show that $\Z(F)$ preserves null-homotopic morphisms.
By Lemma \ref{restrictionlemma}, the functor $F$ is $p$-dg isomorphic to $\overline{\left.F\right|_{\C}}$. As $p$-dg isomorphisms preserve null-homotopic morphisms, setting $G=\left.F\right|_{\C}$, it suffices to show that any $p$-dg functor of the form $\Z(\overline{G})$ preserves null-homotopic morphisms. Indeed, if a morphism $f$ in $\overline{\C}$ factors through an object $X\otimes V$, then $\Z(\overline{G})(f)=\overline{G}(f)$ factors through $\overline{G}(X\otimes V)=\overline{G}(X)\otimes V$. This follows from the description of the action of $H$-modules on $\overline{\C}$ in \eqref{Hmoduletensor2} and the componentwise definition of $\overline{G}$.
Hence it follows from the functoriality in Lemma \ref{extensionlemma}\eqref{extensionlemma1} that $\St(-)$ is strictly functorial, as the induced morphisms on the stable category are the functors $\overline{(-)}$, quotienting out some morphisms.
\end{proof}

For later use, we require a description of the {\bf cone} of a morphism in $\Z(\overline{\C})$. This adapts the corresponding construction for $p$-dg modules in \cite{Kh}.

\begin{lemma}\label{conelemma}
Let $f\colon X\to Y$ be a morphism in $\Z(\overline{\C})$, where $X=(\bigoplus_{i=1}^mX_i,\alpha)$, $Y=(\bigoplus_{j=1}^nY_j,\beta)$. Then, abbreviating $\tilde f:= f\langle -2\rangle$ and $\tilde \tI_m:=\tI_m\langle -2\rangle$, the cone $C_f$ of $f$ is the object 
\begin{align}\label{cone}
C_f=\left(Y\oplus X\langle 2\rangle\oplus X\langle 4\rangle\oplus\ldots \oplus X\langle 2p-2\rangle, \begin{pmatrix}
\beta&\tilde f&0&0&\hdots&0\\
0&\alpha&\tilde \tI_m&0&\hdots&0\\
0&0&\alpha&\tilde \tI_m&0&0\\
\vdots&&\ddots&\ddots&\ddots&\vdots\\
0&\hdots&&0&\alpha&\tilde \tI_m\\
0&\hdots&&&0&\alpha
\end{pmatrix}\right).
\end{align}

It is part of the pushout diagram
\begin{align}\label{conediag}\vcenter{\hbox{
\xymatrix{
X\ar[rr]^{f}\ar[d]_{\id_X\otimes \del^{p-1}\langle 2p-2\rangle}&&Y\ar[d]^v\\
X\otimes H\langle 2p-2\rangle\ar[rr]_{u}&&C_f}}}
\end{align}
in $\Z(\overline{\C})$, where $v$ is the embedding into the first summand, and $u$ is given by the diagonal matrix with $f,\id_X,\ldots, \id_X$ on the diagonal.
\end{lemma}
\begin{proof}
Following \cite{Kh}, the cone of $f$ can be defined to be right bottom object in the pushout diagram \eqref{conediag}. We claim that $C_f$ as defined in \eqref{cone} satisfies its universal property. It is easy to check that $C_f$ makes the diagram commute and that all morphisms are $p$-dg morphisms. Now let $\gamma=(\gamma_1,\ldots, \gamma_{p})\colon X\otimes H\langle 2p-2\rangle\to Z$ and $\tau\colon Y\to Z$ be morphisms in $\Z(\overline{\C})$ such that $\gamma(\id\otimes \del^{p-1}\langle 2p-2\rangle)=\tau f$. This implies that $\gamma_1=\tau f$. For a morphism $\rho=(\rho_1,\ldots,\rho_{p})$ to satisfy $\rho u=\gamma$ and $\rho v=\tau$, it is necessary that $\rho_1 f=\gamma_1$ and $\rho_1=\tau$, as well as $\rho_i=\gamma_i$ for $i=2,\ldots, p$. Hence the unique morphism satisfying this requirement is $\rho=(\tau, \gamma_2,\ldots,\gamma_{p})$. We can show that $\rho$ is a $p$-dg morphism. Denote the differential upper triangular matrix on $Z$ by $\beta'$. As $\tau$ and $\gamma$ are $p$-dg morphisms, we obtain
\begin{align*}
\del\tau+\beta'\tau-\tau\beta&=0,\\
\del\gamma_i+\beta'\gamma_i-\gamma_i\alpha&=\gamma_{i-1}.
\end{align*}
Again using $\tau f=\gamma_1$ and the differential on $C_f$ from \eqref{cone},  we verify that these equations imply that $\rho=(\tau, \gamma_2,\ldots, \gamma_{p})$ is a $p$-dg morphism from $C_f$ to $(Z,\beta')$.
\end{proof}

Now consider the shift functor $\Sigma\colon \overline{\C}\to \overline{\C}$. It corresponds to tensoring (as defined in Section \ref{tensorproducts}) with the graded $H$-module $V_{p-2}\langle 2p-2\rangle=(H/\del^{p-1}H)\langle 2p-2\rangle$ (cf. \cite[Section 3]{Kh}). Tensoring an object $X$ of $\overline{X}$ with the short exact sequence of graded $H$-modules $$0\to\Bbbk{\longrightarrow}H\langle 2p-2\rangle\to V_{p-2}\langle 2p-2\rangle\to 0,$$ where the map from $\Bbbk$ to $H\langle 2p-2\rangle$ is given by $\del^{p-1}\langle 2p-2\rangle$,
gives a short exact sequence
$$0\longrightarrow X\longrightarrow X\otimes H\langle 2p-2\rangle\longrightarrow \Sigma X\longrightarrow 0$$
in the $\Bbbk$-additive category $\Z(\overline{\C})$. Using Lemma \ref{conelemma} we obtain an analogue of a diagram of short exact sequences  in $\Z(\overline{\C})$ considered in \cite{Kh}: 
\begin{align}
\vcenter{\hbox{
\xymatrix{
0\ar[r]&X\ar[r]\ar[d]_f &X\otimes H\langle 2p-2\rangle\ar[r]^-{q}\ar[d]^{u}& \Sigma X\ar[r]\ar[d]^{\id}&0\\
0\ar[r]&Y\ar[r]^{v}&C_f\ar[r]^{r}& \Sigma X\ar[r]&0,
}}}
\end{align}
where the $p$-dg morphisms $q$ and $r$ are both given by the block matrix $\begin{pmatrix}0&\tI_{p-1}\end{pmatrix}$. These morphisms are annihilated by the differential, and, moreover, $q(\id\otimes \del^{p-1})=0$ and $vr=0$. 

\begin{definition}
Following \cite{Kh}, we define {\bf standard distinguished triangles} in $\St(\overline{\C})$ to be sequences of the form
$$X\stackrel{f}{\longrightarrow}Y\stackrel{v}{\longrightarrow}C_f\stackrel{r}{\longrightarrow}\Sigma X,$$
for a morphism $f\colon X\to Y$ in $\Z(\overline{\C})$. A {\bf distinguished triangle} is any diagram isomorphic to a standard distinguished triangle in $\St(\overline{\C})$.
\end{definition}

The following theorem follows similarly to \cite[Theorem~1]{Kh}, which follows the strategy of \cite{Ha}.

\begin{theorem}\label{triangularfunctor}
Let $\C$ be a $p$-dg category in $\pdgcat$. Then the stable category $\St(\overline{\C})$ is triangulated. Further, a $p$-dg functor $\rF\colon \overline{\C}\to \overline{\D}$ induces a triangulated functor $\St(\rF)\colon \St(\overline{\C})\to \St(\overline{\D})$.
\end{theorem}
\begin{proof}
To verify that $\St(\overline{\C})$ is triangulated, the proof of \cite{Kh} (based on \cite{Ha}) can be adapted as it is formal, and \cite[Lemma~1]{Kh} can also be used in our setup as  it is based on the inclusion of graded $H$-modules $\id_H\otimes \del^{p-1}\langle 2p-2\rangle
:H\to H\otimes H\langle 2p-2\rangle$ splitting, which we can also utilize thanks to the tensor product action of graded $H$-modules on $\overline{\C}$ from Section \ref{tensorproducts}.

Using Lemmas \ref{restrictionlemma} and \ref{derivedfunctor} applied to the explicit description of the cone in Lemma \ref{conelemma}, we see that $\rF(C_f)\cong C_{\rF(f)}$ and $\rF(\Sigma X)\cong \Sigma(\rF X)$. Hence, $\St(\rF)$ commutes with the shift functor $\Sigma$ up to isomorphism, and standard distinguished triangles are mapped to distinguished triangles under $\St(\rF)$.
\end{proof}

Finally, we can consider natural transformations $\lambda\colon F\to G$ such that $\del\lambda=0$, for $F,G\colon \overline{\C}\to \overline{\D}$. Then if $\alpha \colon X\to Y\in \overline{\C}$ is null-homotopic, both $\alpha\lambda_X$ and $\lambda_Y\alpha$ are null-homotopic. Hence $\lambda$ induces a well-defined natural transformation $\St(\lambda)\colon \St(\overline{\C})\to \St(\overline{\D})$, which at an object $X$ of $\overline{\C}$ is defined as $\St(\lambda)_X=\lambda_X$, the image of $\lambda_X$ in the stable category $\St(\overline{\D})$. The next lemma follows immediately.

\begin{lemma}\label{stabletransformations}
Let $\C$, $\D$, $\E$ be categories in $\pdgcat$, and $F,F',F''\colon \overline{\C}\to\overline{\D}$, $G,G'\colon \overline{\D}\to \overline{\E}$ be $p$-dg functors and $\alpha\colon F\to F'$, $\alpha'\colon F'\to F''$, and $\beta\colon G\to G'$ be $p$-dg natural transformations. Then 
\begin{align}
\St(\alpha'\circ_1\alpha)&=\St(\alpha')\circ_1\St(\alpha),\\
\St(\beta\circ_0\alpha)&=\St(\beta)\circ_0\St(\alpha).
\end{align}
\end{lemma}

Denote by $\overline{\C}^\circ$ the $p$-dg idempotent completion of $\overline{\C}$ (cf. Section \ref{fincat}). Similarly to Definition \ref{Stdefinition}, we can define its stable category $\St(\overline{\C}^\circ)$, and all constructions carry through.

\begin{lemma}\label{derivedstablelemma}
Let $\C$ be a locally finite $p$-dg category with a finite set of objects $\mathtt{X}:=\{X_1,\dots, X_r\}$ such that every object in $\C$ has a fantastic filtration by shifts of objects in $\mathtt{X}$.
Then $\St(\overline{\C})$ is a full triangulated subcategory of the compact derived category $\D^c(\C^{\op}\plmod)$, and $\St(\overline{\C}^\circ)$ is equivalent to $\D^c(\C^{\op}\plmod)$.
\end{lemma}

\proof 
It follows from the discussion in Section \ref{algcat} and \cite[Corollary 6.8]{Qi} that $\Z(\overline{\C}^\circ)$ is equivalent to the category of compact cofibrant $A$-modules (see \cite[Section 6]{Qi}) for $A=\End_\C(\bigoplus_{j=1}^r X_j)$. Given the equivalence $\C^{\op}\plmod\simeq A\plmod$ from Lemma \ref{moduletranslation}, the claim follows from  \cite[Corollary 7.15]{Qi}.
\endproof
Compare for example \cite[Proposition~2.4]{Or} for an analogous statement for dg categories.

\section{\texorpdfstring{$p$}{p}-dg \texorpdfstring{$2$}{2}-categories}\label{pdg2cats}

In this section, we collect definitions and results on $2$-categorical constructions with $p$-dg enrichment to introduce the type of $2$-categories for which we will construct $2$-representations in Chapter \ref{pdg2repchapter}.

\subsection{Finitary \texorpdfstring{$p$}{p}-dg \texorpdfstring{$2$}{2}-categories}\label{pdg2cat}

We call a $2$-category $\cC$ a {\bf $p$-dg $2$-category} if the categories $\cC(\ti,\tj)$ are   $p$-dg categories for any pair of objects $\ti,\tj\in \cC$, and horizontal composition is a biadditive $p$-dg functor.

We say that a  $p$-dg $2$-category is  {\bf locally finite} (respectively,  {\bf strongly finitary}) if 
\begin{itemize} \item it has finitely many objects,
\item  for any pair of objects $\ti,\tj\in \cC$, the categories $\cC(\ti,\tj)$ are locally finite (respectively,  strongly finitary);
\item for each $\ti\in \cC$, the identity $1$-morphism $\one_\ti$ is $p$-dg (respectively $\Bbbk$-indecomposable).
\end{itemize}

{\bf A (strict) $p$-dg $2$-functor} $\Phi\colon \cC \to \cD$ of $p$-dg $2$-categories a strict functor of $2$-categories such that for any two objects $\ti,\tj$ of $\cC$ the restriction $\Phi\colon \cC(\ti,\tj)\to \cD(\Phi \ti,\Phi\tj)$ is a $p$-dg functor.

We will later study $p$-dg categories up to \emph{$p$-dg biequivalence}, this uses a weaker concept of $2$-functor, namely that of a \emph{bifunctor} between $2$-categories. The definition in the $p$-dg context is as follows: A {\bf $p$-dg bifunctor} $\Phi\colon \cC\to \cD$ is a bifunctor (see e.g. \cite[4.1]{MP}, \cite[1.1]{Le}) such that the restrictions $\Phi_{\ti,\tj}\colon \cC(\ti,\tj)\to \cD(\Phi\ti,\Phi\tj)$ are $p$-dg functors and there are $p$-dg isomorphisms
\begin{align*}
\one_{\Phi \ti} &\to \Phi_{\ti,\ti} (\one_\ti) , & (-)\circ (\Phi_{\tj,\tk}\times \Phi_{\ti,\tj})\to \Phi_{\ti,\tk}\circ (-),
\end{align*}
satisfying the same coherences as in the case of non-enriched $2$-categories.

\begin{definition}
Let $\cC$ and $\cD$ be $p$-dg $2$-categories. We say that there is a {\bf $p$-dg biequivalence} $\cC\approx \cD$ if there is a biequivalence of $2$-categories given by $p$-dg bifunctors $\Phi\colon \cC\to \cD$, $\Psi \colon \cD\to \cC$ such that the natural isomorphisms $\theta\colon \id_{\ccD}\to \Phi\Psi$, $\lambda\colon \Psi\Phi\to \id_{\ccC}$ are $p$-dg $2$-isomorphisms (graded, of degree zero).
\end{definition}

\begin{lemma}\label{pdgbieqchar}
A $p$-dg bifunctor $\Phi$ is part of a $p$-dg biequivalence if and only if $\Phi$ is surjective on objects (up to $p$-dg equivalence) and locally, i.e. for each pair of objects $\ti$, $\tj$, $\Phi_{\ti,\tj}$ is a $p$-dg equivalence.
\end{lemma}
\begin{proof}
We sketch how this result from the theory of $2$-categories (see e.g. \cite[2]{Le}) extends to the $p$-dg enriched version.

Note that if $\Phi$ is part of a $p$-dg biequivalence, then it clearly has the stated properties. We prove the converse by constructing $\Psi$.

Let $\tk$ be an object of $\cD$. Using $p$-density, there exists a $p$-dg equivalence $\rF_\tk\colon \tk \to \Phi\ti$ for some object $\ti$ of $\cC$, i.e. for the $1$-morphism $\rF_\tk$ there exists $\rG_\tk\colon \Phi\ti\to \tk$ such that $\rF_\tk\rG_\tk$ is $p$-dg isomorphic to $\one_{\Phi \ti}$, and $\rG_\tk\rF_\tk$ to $\one_{\tk}$. We define $\Psi \tk:=\ti$. By assumption, we can further choose $p$-dg functors $\Psi'_{\ti,\tj}\colon \cD(\Phi\ti,\Phi\tj) \to \cC(\ti,\tj)$ and $p$-dg isomorphisms $\Psi'_{\ti,\tj}\Phi_{\ti,\tj}\cong  \id_{\ccC(\ti,\tj)}$, $\Phi_{\ti,\tj}\Psi'_{\ti,\tj}\cong \id_{\ccD(\Phi\ti,\Phi\tj)}$. Note that the $p$-dg functors $\Psi'_{\ti,\tj}$ also inherit coherences $\circ (\Psi'_{\ti,\tj}\times \Psi'_{\th,\ti})\to \Psi'_{\th,\tj}\circ$ from $\Phi$ for any objects $\th,\ti$, and $\tj$ of $\cC$, as well as a $p$-dg isomorphism $\one_\ti\cong \Psi'_{\ti,\ti} (\one_{\Phi\ti})$ using the composition $\one_\ti \cong \Psi'\Phi(\one_\ti)\cong \Psi'(\one_{\Phi\ti})$.
The $p$-dg functors $\Psi'$ can now be used to define for a $1$-morphism $\rH\colon \tk\to \tl$ in $\cD$
\begin{align*}
\Psi_{\tk,\tl}(\rH)&:=\Psi'_{\Psi \tk,\Psi \tl}(\rF_{\tl}\rH\rG_\tk),
\end{align*}
and for $\alpha\colon \rH\to \rH'$ in $\cD(\tk,\tl)$,
\begin{align*}
\Psi_{\tk,\tl}(\alpha)&:=\Psi'_{\Psi\tk,\Psi\tl}(\id_{\rF_{\tl}}\circ_0 \alpha \circ_0 \id_{\rG_{\tk}}).
\end{align*}
This way, we obtain a $p$-dg functor $\Psi_{\tk,\tl}\colon \cD(\tk,\tl)\to \cC(\Psi\tk,\Psi\tl)$, noting that $\tk=\Phi\Psi\tk$. This data gives a bifunctor $\cD\to \cC$, where the isomorphism $\one_{\Psi \tk} \to \Psi_{\tk,\tk} (\one_\tk)$ is given by the composition of $p$-dg isomorphisms
$$\one_{\Psi\tk}=\one_{\ti}\cong \Psi'_{\ti,\ti}(\one_{\Phi\ti})\cong \Psi'_{\ti,\ti}(\rF_\tk\rG_\tk)=\Psi_{\tk,\tk} (\one_\tk),$$
and the composition isomorphism $\circ (\Psi_{\tl,\tm}\times \Psi_{\tk,\tl})\to \Psi_{\tk,\tm}\circ$ is obtained as the composition of $p$-dg isomorphisms
\begin{align*}
\Psi_{\tl,\tm}(\rL)\circ \Psi_{\tk,\tl}(\rK)&=
\Psi'_{\Psi\tl,\Psi\tm}(\rF_{\tm}\rL\rG_\tl)\circ
\Psi'_{\Psi\tk,\Psi\tl}(\rF_\tl\rK \rG_{\tk})\\
&\cong \Psi'_{\Psi\tk,\Psi\tm}(\rF_{\tm}\rL\rG_\tl \rF_\tl\rK \rG_{\tk})\\
&\cong\Psi'_{\Psi\tk,\Psi\tm}(\rF_{\tm}\rL\rK \rG_{\tk})=\Psi_{\tk,\tm}(\rL\rK).
\end{align*}
Similarly to the non-enriched case, we see that $\Psi$ defines a bifunctor, which hence is a $p$-dg bifunctor. It remains to note that by construction, the isomorphisms $\Psi_{\Phi\ti,\Phi\tj}\Phi_{\ti,\tj}\cong \id_{\ccC(\ti,\tj)}$ and $\Phi_{\Psi\tk,\Psi\tl}\Psi_{\tk,\tl}\cong \id_{\ccD(\tk,\tl)}$ are $p$-dg isomorphisms, hence we obtain a $p$-dg biequivalence.
\end{proof}

Let $\cEnd(\coprod_{\ti\in \tI}\C_i)$ denote the $2$-category given by objects $\ti \in \tI$ corresponding to $p$-dg categories $\C_\ti$, and $1$-morphisms given by all $p$-dg functors between these categories. Defining $2$-morphisms to be all $\Bbbk$-linear natural transformations we obtain a $p$-dg $2$-category. The following lemma will be used later:

\begin{lemma}\label{functorbieq}
A collection of $p$-dg equivalences of $p$-dg categories $\C_\ti$ and $\D_\ti$ for any $\ti\in \tI$ gives a $p$-dg biequivalence of $\cEnd(\coprod_{\ti\in I}\C_\ti)$ and $\cEnd(\coprod_{\ti\in I}\D_\ti)$.
\end{lemma}
\begin{proof}
We aim to define a $p$-dg bifunctor $\Psi\colon \cC:=\cEnd(\coprod_{\ti\in I}\C_i)\to \cD:=\cEnd(\coprod_{\ti\in I}\D_i)$ and apply Lemma \ref{pdgbieqchar}. We set $\Psi\ti:=\ti$ and fix $p$-dg functors $\rF_\ti\colon \C_\ti\to \D_\ti$, $\rG_\ti\colon \D_\ti \to \C_\ti$, with $p$-dg isomorphisms $\theta_i\colon \rG_\ti\rF_\ti\to \id_{\C_\ti}$, $\lambda_\ti \colon \id_{\D_\ti}\to \rF_\ti\rG_\ti$. For $\rK \in \cC(\ti,\tj)$, set $\Psi (\rK):= \rF_\tj\rK\rG_\ti$, giving an object of $\cD(\ti,\tj)$. For a $2$-morphism $\alpha\colon \rK\to \rL$, we define $\Psi (\alpha):= \id_{\rF_\tj}\circ_0\alpha\circ_0 \id_{\rG_\ti}$. This gives a $p$-dg functor $\Psi_{\ti,\tj}\colon \cC(\ti,\tj) \to \cD(\ti,\tj)$ which is part of a $p$-dg equivalence.

In order to obtain a $p$-dg bifunctor, we define the $p$-dg isomorphism $\one_{\Psi_\ti}=\one_{\ti}\to \Psi(\one_\ti)=\rF_\ti \rG_\ti$ to be $\lambda_\ti$, and a $p$-dg isomorphism $\Psi_{\tj,\tk}(\rL)\circ_0\Psi_{\ti,\tj}(\rK)=\rF_\tk\rL\rG_\tj\rF_\tj\rK\rG_\ti \to \rF_\tk\rL\rK\rG_\ti=\Psi_{\ti,\tk}(\rL\rK)$ using the $p$-dg isomorphism $\id_{\rF_\tk \rL}\circ\theta_\tj\circ_0\id_{\rK\rG_\ti}$. One checks that this assignment satisfies the coherence axioms of a $p$-dg bifunctor similarly to the non-enriched case.
\end{proof}

For a $p$-dg $2$-category $\cC$ we define the $2$-category $\cZ\cC$ as the $2$-subcategory which has the same objects, and the same $1$-morphisms as $\cC$, but $\Hom_{\ccZ\ccC}(\rF,\rG) = \Z\Hom_{\ccC}(\rF,\rG)$ for $1$-morphisms $\rF$ and $\rG$, i.e. $\cZ\cC(\ti,\tj):=\Z(\cC(\ti,\tj))$. 

\subsection{Subquotient idempotent completion of \texorpdfstring{$p$}{p}-dg \texorpdfstring{$2$}{2}-categories}

For application in Section \ref{sl2cat}, we record that the subquotient idempotent completion from Section \ref{fincat} is a coherent way to enable the {\bf subquotient idempotent completion ${\cC\,}^{\bullet}$} of a $p$-dg $2$-category $\cC$. 

\begin{proposition}\label{2completion}
Let $\cC$ be a locally finite $2$-category. Then there exists a locally finite 
 $2$-category $\cC\,^{\bullet}$ such that $\cC\,^{\bullet}(\ti,\tj)=\cC(\ti,\tj)^{\bullet}$ for any choice of two objects $\ti$, $\tj$. Further, there exists a $2$-fully faithful $p$-dg $2$-functor $\cC\hookrightarrow \cC\,^{\bullet}$ (cf. Section \ref{pdg2cat}).
\end{proposition}
\begin{proof}
For $\rF_e=(\rF,e)$ an object in $\cC(\ti,\tj)^{\bullet}$ (where $e\colon \rF\to \rF$ is an subquotient idempotent) and $(\rG,f)$ in $\cC(\tj,\tk)^{\bullet}$, we define 
$$(\rG,f)\circ (\rF,e)=(\rG\circ \rF, f\circ_0 e).$$
We claim that this is well-defined. Indeed, firstly, $f\circ_0 e$ is a subquotient idempotent in $\cC(\ti,\tk)$. It is dominated by $w\circ_0 v$, where $w$ dominates $f$ and $v$ dominates $e$. Note that
$$\del(w\circ_0 v)=\del(w)\circ_0 v+w\circ_0 \del(v)=\del(w)w\circ_0 v+w\circ_0 \del(v)v=(\del(w\circ_0 v))\circ_1 (w\circ_0 v).$$
One similarly sees that $(w-f)\circ_0 (v-e)$ is also a submodule idempotent.
Secondly, a morphism $\alpha=e_2\alpha e_1\colon (\rF_1,e_1)\to (\rF_2,e_2)$ in $\cC(\ti,\tj)$ is mapped to $\id_{\rG}\circ_0\alpha$. It follows easily that the operation thus defined gives a $p$-dg functor $(\rG,f)\circ (-)$
 from $\cC(\ti,\tj)^{\bullet}\to \cC(\ti,\tk)^{\bullet}$. Associativity is inherited from $\cC$, and the assignments commute with the differential using the Leibniz rule for $\circ_0$. It also follows directly that for $\beta=f_2\beta f_1\colon (\rG_1,f_1)\to (\rG_2,f_2)$, $\beta\circ_0 (-)$ defines a natural transformation $(\rG_1,f_1)\circ(-)\to (\rG_2,f_2)\circ(-)$. The assignment $\beta\mapsto \beta\circ_0(-)$ is again functorial and commutes with $\del$. 
 
It is clear that $\cC$ embeds $2$-functorially into $\cC\,^{\bullet}$ by mapping $\rF\mapsto(\rF,\id_{\rF})$. Since for any locally finite category $\D$, $\D\hookrightarrow \D^{\bullet}$ is fully faithful, the $2$-functor obtained this way is $2$-fully faithful.
\end{proof}

\subsection{Passing to stable \texorpdfstring{$2$}{2}-categories}\label{stable2cat}

Let $\cC$ be a locally finite $2$-category. Then all the $p$-dg categories $\cC(\ti,\tj)$ are locally finite, and we can pass to $\St(\overline{\cC(\ti,\tj)})$. In this section, we show that this passage is functorial, and we hence obtain a stable $2$-category $\cSt(\overline{\cC})$. Composition in the stable $2$-category is compatible with the triangulated structures on the categories $\St(\overline{\cC(\ti,\tj)})$.

First, we need the following construction of extending a given $p$-dg $2$-category $\cC$ to categories of semi-free modules $\overline{(-)}$ over the $\cC(\ti,\tj)$: The {\bf semi-free $p$-dg $2$-category $\overline{\cC}$} associated  to $\cC$ is defined as having
\begin{itemize}
\item the same objects $\ti$ as $\cC$;
\item the $1$-hom categories $\overline{\cC}(\ti,\tj)=\overline{\cC(\ti,\tj)}$ for any choice of two objects $\ti$, $\tj$;
\item  horizontal composition of two $1$-morphisms $\rX=(\bigoplus_{m=1}^s\rF_m,\alpha)\in \overline{\cC(\tk,\tl)}$ and $X'=(\bigoplus_{n=1}^t \rF'_n,\alpha')\in\overline{\cC(\tj,\tk)}$ given by
\begin{align}\label{orderingconvention}
\rX\circ \rX'=\left( \bigoplus_{(m,n)}\rF_m\rF_n', (\delta_{k',l'}\alpha_{k,l}\circ_0\id_{\rF'_{k'}}+\delta_{k,l}\id_{\rF_k}\circ_0\alpha'_{k',l'})_{(k,k'),(l,l')} \right),
\end{align}
where we fix the convention that pairs $(m,n)$ are ordered lexicographically;
\item vertical composition given simply by composition in $\overline{\cC(\ti,\tj)}$, while horizontal composition $\gamma\circ_0\tau$ for morphisms $\gamma=(\gamma_{k,l})_{k,l}\in \overline{\cC(\tj,\tk)}$, $\tau=(\tau_{m,n})_{m,n}\in \overline{\cC(\ti,\tj)}$ is given by
\begin{align}\label{horizontalcomp}
\left(\gamma\circ_0\tau\right)_{(k,m),(l,n)}&=\gamma_{k,l}\circ_0\tau_{m,n},
\end{align}
using the same ordering convention on pairs.
\end{itemize}

\begin{proposition}\label{overline2cat}
We can equip $\overline{\cC}$ with the structure of a locally finite $p$-dg $2$-category into which $\cC$ embeds as a $p$-dg $2$-subcategory.
\end{proposition}
\begin{proof}
We start by considering a general object $\rX=(\bigoplus_{m=1}^s\rF_m,\alpha)\in \overline{\cC(\ti,\tj)}$ and show that a right composition $(-)\circ\rX\colon \cC(\tj,\tk)\to \cC(\ti,\tk)$ can be defined, and is strictly functorial. For this, we map an object $\rG\in \cC(\tj,\tk)$, and a morphism $\tau\colon \rG\to \rG' \in\cC(\tj,\tk)$, to 
\begin{align}
\rG\circ \rX&=\left(\bigoplus_{m=1}^s\rG\rF_m,( \id_{\rG}\circ_0\alpha_{k,l})_{k,l} \right),\\
\tau\circ \rX&=\tau\circ_0\id_{\rX}=\left(\delta_{m,m'}\tau\circ_0\id_{\rF_m}\right)_{m,m'}.
\end{align}

Note that this definition is \emph{not} just up to isomorphism as the object $X$ is a list (rather than an internal direct sum) of objects, cf. Remark \ref{endostrongfin}. 
We have to verify that the functor $(-)\circ\rX$ thus defined is a $p$-dg functor. Indeed,
\begin{align*}
\del(\tau\circ_0\rX)&=\del(\tau\circ_0\id_{\rX})+\left( \id_{\rG'}\circ_0\alpha\right)\circ_1\left(\tau\circ_0\id_{\rX}\right)-\left(\tau\circ_0\id_{\rX}\right)\circ_1\left(\id_{\rG}\circ_0 \alpha\right)\\
&=\del(\tau)\circ_0\id_{\rX},
\end{align*}
using that $\cC$ is a $2$-category.
Further, given a morphism $\gamma=(\gamma_{k,l})_{k,l} \colon \rX\to \rY$ in $\overline{\cC(\ti,\tj)}$, we can consider the induced morphism
\begin{equation}
\rG\circ\gamma=(\id_\rG\circ_0\gamma_{k,l})_{k,l}\colon \rG\circ \rX\to \rG\circ \rY.
\end{equation}
It follows that $(-)\circ\gamma$ gives a natural transformation $(-)\circ\rX\to (-)\circ\rY$, using that $\cC$ is a $2$-category. The construction is strictly compatible with vertical composition of morphisms in $\overline{\cC(\ti,\tj)}$ and commutes with the differential.

We can induce a $p$-dg functor $\overline{\rX}\colon \overline{\cC(\tj,\tk)}\to \overline{\cC(\ti,\tk)}$ using Lemma \ref{extensionlemma}\eqref{extensionlemma1} applied to $(-)\circ \rX$, and then obtain induced natural transformations $(-)\circ\gamma\colon \overline{\rX}\to \overline{\rY}$ for $\gamma\colon\rX\to \rY$ a morphism in $\in \overline{\cC(\ti,\tj)}$, using Lemma \ref{extensionlemma}\eqref{extensionlemma3}. By the same lemma, composition of the induced natural transformations will be functorial in $\gamma$ and preserve the differential. If we can show that $\rG\circ(\rX\circ\rX')=(\rG\circ\rX)\circ\rX'$,  then $\overline{\rX\circ\rX'}=\overline{\rX}\circ\overline{\rX'}$ will also hold, using the functoriality statement  in Lemma \ref{tensorlemma}\eqref{tensorlemma1}. Indeed, for $\rX'=(\bigoplus_{n=1}^t \rF'_n,\alpha')$ in $\overline{\cC(\th,\ti)}$, we have that $(\rG\circ\rX)\circ\rX'$ equals
\begin{align*}
\left(\bigoplus_{m=1}^s\rG\rF_m \left(\bigoplus_{n=1}^t\rF_{n}'\right),
\left(\id_{\rG}\circ_0(\alpha_{k,l}\circ_0 \delta_{k',l'}\id+\delta_{k,l}\id\circ_0\alpha_{k',l'})\right)_{(k,k'),(l,l')}\right),
\end{align*}
applying Lemma \ref{extensionlemma}.
With our ordering convention in \eqref{orderingconvention}, and using strict associativity of composing $1$-morphisms in $\cC$, this object equals $\rG\circ (\rX\circ \rX')$.

Note that evaluating the identity just proved at a more general object $\rY\in\overline{\cC(\tj,\tk)}$ we find that
$$(\rY\circ\rX)\circ\rX'=\rY\circ(\rX\circ\rX'),$$
i.e. composition of $1$-morphisms in $\overline{\cC}$ is strictly associative.

Next, we have to show that the resulting structure $\overline{\cC}$ gives a $2$-category. This follows from the fact that  $\cC$ is a $2$-category, and the following computation for $(\gamma_{i,l})_{i,l}\colon \rX\to \rX'$, $(\gamma'_{k,i})_{k,i}\colon \rX'\to \rX''$ in $\overline{\cC(\tj,\tk)}$ and $(\tau_{j,n})_{j,n}\colon \rY\to \rY'$, $(\tau'_{m,j})_{m,j}\colon \rY'\to \rY''$ in $\overline{\cC(\ti,\tj)}$:
\begin{align*}
\left((\gamma'\circ_1\gamma)\circ_0(\tau'\circ_1\tau)\right)_{(k,l), (m,n)}&=(\gamma'\circ_1\gamma)_{(k,m)}\circ_0(\tau'\circ_1\tau)_{(l,n)}\\&=
\left(\sum_{i}\gamma'_{k,i}\circ_1\gamma_{i,l}\right)\circ_0\left(\sum_{j}\tau'_{m,j}\circ_1\tau_{j,n}\right)\\
&=\sum_{(i,j)}\left(\gamma'_{k,i}\circ_1\gamma_{i,l}\right)\circ_0\left(\tau'_{m,j}\circ_1\tau_{j,n}\right)\\
&=\sum_{(i,j)}\left(\gamma'_{k,i}\circ_0\tau'_{m,j}\right)\circ_1\left(\gamma_{i,l}\circ_0\tau_{j,n}\right)\\
&=\sum_{(i,j)}\left(\gamma'\circ_0\tau'\right)_{(k,m),(i,j)}\circ_1\left(\gamma\circ_0\tau\right)_{(i,j),(l,n)}\\
&=\left((\gamma'\circ_0\tau')\circ_1(\gamma\circ_0\tau)\right)_{(k,m), (l,n)}.\qedhere
\end{align*}
\end{proof}

By Lemma \ref{derivedfunctor}, passing to the quotient categories $\St(\overline{\cC(\ti,\tj)})$ is functorial, since null-homotopic morphisms are mapped to null-homotopic morphisms. Hence we can define the {\bf stable $2$-category $\cSt\overline{\cC}$ associated to $\cC$} as having
\begin{itemize}
\item the same objects $\ti$ as $\cC$;
\item morphism categories $\cSt\overline{\cC}(\ti,\tj)=\St(\overline{\cC(\ti,\tj)})$;
\item the induced composition structures from $\overline{\cC}$ for $1$-morphisms, as well as horizontal and vertical composition of $2$-morphisms.
\end{itemize}
To see that $\St(\overline{\cC})$ gives a well-defined $2$-category, we require the following Lemma:

\begin{lemma}\label{nullhomotopiccomposition}
The set of null-homotopic $2$-morphisms in $\overline{\cC}$ is stable under horizontal composition on the right and on the left.
\end{lemma}
\begin{proof}
Note that if $\gamma$ is null-homotopic in $\overline{\cC(\tj,\tk)}$, then $\gamma\circ_1\gamma'$ is null-homotopic for compatible $\gamma'$. We have to check that also $\gamma\circ_0\tau$ is null-homotopic for any $\tau \in \overline{\cC(\ti,\tj)}$, and the analogous statement for horizontal composition on the right with a null-homotopic morphism in $\overline{\cC(\ti,\tj)}$. To see this, we observe that if $\rX'$, $\rX$ are compatible $1$-morphisms in $\cC$ and $V$ in $H\lmod$, then 
\begin{align}
\rX'\circ(\rX\otimes V)\cong (\rX'\circ\rX)\otimes V\cong (\rX'\otimes V)\circ\rX
\end{align}
are $p$-dg isomorphic $1$-morphisms.
This follows by considering the upper triangular matrices corresponding to the differentials in \eqref{orderingconvention} and \eqref{Hmoduletensor2}. Indeed, if $\rX=(\bigoplus_{m=1}^s\rF_m,\alpha)\in \overline{\cC(\ti,\tj)}$, $\rX'=(\bigoplus_{n=1}^t\rF'_n,\alpha)\in \overline{\cC(\tj,\tk)}$, and $V_i$ the indecomposable graded $H$-module of dimension $i+1$, then 
\begin{align*}
(\rX'\circ \rX)\otimes V_i&=\left( \bigoplus_{(n,m)}(\rF_n'\rF_m)^{\dim_t V_i},(\alpha'\circ_0\id_{\rX}+\id_{\rX'}\circ_0\alpha)\otimes \tI_{i+1}+\id_{\rX'\circ\rX}\otimes J_i\right)\\
&\cong\left(\bigoplus_{(n,m)}\rF'_{n}(\rF_m^{\dim_t V_i}),\alpha'\circ_0\id_{\rX\otimes V_i}+\id_{\rX'}\circ_0(\alpha\otimes \tI_{i+1}+\id_{\rX}\otimes J_i)\right),
\end{align*}
where for the second isomorphism we have to reorder the list of objects, which corresponds to a $p$-dg isomorphism. Here, $J_i$ denotes a Jordan block as in \eqref{Hmoduletensor2}. Further,
\begin{align*}
(\rX'\otimes V_i)\circ\rX&=\left(
\bigoplus_{(n,m)}{\rF'}_n^{\dim_t V_i}\rF_m, (\alpha'\otimes \tI_{i+1}+\id_{\rX'}\otimes J_i)\circ_0\id_{\rX}+\id_{\rX'\otimes V_i}\circ_0\alpha\right),
\end{align*}
which, again, is $p$-dg isomorphic to $(\rX'\circ \rX)\otimes V_i$ using a permutation of the list of objects.
These computations use that the degree of compositions $\rF'_n\rF_m$ is the sum of the degrees, and hence
$$(\rF'_n\rF_m)\langle1\rangle\cong(\rF'_n\langle1\rangle)\rF_m\cong\rF'_n(\rF_m\langle1\rangle)$$
for the degree shifts.
\end{proof}

The preceding Lemma also gives us a compatibility between the triangulated structures for different $1$-hom categories $\cSt\overline{\cC}(\ti,\tj)$. In fact, denote the shift functor of the triangulated category $\St(\cC(\ti,\tj))$ by $\Sigma$. According to Lemma \ref{nullhomotopiccomposition}, we see that there are $p$-dg isomorphisms
\begin{align}\label{shiftcomp}
\Sigma(\rX)\circ \rY\cong \rX\circ \Sigma(\rY)\cong \Sigma(\rX\circ \rY),
\end{align}
for  compatible $1$-morphisms $\rX$, $\rY$ in $\cSt\overline{\cC}$.

\begin{lemma}\label{trianglespreserved}
For any distinguished triangle
$$\rX\longrightarrow \rY\longrightarrow \rZ\longrightarrow\Sigma \rX$$ in $\cSt\overline{\cC}(\ti,\tj)$ and any objects $\rM$ in $\cSt\overline{\cC}(\th,\ti)$ and $\rN$ in $\cSt\overline{\cC}(\tj,\tk)$, we have two distinguished triangles 
\begin{align*}
\rX\circ \rM\longrightarrow \rY\circ \rM&\longrightarrow \rZ\circ \rM\longrightarrow \Sigma (\rX\circ \rM),\\
\rN\circ \rX\longrightarrow \rN\circ \rY&\longrightarrow \rN\circ \rZ\longrightarrow \Sigma(\rN\circ  \rX).
\end{align*}
\end{lemma}

\begin{proof}
For a fixed $1$-morphism $\rM$ in $\overline{\cC}$, composition gives a $p$-dg functor $(-)\circ \rM\colon \overline{\cC(\ti,\tj)}\to \overline{\cC(\th,\tj)}$. It then follows from Theorem \ref{triangularfunctor} that the functor obtained from $(-)\circ\rM$ by passing to the stable $1$-hom categories preserves distinguished triangles. The isomorphisms in \eqref{shiftcomp} now enable us to produce a distinguished triangle as required.
\end{proof}

Hence, $\cSt\overline{\cC}$ preserves distinguished triangles and is compatible with shifts. We therefore regard $\cSt\overline{\cC}$ as a {\bf triangulated $2$-category}, although we refrain from giving an axiomatic description of this concept.

\begin{remark}
We can compare the compatibilities in $\cSt(\overline{\cC})$ with some of the axioms given in \cite[Section~4]{Ma} for a triangulated closed symmetric monoidal category. Lemma \ref{trianglespreserved} gives the first two induced distinguished triangles in \cite[(TC2)]{Ma}. 
\end{remark}

Moreover, we can informally regard a $2$-functor (or bifunctor) which locally preserves distinguished triangles and the shift $\Sigma$ as {\bf triangulated $2$-functors}, and biequivalences satisfying such a compatibility as \textbf{triangulated biequivalences}.

\begin{corollary}
If $\cC$ and $\cD$ are $p$-dg biequivalent $p$-dg $2$-categories, then $\cSt\overline{\cC}$ and $\cSt\overline{\cD}$ are biequivalent triangulated $2$-categories.
\end{corollary}
\begin{proof}
Given a $p$-dg biequivalence $\Psi\colon \cC\to \cD$, we locally obtain $p$-dg functors
$\Psi_{\ti,\tj}\colon \cC(\ti,\tj)\to \cD(\ti,\tj)$. These extend to $p$-dg functors $\overline{\Psi_{\ti,\tj}}\colon \overline{\cC(\ti,\tj)}\to \overline{\cD(\ti,\tj)}$. Using Lemma \ref{triangularfunctor}, it follows that the induced functors $\St(\overline{\cC(\ti,\tj)})\to \St(\overline{\cD(\ti,\tj)})$ are triangulated.
\end{proof}


\section{\texorpdfstring{$p$}{p}-dg \texorpdfstring{$2$}{2}-representations}\label{pdg2repchapter}

\subsection{Definitions}
In general, a $2$-representation of a $2$-category $\cC$ is a strict $2$-functor $\bfM$ from $\cC$ to a fixed target $2$-category.

An {\em ideal} $\bfI$ of $\bfM$ is given by a collection of 
ideals $\bfI(\ti)\subseteq\bfM(\ti)$ for $\ti\in\cC$, which is stable under the action of $\cC$. More precisely, for any morphism 
$\eta$ in some $\bfI(\ti)$ and any $1$-morphism $\rF$, the composition
$\bfM(\rF)(\eta)$, if  defined, is again in $\bfI$. For example, left $2$-ideals of the $2$-category $\cC$ give rise to ideals in principal $2$-representations.

For the purpose of this paper, we present $p$-dg enriched versions of $2$-representations.

By an {\bf additive $p$-dg $2$-representation} of a locally finite $p$-dg $2$-category $\cC$, we mean a strict $2$-functor $\bfM: \cC \to \cofcat$ such that locally the functors $\cC(\ti,\tj)$ to $\cofcat(\bfM(\ti),\bfM(\tj))$ are $p$-dg functors.
Explicitly, $\bfM$ sends 
\begin{itemize}
\item an object $\ti\in \cC$ to a $p$-dg category $\bfM(\ti)$ $p$-dg equivalent to  $\overline{\C_\ti}$ for a locally finite $p$-dg category $\C_\ti$,
\item a $1$-morphism $\rF \in \cC(\ti,\tj)$ to a $p$-dg functor $\bfM(\ti)\to\bfM(\tj)$,
\item a $2$-morphism $\alpha: \rF\to\rG \in \cC(\ti,\tj)$ to a morphism of $p$-dg functors.
\end{itemize}
We call an additive $p$-dg $2$-representation  
{\bf strongly finitary} if its target is 
$\cofcat^{\mathrm{sf}}$. That is, $\overline{\C_\ti}$ is strongly finitary for any object $\ti$.

We remark that by restriction of the codomain, an additive $p$-dg representation $\bfM$ of $\cC$ is equivalent to a strict $p$-dg $2$-functor $\cC \to \cEnd(\coprod_{\ti}\bfM(\ti))$.

A \textbf{$p$-dg ideal} $\bfI$ in a $p$-dg $2$-representation $\bfM$ is an ideal $\bfI\leq \bfM$, such that the ideals $\bfI(\tj)\leq \bfM(\tj)$ are closed under $\del$.

\begin{lemma}
Given an  additive (or strongly finitary) $p$-dg $2$-representation $\bfM$ of a locally finite $p$-dg $2$-category $\cC$ and a $p$-dg ideal $\bfI$ of $\bfM$, the quotient $2$-representation $\bfM/\bfI$ is again an additive (respectively,  strongly finitary) $p$-dg $2$-representation.
\end{lemma}
\begin{proof}
By construction, the functor $\bfM/\bfI$ is a $p$-dg $2$-functor. We only have to check that for each object $\tj$, the $p$-dg category $\bfM/\bfI(\tj)$ is again in $\cofcat$ (respectively, $\cofcat^{\mathrm{sf}}$) as required. But $\bfM(\tj)$ is $p$-dg equivalent to $\overline{\C_j}$ for a locally finite (respectively, strongly finitary) $p$-dg category $\C_j$ by definition. Denote by $\I_j$ the ideal in $\overline{\C_j}$ corresponding to $\bfI(\tj)$ under this equivalence, and by $\D_j$ the locally finite $p$-dg category with the same objects as $\C_j$ but morphism spaces $\Hom_{\D_j}(X,Y) = \Hom_{\overline{\C_j}/\I_j}((X,(0)),(Y,(0))$ for objects $X$ and $Y$ of $\C_j$. As $\I_j$ is closed under $\del$, this is well-defined.
Using composition with injections an projections, it is easy to check that $\overline{\C_j}/\I_j$ (and hence $\bfM/\bfI(\tj)$) is $p$-dg equivalent to $\overline{\D_j}$, which has the correct finiteness properties, and the claim follows.
\end{proof}

\begin{defex}\label{principal}
For $\cC$ a locally finite $p$-dg $2$-category and $\ti$ one of its objects, we define the {\bf $\ti$-th principal $p$-dg $2$-representation} $\bfP_\ti$, which sends 
\begin{itemize}
\item an object $\tj$ to $\overline{\cC(\ti, \tj)}$,
\item a $1$-morphism $\rF$ in $\cC(\tj, \tk)$ to the functor $\overline{\cC(\ti, \tj)}\to \overline{\cC(\ti, \tk)}$ induced by composition, using Lemma \ref{extensionlemma}(i),
\item a $2$-morphism to the induced morphism of $p$-dg functors, obtained by Lemma \ref{extensionlemma}(ii).
\end{itemize}

Observe that $\bfP_\ti$ is strongly finitary provided that $\cC$ is.
\end{defex}

\subsection{The \texorpdfstring{$2$}{2}-category of \texorpdfstring{$p$}{p}-dg \texorpdfstring{$2$}{2}-representations}\label{2rep2cat}

Let $\bfM$ and $\bfN$ be two additive $p$-dg $2$-representations  of a locally finite $2$-category $\cC$. By a {\bf morphism of $p$-dg $2$-representations} $\Psi: \bfM\to\bfN$ we mean a (non-strict) $2$-natural transformation consisting of
\begin{itemize}
\item a map, which assigns to every $\ti\in\cC$ a $p$-dg functor $\Psi_\ti:\bfM(\ti)\to \bfN(\ti)$ and 
\item for any $1$-morphism $\rF\in\cC(\ti,\tj)$ a natural $p$-dg isomorphism $$\eta_{\rF}=\eta^{\Psi}_{\rF}:\Psi_\tj\circ \bfM(\rF)\longrightarrow \bfN(\rF)\circ \Psi_\ti,$$ 
such that for composable $1$-morphisms $\rF$ and $\rG$, we have
\begin{equation}\label{repmorcomp}
\eta_{\mathrm{F}\mathrm{G}}=(\mathrm{id}_{\mathbf{N}(\mathrm{F})}\circ_0\eta_{\mathrm{G}})\circ_1
(\eta_{\mathrm{F}}\circ_0\mathrm{id}_{\mathbf{M}(\mathrm{G})}).
\end{equation}
\end{itemize}
If all the $\eta^\Psi_{\rF}$ are identities, the morphism $\Psi$ is called \emph{strict}.

Here naturality of $\eta_\rF$ means that for any $\rG\in\cC(\ti,\tj)$ and any $\alpha:\rF\to \rG$ we have 
\begin{align}\label{2repmornatural}
\eta_{\rG}\circ_1 (\id_{\Psi_\tj}\circ_0 \bfM(\alpha))=
(\bfN(\alpha)\circ_0 \id_{\Psi_\ti})\circ_1\eta_{\rF},
\end{align}
or, in other words, that in 
$$
\xymatrix@R=3mm{ 
\bfM(\ti)\ar[rrr]^{\bfM(\rF)}\ar[ddd]_{\Psi_\ti}&&&
\bfM(\tj)\ar[ddd]^{\Psi_\tj}\\&&&\\&\ar@{<:}[ru]^{\eta_{\rF}}&\\
\bfN(\ti)\ar[rrr]^{\bfN(\rF)}
&&&\bfN(\tj)\\
}
\quad
\xymatrix{ 
\Psi_\tj\circ \bfM(\rF)\ar[rr]^{\eta_{\rF}}
\ar[d]_{\id_{\Psi_{\tj}}\circ_0 \bfM(\alpha)}
&& \bfN(\rF)\circ \Psi_{\ti}
\ar[d]^{\bfN(\alpha)\circ_0 \id_{\Psi_\ti}}\\
\Psi_\tj\circ \bfM(\rG)\ar[rr]^{\eta_{\rG}}
&& \mathbf{N}(\rG)\circ \Psi_\ti
}
$$
the left diagram commutes up to $\eta_{\rF}$ while the right
diagram commutes.

Given two $2$-natural transformations $\Psi$ and $\Phi$ as above, a {\bf modification} $\theta:\Psi\to\Phi$ is a map
which assigns to each $\ti\in\cC$ a morphism of $p$-dg functors $\theta_\ti:\Psi_\ti\to
\Phi_\ti$ such that for any $\rF,\rG\in\cC(\ti,\tj)$ and any 
$\alpha:\rF\to \rG$ we have 
\begin{align}
\eta_\rG^{\Phi}\circ_1 (\theta_\tj\circ_0 \bfM(\alpha))=
(\bfN(\alpha)\circ_0 \theta_\ti)\circ_1\eta_\rF^{\Psi}, \label{modificationcond}
\end{align}i.e. the diagram
$$
\xymatrix@R=3mm{
 \Psi_\tj \circ \bfM(\rF)\ar[rr]^{\eta_{\rF}^\Psi}\ar[dd]_{\theta_\tj \circ_0 \bfM(\alpha)}&&\bfN(\rF)\circ \Psi_\ti \ar[dd]^{\bfN(\alpha)\circ_0 \theta_\ti}\\
 &&\\
\Phi_\tj \circ \bfM(\rG)\ar[rr]^{\eta_{\rG}^{\Phi}}&&\bfN(\rG)\circ \Phi_\ti
}
$$
commutes.

\begin{proposition}\label{pamod}
Let $\cC$ be a locally finite $2$-category.
Additive $p$-dg $2$-representations of $\cC$ together with $2$-natural transformations and modifications form a $p$-dg $2$-category.
\end{proposition} 

\proof
The fact that these form a $2$-category follows as in \cite[Proposition 1]{MM3}. It thus only remains to show that composition of $2$-natural transformations is indeed a $p$-dg functor.
For this, we need to verify that the differential of a $2$-natural transformation, which is defined as $(\partial \theta)_i:=\partial(\theta_i)$, again satisfies the naturality condition (\ref{modificationcond}). In fact,
\begin{align*}
&\eta_\rG^{\Phi}\circ_1 (\partial(\theta_\tj)\circ_0 \bfM(\alpha))\\
&=\partial(\eta_\rG^{\Phi}\circ_1 (\theta_\tj\circ_0 \bfM(\alpha)))-\partial(\eta_\rG^{\Phi})\circ_1 (\theta_\tj\circ_0 \bfM(\alpha))-\eta_\rG^{\Phi}\circ_1 (\theta_\tj\circ_0 \partial(\bfM(\alpha)))\\
&=\partial((\bfN(\alpha)\circ_0 \theta_\ti)\circ_1\eta_\rF^{\Psi})-\eta_\rG^{\Phi}\circ_1 (\theta_\tj\circ_0 \bfM(\partial\alpha))\\
&=\partial((\bfN(\alpha)\circ_0 \theta_\ti)\circ_1\eta_\rF^{\Psi})-(\bfN(\partial\alpha)\circ_0 \theta_\ti)\circ_1\eta_\rF^{\Psi}\\
&=(\bfN(\alpha)\circ_0 \partial(\theta_\ti))\circ_1\eta_\rF^{\Psi}.
\end{align*}
Here, we use that horizontal composition $\circ_0$ and vertical composition $\circ_1$ of natural transformations are morphisms of $H$-modules, i.e. satisfy a Leibniz rule. Further, $\bfM, \bfN$ are $p$-dg functors and hence commute with $\partial$, and $\eta_G^\Phi$, $\eta_F^\Psi$ are $p$-dg isomorphisms and hence annihilated by $\partial$.
\endproof

We denote the resulting $2$-category by $\cC\pamod$ and the full $2$-subcategories consisting of strongly finitary $2$-representations by 
$\cC\pamod^{\mathrm{sf}}$. The following Lemma will help simplify subsequent proofs.

\begin{lemma}\label{wloglemma}
Let $\bfM$ be an additive $p$-dg $2$-representation over $\cC$. Then $\bfM$ is $p$-dg equivalent to a $p$-dg $2$-representation $\overline{\bfM}$ where $\overline{\bfM}(\ti)=\overline{\bfM(\ti)}$ for any object $\ti$ of $\cC$.
\end{lemma}
\begin{proof}
Given $\bfM$, we define $\overline{\bfM}(\ti)=\overline{\bfM(\ti)}$, and for a $1$-morphism $\rF$, we let $\overline{\bfM}(\rF)$ be the induced $p$-dg functor $\overline{\bfM(\rF)}$ from Lemma \ref{extensionlemma}\eqref{extensionlemma1}. Using Lemma \ref{extensionlemma}\eqref{extensionlemma2} we can induce natural isomorphisms $\overline{\bfM}(\alpha)$, given a $2$-morphism $\alpha$. It was shown in the same lemma that taking $\overline{(-)}$ is $2$-functorial. Hence $\overline{\bfM}$  gives a $p$-dg $2$-functor, and therefore an object in $\cC\pamod$. 

Consider the natural transformations $\iota_{\ti}\colon \bfM(\ti)\to \overline{\bfM(\ti)}$. Using Lemma \ref{extensionlemma}\eqref{extensionlemma3}, we see that $\iota_{\ti}$ is part of a $p$-dg equivalence for each $\ti$. This uses that $\bfM(\ti)$ is $p$-dg equivalent to $\overline{\C_\ti}$ for locally finite $\C_\ti$. It remains to show that the $\iota_{\ti}$ carry the structure of a morphism of $2$-representations.
However, given a $1$-morphism $\rF\in \cC(\ti,\tj)$, we can chose 
$\eta_\rF\colon \iota_{\ti}\circ \bfM(\rF)\to \overline{\bfM(\rF)}\circ \iota_{\ti}$ to be the identity by Lemma \ref{extensionlemma}\eqref{extensionlemma1}. It is then clear that \eqref{repmorcomp} is satisfied.
\end{proof}

An analogue of the Yoneda lemma exists in this setting:

\begin{proposition}\label{yoneda}
Let $\cC$ be a locally finite $2$-category, $\ti \in \cC$ and $\bfM\in \cC\pamod$. Then there is a $p$-dg equivalence
\begin{equation*}
\rY_\ti\colon \bfM(\ti)\stackrel{\sim}{\longrightarrow}\Hom_{\ccC\pamod}(\bfP_\ti,\bfM).
\end{equation*}
\end{proposition}
\begin{proof} 
Assume given any object $M \in \bfM(\ti)$.
For $\tj\in \cC$ define the $p$-dg functor $\Phi^M_{\tj}:
\overline{\cC(\ti,\tj)}\to \bfM(\tj)$ by applying Lemma \ref{extensionlemma} to the $p$-dg functor $\cC(\ti,\tj)\to \bfM(j)$ that maps $\rF$ to $\bfM(\rF)(M)$ and a morphism $\alpha\colon \rF\to \rG$ in $\cC(\ti,\tj)$ to $\bfM(\alpha)_M$. Referring to Lemma \ref{wloglemma}, we can (without loss of generality) assume that $\bfM(\ti)$ is of the form $\overline{\M_\tj}$ for some locally finite $p$-dg category $\M_j$, and as $\overline{\overline{\M_\tj}}$ is $p$-dg equivalent to $\overline{\M_\tj}$ we get a $p$-dg functor $\Phi^M_\tj$ as required. This functor maps the object $X$ in $\overline{\cC(\ti,\tj)}$ given by the pair $(\bigoplus_{m=1}^s \rF_m, (\alpha_{k,l})_{k,l})$ with $\alpha_{k,l}\colon \rF_l\to \rF_k \in \cC(\ti, \tj)$ to the object in $\bfM(\tj)$ given by the pair  $\left(\bigoplus_{m=1}^s \bfM(\rF_m)(M), (\bfM(\alpha_{kl})_M)_{k,l}\right)$ in $\bfM(\tj)$.

Consider a $1$-morphism $\rG\in \cC(\tj,\tk)$. We require a natural $p$-dg isomorphism $\eta_\rG \colon \Phi^M_\tk\circ \bfP_\ti(\rG)\to \bfM(\rG)\circ \Phi_{\tj}^M$. Given an object $\rF$ of $\cC(\ti,\tj)$, to define this $p$-dg isomorphism,
$$(\eta_{\rG})_{\rF} \colon \bfM(\rG\rF)(M)\to  \bfM(\rG)(\bfM(\rF)(M))$$ can be taken to be the identity as $\bfM$ is a strict $2$-functor. We now use Lemma \ref{extensionlemma}(ii) in order to obtain a natural transformation $\eta_{\rG}$ as required. Compatibility (\ref{repmorcomp}) with respect to composition of $1$-morphisms is clear for an object $\rF$ as it amounts to composition of identities.
For more a general object $X\in \overline{\cC(\ti,\tj)}$, 
$(\eta_{\rK\rG})_X$ is the diagonal matrix with entries $(\eta_{\rK\rG})_{\rF_m}$, and the same compatibility holds. 

Hence we have constructed a morphism of $2$-representations $\Phi^M\colon \bfP_\ti \to \bfM$. We can make this construction functorial. Given a morphism $\tau \colon M \to N$ in $\bfM(\ti)$, we construct a modification $\theta^\tau \colon \Phi^M\to \Phi^N$. Again using Lemma \ref{extensionlemma}, we can define $\theta^\tau_\tj\colon \Phi^M_\tj \to \Phi^N_\tj$ by extending  the natural transformation given for $\rF\in \cC(\ti,\tj)$ by
$$
(\theta^\tau_\tj)_{\rF}:=\bfM(\rF)(\tau)\colon \left(\Phi_\tj^M(F)=\bfM(\rF)(M)\right)\longrightarrow \left(\Phi_\tj^N(F)=\bfM(\rF)(N)\right)
$$
to all of $\overline{\cC(\ti,\tj)}$.
Next, consider the compatibility condition (\ref{modificationcond}) for $\alpha \colon \rG\to \rK$ in $\cC(\tj,\tk)$, which is a diagram of natural transformations between functor $\cC(\ti,\tj)\to \bfM(\tk)$. Evaluated at $\rF\in \cC(\ti,\tj)$, this now amounts to 
\begin{align*}\left((\bfM(\alpha)\circ_0 \theta_\tj^\tau)\circ_{1}\eta_{\rG}\right)_{\rF}
=&\bfM(\rG\rF)(\tau\circ_0 \alpha)\\=&\bfM(\tau)\circ_0\bfM(\rF)(\alpha)=\left(\eta_\rK\circ_1 \left(\theta^\tau_\tk\circ_0 \bfP_{\ti}(\alpha)\right) \right)_{\rF}.
\end{align*}
It is clear that this condition will still hold for more general objects and morphisms in $\overline{\cC(\ti,\tj)}$ reasoning as above. Functoriality is easy to check when looking at objects $\rF\in \cC(\ti,\tj)$, and follows for the extensions of the $\Phi_\tj^\tau$ to all of $\overline{\cC(\ti,\tj)}$ from Lemma \ref{extensionlemma}.

This way we obtain a $p$-dg functor $\rY_\ti\colon \bfM(\ti)\to \Hom_{\ccC\pamod}(\bfP_\ti,\bfM)$, with $\rY_\ti(M)=\Phi^M$ and $\rY_\ti(\tau)=\Phi^\tau$. Indeed, $\partial(\Phi_\tj^\tau)$ is defined as the differential of natural transformations $\partial((\Phi_\tj^\tau)_X)$, for all $\tj$. But this morphism is a diagonal matrix with diagonal entries $(\Phi_\tj^\tau)_{\rF_{m}}$ for $X=(\bigoplus_{m=1}^s{\rF_m},\alpha)$, and hence its differential in $\overline{\cC(\ti,\tj)}$ is given by 
$$(\delta_{k,l}\del(\bfM(\rF_k)(\tau)))_{k,l}=(\delta_{k,l}\bfM(\rF_k)(\del\tau))_{k,l},
$$
as the diagonal matrix commutes with the upper triangular matrix $\alpha$ of $X$.

It remains to show that $\rY_\ti$ is a part of a $p$-dg equivalence.
For any morphism $\Psi\colon \bfP_\ti \to \bfM$, we consider the image $M^{\Psi}$ of $\one_\ti$ under $\Psi_\ti\colon \overline{\cC(\ti,\ti)}\to \bfM(\ti)$, which is an object in $\bfM(\ti)$; and for any modification $\theta\colon \Psi\to \Upsilon$, $(\theta_\ti)_{\one_\ti}\colon \Psi_\ti(\one_\ti)\to \Upsilon_\ti(\one_\ti)$ gives a morphism $\tau^\theta$ in $\bfM(\ti)$ between the respective objects. It is clear that $\Phi^M_\ti(\one_\ti)=M$ and $\theta^\tau(\one_\ti)=\tau$ by construction. This shows that the $p$-dg functor $\rY_\ti$ is full. 

Assume we are given a modification $\theta\colon \Psi\to \Upsilon$ and $\rF\in \cC(\ti,\tj)$. We can use the $p$-dg isomorphisms $(\eta^\Psi_{\rF})_{\one_\ti} \colon \Psi_{\tj}(\rF)\to \bfM(\rF)(\Psi_\ti(\one_\ti))$, and $(\eta^\Upsilon_\rF)_{\one_\ti}$. Under these $p$-dg isomorphisms, $(\theta_{\tj})_{\rF}$ corresponds to the isomorphism 
$$(\theta^{(\tau^\theta)})_{\rF}= \bfM(\rF)(\tau^\theta)=\bfM(\rF)((\theta_\ti)_{\one_\ti}),$$ showing faithfulness. We also see that $(\eta^\Psi(-))_{\one_{\ti}}$ provides a $p$-dg isomorphism $\Psi_\tj \cong \Phi_{\tj}^{M^{\Psi}}$, and hence a $p$-dg isomorphism of morphisms of $p$-dg representations $\Psi\cong \Phi^{M^\Psi}$ showing $p$-density of $\rY_\ti$.  
\end{proof}

\subsection{Closure under \texorpdfstring{$p$}{p}-dg quotients}

Given a $2$-representation $\bfM\in \cC\pamod$ of a locally finite $2$-category $\cC$, 
we can define the {\bf $p$-dg quotient completed $2$-representation} $\vv{\bfM}$ by setting 
$\vv{\bfM}(\ti):=\vv{\bfM(\ti)}$ and defining the action of $1$- and $2$-morphisms componentwise (recall the definition of $\vv{\rF}$ and $\vv{\alpha}$ from \ref{pdgquots}). 

In analogy to $\cC\pamod$, we can define a new $p$-dg $2$-category of $p$-dg quotient completed $2$-representations which we denote by $\cC\pcmod$. It consists of representations $\bfA$ which send
\begin{itemize}
\item an object $\ti\in \cC$ to a $p$-dg category $\bfA(\ti)$ $p$-dg equivalent to  $\vv{\C}_\ti$ for a locally finite $p$-dg category $\C_\ti$,
\item a $1$-morphism $\rF \in \cC(\ti,\tj)$ to a $p$-dg functor $\bfA(\ti)\to\bfA(\tj)$ which preserves semi-free objects, i.e. objects $p$-dg isomorphic to those in $\overline{\C}\subset\vv{\C}$,
\item a $2$-morphism $\alpha: \rF\to\rG \in \cC(\ti,\tj)$ to a morphism of $p$-dg functors.
\end{itemize}
Now $\cC\pcmod$ is the $p$-dg $2$-category of such $p$-dg $2$-representations together with morphisms of $p$-dg representations, and modifications as defined in Section \ref{2rep2cat}. 

We obtain a $2$-faithful $p$-dg $2$-functor
\begin{align*}
\vv{(~)}\colon \cC\pamod & \longrightarrow\cC\pcmod,&
\bfM&\longmapsto \vv{\bfM}.
\end{align*}

In the other direction, note that, by definition, the $p$-dg subcategory of semi-free objects in $\coprod_{\ti\in\cC} \bfN(\ti)$ is stable under the action of $\cC$. Then we can start with a $2$-representation $\bfN$ in $\cC\pcmod$ and restrict to this subcategory, giving a $2$-representation $\bfN_{\mathrm{csf}}$ in $\cC\pamod$. It is easy to see that, for $\bfM\in \cC\pamod$, the $2$-representation $(\vv{\bfM})_{\mathrm{csf}}$ is $p$-dg equivalent to $\bfM$ and, for $\bfN\in \cC\pcmod$, the $2$-representation $\vv{\bfN_{\mathrm{csf}}}$ is $p$-dg equivalent to $\bfN$.
Note that we do \emph{not} require that morphisms of $2$-representations in $\cC\pcmod$ preserve semi-free objects and hence this bijection between objects does not give rise to any form of equivalence between $\cC\pamod$ and $\cC\pcmod$.

We will need an analogue of the Yoneda Lemma for $p$-dg quotient complete $2$-representations:

\begin{proposition}\label{quotyoneda}
Let $\cC$ be a locally finite $2$-category, $\ti \in \cC$ and $\bfN\in \cC\pamod$. Then there is a $p$-dg equivalence
\begin{equation*}
\vv{\rY}_\ti\colon \vv{\bfN}(\ti)\stackrel{\sim}{\longrightarrow}\Hom_{\ccC\pcmod}(\vv{\bfP_\ti},\vv{\bfN}),
\end{equation*}
extending the $p$-dg equivalence $\rY_\ti$ from Proposition \ref{yoneda}.
\end{proposition}

\proof
We assume, without loss of generality, that $\bfN(\ti)=\overline{\A_{\ti}}$ for some locally finite $p$-dg category $\A_\ti$, cf. Lemma \ref{wloglemma}.

First, assume given an object $X\overset{f}{\to}Y$ in $\vv{\bfN(\ti)}$. We construct a morphism $\phi^f\colon \vv{\bfP_\ti}\to \vv{\bfN}$ in $\cC\pcmod$ such that $\phi_\tj^f(0\to\one_{\ti})=X\overset{f}{\to}Y$. For this, we use the morphisms $\Phi^X$, $\Phi^Y$ and the modification $\theta^f$ defined in Proposition \ref{yoneda}. For an object $\rF\overset{\gamma}{\to}\rG$ in $\vv{\cC(\ti,\tj)}$ we set
$$
\Phi_\tj^f(\rF\overset{\gamma}{\to}\rG)=\Phi_\tj^Y(\rF)\oplus \Phi_\tj^X(\rG)\xrightarrow{\left(\begin{smallmatrix}\Phi_\tj^Y(\gamma),&\left(\theta^f_\tj\right)_{\rG}\end{smallmatrix}\right)}\Phi_\tj^Y(\rG).
$$
We assign to a morphism, given by a diagram
$$\begin{array}{ccc}
\vcenter{\hbox{\xymatrix{\rF\ar_{\varphi_0}[d]\ar^{\gamma}[r]&\rG\ar^{\varphi_1}[d]\\
\rF'\ar^{\gamma'}[r]&\rG'
}}}&\in&\vv{\cC(\ti,\tj)}
\end{array},
$$
the diagram 
\begin{equation}\label{Phiarrow}
\vcenter{\hbox{\xymatrix{\Phi_\tj^Y(\rF)\oplus \Phi_\tj^X(\rG)\ar_-{\left(\begin{smallmatrix}\Phi_\tj^Y(\varphi_0)&0\\ 0&\Phi_\tj^X(\varphi_1) \end{smallmatrix}\right)}[dd]  \ar^-{\left(\begin{smallmatrix}\Phi_\tj^Y(\gamma),&\left(\theta^f_\tj\right)_{\rG}\end{smallmatrix}\right)}[rrrr]&&&&\Phi_\tj^Y(\rG)\ar^-{\left(\begin{smallmatrix}\Phi_\tj^Y(\varphi_1)\end{smallmatrix}\right)}[dd]\\ \\
\Phi_\tj^Y(\rF')\oplus \Phi_\tj^X(\rG')\ar^-{\left(\begin{smallmatrix}\Phi_\tj^Y(\gamma'),&\left(\theta^f_\tj\right)_{\rG'}\end{smallmatrix}\right)}[rrrr]&&&&\Phi_\tj^Y(\rG')
}}}
\end{equation}
in $\bfN(\tj)$.
First of all, it is clear that this gives a morphism in $\vv{\bfN(\tj)}$ as $\del(\Phi_\tj^Y(\gamma))=\Phi_\tj^Y(\del\gamma)=0$, $\del\theta^f=\theta^{\del f}=0$; and the diagram commutes using functoriality of $\Phi^Y_\tj$ applied to $\varphi_1\gamma=\gamma'\varphi_0$ in the first $\oplus$-component, and naturality of $\theta_\tj^f$ in the second $\oplus$-component. It is easy to see that the assignment is functorial, using functoriality of the functors $\Phi^X_\tj$ and $\Phi^Y_\tj$. As these functors are indeed $p$-dg functors, $\Phi^f_\tj$ is also a $p$-dg functor. Finally, given a homotopy $h\colon \rG\to \rF'$ rendering the morphism in $\vv{\cC(\ti,\tj)}$ given by the same diagram zero, we can use the morphism
$$\begin{pmatrix}
\Phi_\tj^Y(h)\\0
\end{pmatrix}\colon \Phi_\tj^Y(\rG)\to \Phi_\tj^Y(\rF')\oplus \Phi_\tj^X(\rG')$$
to factor $\begin{pmatrix}\Phi_\tj^Y(\varphi_1)\end{pmatrix}$  through the bottom left corner object. This shows that $\Phi^f_\tj$ gives a well-defined $p$-dg functor from $\vv{\cC(\ti,\tj)}$ to $\vv{\bfN(\tj)}$.

Next, we show that $\Phi^f=\lbrace \Phi^f_\tj\rbrace_{\tj\in\cC}$ can be given the structure of a morphism in $\cC\pcmod$. For this, let $\rH\in \overline{\cC(\tj,\tk)}$. We need to construct a natural transformation
$$\vv{\eta}_\rH\colon \Phi_\tk^f\circ\vv{\bfP_\ti}(\rH)\to \vv{\bfN}(\rH)\circ\Phi_\tj^f.$$
For an object $\rF\overset{\gamma}{\to}\rG$, we define $\left(\vv{\eta}_\rH\right)_{\rF\overset{\gamma}{\to}\rG}$ as the diagram
\begin{equation}
\vcenter{\hbox{\xymatrix{
\Phi_\tk^Y(\rH\rF)\oplus \Phi_\tk^X(\rH\rG)\ar^-{\left(\begin{smallmatrix}\left(\eta^{\Phi^Y}_{\rH}\right)_{\rF}&0\\ 0&\left(\eta^{\Phi^X}_{\rH}\right)_{\rG}\end{smallmatrix}\right)}[dd]  \ar^-{\left(\begin{smallmatrix}\Phi_\tk^Y(\rH\gamma),&\left(\theta^f_\tk\right)_{\rH\rG}\end{smallmatrix}\right)}[rrrr]&&&&\Phi_\tk^Y(\rH\rG)\ar_-{\left(\begin{smallmatrix}\left(\eta^{\Phi^Y}_{\rH}\right)_{\rG}\end{smallmatrix}\right)}[dd]\\ \\
\bfN(\rH)\left(\Phi_\tj^Y(\rF)\right)\oplus \bfN(\rH)\left(\Phi_\tj^X(\rG)\right)\ar^-{\left(\begin{smallmatrix}\bfN(\rH)\Phi_\tj^Y(\gamma),&\bfN(\rH)\left( \left(\theta^f_\tj\right)_{\rG}\right)\end{smallmatrix}\right)}[rrrr]&&&&\bfN(\rH)\Phi_\tj^Y(\rG).
}}}\label{etadiag}
\end{equation}
Note that the bottom line object equals
$$\vv{\bfN}(\rH)\left(\Phi_\tj^f(\rF\overset{\gamma}{\to}\rG)\right)=\vv{\bfN(\rH)}\left(\Phi_\tj^Y(\rF)\oplus \Phi_\tj^X(\rG)\xrightarrow{\left(\begin{smallmatrix}\Phi_\tj^Y(\gamma),&\left(\theta^f_\tj\right)_{\rG}\end{smallmatrix}\right)}\Phi_\tj^Y(\rG)\right).$$
We remark that here we crucially use that the induced functor $\overline{\bfN(\rH)}$ is given by componentwise application of $\bfN(\rH)$ since the symbol $\oplus$ in $\overline{\A_\tj}$ denotes a list of objects (not the direct sum in $\A_\tj$).

The diagram \eqref{etadiag} commutes using in the first $\oplus$-component that $\eta^{\Phi^Y}$ is part of the data of a morphism in $\cC\pamod$, where in the second $\oplus$-summand we use that $\theta^f$ is a modification. The horizontal morphisms differentiate to zero since all functors applied are $p$-dg functors and $\del(\theta^f)=0$ as above. Finally, the vertical arrows are clearly $p$-dg isomorphisms inheriting this property from their components using Proposition \ref{yoneda}.

Next, we need to show coherence of $\vv{\eta}$. For a given morphism $\beta\colon \rH\to \rH'\in \overline{\cC(\tj,\tk)}$ we require that the diagram
$$
\xymatrix{\Phi_\tk^f\circ \vv{\bfP_\ti(\rH)}\ar_-{\id_{\Phi_\tk^f}\circ_0\vv{\bfP_{\ti}(\beta)}}[d]\ar^-{\vv{\eta}_{\rH}}[rr]&&\vv{\bfN(\rH)}\circ \Phi_\tj^f\ar^-{\vv{\bfN(\beta)}\circ_0\id_{\Phi_{\tj}^f}}[d]\\
\Phi_\tk^f\circ \vv{\bfP_\ti(\rH')}\ar^-{\vv{\eta}_{\rH'}}[rr]&&\vv{\bfN(\rH')}\circ \Phi_\tj^f
}
$$
of natural transformations commutes strictly. Evaluating at an object $\rF\xrightarrow{\gamma}\rG$, this amounts to the equality of the pairs of matrices:

\resizebox{\linewidth}{!}{
  \begin{minipage}{\linewidth}
\begin{align*}
\left(\begin{smallmatrix}
(\bfN(\beta)\circ_0\id)\circ_1\left(\eta_\rH^{\Phi^Y}\right)_{\rF}&0\\
0& (\bfN(\beta)\circ_0\id)\circ_1\left(\eta_\rH^{\Phi^X}\right)_{\rG}
\end{smallmatrix}\right)
&=
\left(\begin{smallmatrix}
\left(\eta_{\rH'}^{\Phi^Y}\right)_{\rF}\circ_1(\id\circ_0\bfP_\ti(\beta))&0\\
0& \left(\eta_{\rH'}^{\Phi^X}\right)_{\rG}\circ_1(\id\circ_0\bfP_\ti(\beta))\end{smallmatrix}\right)\\
(\bfN(\beta)\circ_0\id)\circ_1\left(\eta_\rH^{\Phi^Y}\right)_{\rG}&=\left(\eta_{\rH'}^{\Phi^Y}\right)_{\rG}\circ_1(\id\circ_0\bfP_\ti(\beta)),
\end{align*}
\end{minipage}}

which follow componentwise from the corresponding identities in Proposition \ref{yoneda}.
This shows that $(\Psi^f, \vv{\eta})$ defines a morphism in $\cC\pcmod$.

It remains to show that a morphism $$\begin{array}{ccc}\varphi=\vcenter{\hbox{\xymatrix{X\ar^-{f}[r]\ar_-{\varphi_0}[d]&Y\ar^-{\varphi_1}[d]\\X'\ar^-{f'}[r]&Y'}}}&\in&\vv{\bfN}(\ti)\end{array}$$ induces a modification $\theta^\varphi\colon \theta^f\to \theta^{f'}$. We can define, at an object $\rF\overset{\gamma}{\to}\rG$,
\begin{equation}\label{thetaarrow}
\left(\theta_\tj^\varphi\right)_{\rF\overset{\gamma}{\to}\rG}=\left(\left(\begin{smallmatrix}\left(\theta_\tj^{\varphi_1}\right)_{\rF}&0\\0&\left(\theta_\tj^{\varphi_0}\right)_{\rG}\end{smallmatrix}\right), \left(\theta_\tj^{\varphi_1}\right)_{\rG}\right),
\end{equation}
which is a morphism in $\vv{\bfN(\tj)}$ as required, using that $\theta^{\varphi_0}$ is a natural transformation from Proposition \ref{yoneda} in the first $\oplus$-component, and that the assignment $\rY_\ti$ is functorial, applied to the diagram defining $\varphi$, in the second $\oplus$-component.

We have to check that $\theta_\ti^\varphi$ is natural. For this, we observe that both $\theta^\varphi_\ti$, and $\Phi_\tj^f(\gamma)$ defined in \eqref{Phiarrow} are given by pairs of block diagonal matrices using the corresponding constructions in Proposition \ref{yoneda} as entries. Hence naturality follows from naturality there (combined with functoriality of $\rY_\ti$). The same reasoning works to verify condition \eqref{modificationcond} using that the morphisms $\vv{\eta}$ are also defined using block diagonal matrices of the corresponding construct in the Yoneda Lemma for $\cC\pamod$, as displayed in \eqref{etadiag}.

From the definition of $\theta^\varphi$ in \eqref{thetaarrow} we see that the assignment $\varphi\mapsto \theta^\varphi$ is functorial, using functoriality of $\rY_\ti$ in \ref{yoneda}. It is also clear with this description that a homotopy $h\colon Y\to X'$ induces a homotopy $\begin{pmatrix}
0\\\left(\theta_\tj^{h}\right)_{\rG}
\end{pmatrix}$.
That the assignment commutes with $\del$ is seen in the same way, giving a $p$-dg functor $\vv{\rY_\ti}$ as desired.

Conversely, given a morphism $\Phi$ and a modification $\theta$, we send them to their evaluations at $\one_\ti$ as usual. It can be verified as in Proposition \ref{yoneda} that this defines a p-dg functor providing a quasi-inverse to $\vv{\rY_\ti}$ as desired.
\endproof

\subsection{Derived \texorpdfstring{$p$}{p}-dg \texorpdfstring{$2$}{2}-representations}

For this subsection, assume given a $2$-repre\-sentation $\bfM\in \cC\pamod$ of a locally finite $2$-category $\cC$. Recall that for every object $\ti$, $\bfM(\ti)$ is, up to $p$-dg equivalence, of the form $\overline{\C_{\ti}}$, where $\C_\ti$ is locally finite.
Using the Lemma \ref{derivedfunctor}, we are hence able to derive $p$-dg $2$-representations by passing to stable categories as in Section \ref{derivedcats}:

\begin{definition}\label{stablerep}
The assignment
\begin{align*}
\ti &\mapsto \St(\bfM(\ti)),\\
(\rF\colon \ti \to \tj) & \mapsto \St(\bfM(\rF))
\end{align*}
extends to a $2$-functor defined on $\cZ\cC$. We denote this $2$-representation by $\bfD\bfM$, the \textbf{derived representation of $\bfM$}.
\end{definition}

The rest of this subsection will discuss on which $2$-categories (instead of $\cZ\cC$) the representation $\bfD\bfM$ can be defined. First, we discuss a maximal quotient $\cD\cC$ on which $\bfD\bfM$ is defined. This construction does not have triangulated $1$-hom categories. To render the latter, we also describe how $\bfD\bfM$ can be defined on $\cSt(\overline{\cC})$.

\begin{lemma}
The set of  all $2$-morphisms $\alpha\colon \rF\to \rF'$ in $\cZ\cC$ which induce the zero natural transformation $\St(\rF)\to \St(\rF')$ in $\bfD\bfM$ forms a $2$-ideal.
\end{lemma}

\begin{proof}
If a morphism $\alpha\colon \rF\to \rF'$ induces the zero natural transformation on the stable categories, then $\bfM(\alpha)_X$ is null-homotopic for all objects $X$ of $\bfM(\ti)$. This condition is closed under left and right composition by other morphisms of functors as null-homotopic morphisms form an ideal. Further, the functors $\bfM(\rG)$ preserve null-homotopic morphisms by Lemma \ref{derivedfunctor}. Therefore, we obtain a $2$-ideal.
\end{proof}

Now we turn our attention to the principal $2$-representations $\bfP_{\ti}$. By Definition \ref{principal}, we see that for a $1$-morphism $\rF\colon \tj\to \tk$, the $p$-dg functor $\bfP_\ti\colon \overline{\cC(\ti,\tj)}\to \overline{\cC(\ti,\tk)}$ is induced by the composition functor $\cC(\ti,\tj)\to \cC(\ti,\tk)$. Hence we obtain an induced $2$-functor defined on $\cZ\cC$, mapping $\tj$ to $\St(\overline{\cC(\ti,\tj)})$.

Similarly to the above lemma, we find that the collection of all $2$-morphisms $\alpha \in \cC(\ti,\tj)$ which induce null-homotopic natural transformations in $\bfP_\ti$ (for the same object $\ti$) forms a $2$-ideal in $\cZ\cC$, which we denote by $\cN$. We further denote by $\cD\cC$ the quotient $\cZ\cC/\cN$, where we can disregard {\bf acyclic $1$-morphisms} (i.e. $1$-morphisms $\rF$ whose identity $\id_{\rF}$ is acyclic). We call $\cD\cC$ the {\bf stable quotient $2$-category of $\cC$}.

\begin{proposition}
Given an additive $p$-dg $2$-representation $\bfM$ as above, the $2$-functor $\bfD\bfM$ from Definition \ref{stablerep} descends to a well-defined $2$-functor on $\cD\cC$.
\end{proposition}

\begin{proof}
Let $\alpha\colon \rF\to\rF'$ be a morphism in $\cC(\ti,\tj)$ such that $\bfP_{\ti}(\alpha)$ is null-homotopic. By specializing, this gives that the morphism $\bfP_{\ti}(\alpha)_{\one_{\ti}}=\alpha$ in $\overline{\cC(\ti,\tj)}$ is null-homotopic. Hence, $\alpha$ factors through an object of the form $X\otimes V$, where $X=\left(\bigoplus_{m=1}^s\rF_m, (\beta_{k,l})_{k,l}\right)\in \overline{\cC(\ti,\tj)}$ and $V$ is a projective $H$-module.
Observe that, using this factorization of $\alpha$, the morphism $\bfM(\alpha)_Y$ factors through the object $\left(\bigoplus_{m=1}^s\bfM(\rF_m)(Y), \left(\bfM(\beta_{k,l})_Y\right)_{k,l}\right)\otimes V$ in $\bfM(\tj)$,  for any object $Y$ of $\bfM(\ti)$. This shows that $\bfM(\alpha)_Y$ is a null-homotopic morphism in $\bfM(\tj)$. Hence $\bfM(\alpha)$ is a null-homotopic  natural transformation. As such, it induces the zero natural transformation when passing to the stable categories.
\end{proof}

We also denote the resulting $2$-functor defined on $\cD\cC$ obtained this way by the same symbol $\bfD\bfM$ as the $2$-functor defined on $\cZ\cC$.  Note that Theorem \ref{triangularfunctor} yields that the images of $\ti$ under $\bfD\bfM$ are triangulated categories, and $1$-morphisms induce triangulated functors.

\begin{proposition}\label{stable2functoriality}
Given a morphism $\Phi\colon\bfM\to \bfN$ in $\cC\pamod$, there exists an induced morphism $\bfD\Phi\colon\bfD\bfM\to \bfD\bfN$. This assignment satisfies 
\begin{align}\bfD\Phi\circ \bfD\Psi=\bfD(\Phi\circ\Psi).\end{align}

Moreover, given a modification $\theta\colon \Psi\to \Phi$ such that $\del\theta=0$, there is an induced modification $\bfD\theta\colon \bfD\Psi\to \bfD\Phi$. For such modifications, $\bfD$ commutes with horizontal and vertical composition.
\end{proposition} 
\begin{proof}
Given $\Phi$, and an object $\ti$ of $\cC$, the $p$-dg functor $\Phi_\ti\colon \bfM(\ti)\to \bfN(\ti)$ preserves null-homotopic morphisms by Lemma \ref{derivedfunctor}. Hence we have an induced functor $\St(\bfM(\ti))\to \St(\bfN(\ti))$. The data of $\Phi$ includes,  for each $1$-morphism $\rF$, a $p$-dg isomorphism $\eta_\rF\colon \Phi_\tj\circ \bfM(\rF)\to \bfN(\rF)\circ\Phi_\ti$, which induces a morphism $\St(\eta_\rF)$ of $p$-dg functors after passing to the stable categories.
Since passing to the stable categories is $2$-functorial, we see that \eqref{repmorcomp} still holds for $\St(\eta)$. Further, \eqref{2repmornatural} holds, in particular, for all morphisms annihilated by the differential. If a $2$-morphism $\alpha$ is null-homotopic, then both $\bfM(\alpha)$ and $\bfN(\alpha)$ are zero on the stable categories, and hence both terms appearing in \eqref{2repmornatural} are zero.

With the above construction of $\bfD\Psi$, it follows directly that $\bfD\Phi\circ \bfD\Psi=\bfD(\Phi\circ\Psi)$ is functorial, using the functoriality in Lemma \ref{derivedfunctor}.

Now let $\theta\colon \Psi\to\Phi$ be a modification such that $\del\theta=0$. For each object $\ti$, there exists an induced morphism of functors $\St(\theta)\colon \St(\Psi_\ti)\to \St(\Phi_\ti)$. This assignment is functorial using Lemma \ref{stabletransformations}. Condition \eqref{modificationcond} still holds after passing to the stable categories, using that horizontal composition preserves null-homotopic morphisms.
\end{proof}

\begin{definition}Let $\cC$ be a locally finite $p$-dg $2$-category.
The $2$-category of all stable $2$-representations $\bfD\bfM$ for $\bfM\in \cC\pamod$ is denoted by $\cC\dmod$, with morphisms and modifications induced from $\cC\pamod$ via Proposition \ref{stable2functoriality}.
\end{definition}

Note that the assignments in Proposition \ref{stable2functoriality} provide a $2$-functor
$$\bfD\colon \cZ(\cC\pamod)\to \cC\dmod.$$

\begin{example}
We will start with a small example. Let $\cC$ be the $p$-dg $2$-category with one object $\star$ corresponding to $D\cof$, where $D=\Bbbk[x]/(x^p)$, and $1$-morphisms generated (under multiplication, $p$-dg direct sums and shifts) by the $\Bbbk$-indecomposable objects $\one=D\otimes_D(-)$ and $\rF=F\otimes_D(-)$ for the regular projective bimodule $F=D\otimes_\Bbbk D$. The $2$-morphisms correspond to morphisms of $D$-$D$-bimodules. (We will treat such $p$-dg $2$-categories in more generality in Section \ref{CAsec}.)

The principal $2$-representation $\bfP=\bfP_\star$ is fully faithful so that we have $p$-dg isomorphisms $\Hom_{\bfP(\star)}(\bfP(\rX),\bfP(\rX'))\cong \Hom_{D\otimes D^{\op}}(X,X')$. Hence, we can describe the null-homotopic morphisms in $\cN$ between the $\Bbbk$- indecomposables $\one$ and $\rF$ explicitly. 

In the case $p=3$, after passing to $\cD$, the only remaining $2$-morphisms  of indecomposables (up to degree shifts) are
\begin{align*}
1,~ x^2\colon D\to D, && 1,~x^2\colon F\to D, &&11, ~x^21,~1x^2,~ x^2x-xx^2\colon F\to F.
\end{align*}
The path category of $\cD\C(\star,\star)$ is hence fully described by
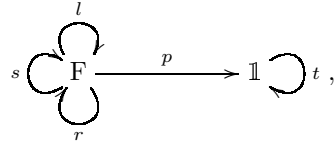
\begin{gather*}
\xymatrix{\rF\ar[rr]^-{p}\ar@(ul,ur)^{l}\ar@(dr,dl)^{r}\ar@(dl,ul)^{s} &&\one\ar@(ur,dr)^-{t}},\\
pl=pr=tp,\quad ps=lr=rl=sl=ls=rs=sr=l^2=r^2=s^2=t^2=0.
\end{gather*}
Here, the morphisms $r,l,t$ are of degree two, and $s$ is of degree three.
\end{example}

Observe that $\cD\cC(\ti,\tj)$ is no longer a $p$-dg category.
Note further that $\cD\cC(\ti,\tj)$ is not necessarily triangulated as cone objects may not exist.
To remedy this, recalling the definition of the stable $2$-category $\cSt\overline{\cC}$ associated to $\cC$ from Section \ref{stable2cat}, we can obtain derived representations in an alternative way: We first complete a $p$-dg $2$-representation of $\cC$ to one of $\overline{\cC}$ and then apply the above constructions, giving a $2$-representation of $\cSt\overline{\cC}$.

\begin{proposition}\label{overline2reps}
Let $\cC$ be a locally finite $2$-category and $\bfM$ an additive $p$-dg $2$-representation over $\cC$. Then $\bfM$ extends to an additive $p$-dg $2$-representation $\overline{\bfM}$ over $\overline{\cC}$. This way, we obtain a $2$-faithful $p$-dg $2$-functor
\begin{align*}
\bfInd \colon \cC\pamod&\longrightarrow \overline{\cC}\pamod, &\bfM\longmapsto\overline{\bfM}.
\end{align*}
\end{proposition}
\begin{proof}
Without loss of generality, assume that $\bfM(\ti)=\overline{\C_{\ti}}$ for any object $\ti$ of $\cC$, cf. Lemma \ref{wloglemma}.

Assume given an additive $p$-dg representation $\bfM$ over $\cC$. We want to construct a $p$-dg $2$-functor $\overline{\bfM}\colon \overline{\cC}\to \cofcat$. On objects, we assign $\overline{\bfM}(\ti)=\bfM(\ti)$. Given an object $\rX=\left( \bigoplus_{m}\rF_m,\alpha\right)\in \overline{\cC}(\ti,\tj)=\overline{\cC(\ti,\tj)}$ we define $\overline{\bfM}(\rX)$ to be the $p$-dg functor $\bfM(\ti)\to\bfM(\tj)$ extended from \begin{align*}M\longmapsto \left( \bigoplus_m\bfM(\rF_m)(M),\bfM(\alpha)\right), &&M\in \C_\ti,\end{align*}
using Lemma \ref{extensionlemma}\eqref{extensionlemma1}.
Explicitly, given an object $Y=\left(\bigoplus_{n}M_n,\beta\right)$ in $\overline{\C_{\ti}}$, we have that $\overline{\bfM}(\rX)(Y)$ equals the object
\begin{align}
\left(\bigoplus_{(m,n)}\bfM\rF_m(M_n), (\delta_{m,m'}\bfM\rF_m(\beta_{n,n'})+\delta_{n,n'}\bfM(\alpha_{m,m'})_{M_n})_{(m,n),(m',n')}\right),
\end{align}
where we fix the convention to order pairs $(m,n)$ lexicographically.
Given a morphism $\tau\colon Y\to Y'$ in $\bfM(\ti)$, we obtain
\begin{align}
\overline{\bfM}(\rX)(\tau)=\left(\delta_{m,m'}\bfM(\rF_m)(\tau_{n,n'})\right)_{(m,n),(m',n')},
\end{align}
which is a block diagonal matrix.

To a morphism $\gamma\colon \rX\to \rX'$ in $\overline{\cC(\ti,\tj)}$, we associate the matrix
\begin{align*}
\overline{\bfM}(\gamma)_Y=\left(\delta_{n,n'}\bfM(\gamma_{m,m'})_{M_n}\right)_{(m,n),(m',n')}.
\end{align*}
Using Lemma \ref{extensionlemma}\eqref{extensionlemma2} we find that this gives a natural transformation $\overline{\bfM}(\gamma)\colon \overline{\bfM}(\rX)\to \overline{\bfM}(\rX')$. Moreover, the assignment $\gamma\mapsto\overline{\bfM}(\gamma)$ commutes with the differential. Indeed, first, restricting to $\C_\ti$, the natural transformations are just given by $\bfM(\gamma)_M$, for an object $M$. Since $\bfM(\gamma_{m,m'})_M$ is natural in $M$ for any $m,m'$, this gives a natural transformation, and the assignment commutes with the differential. Now Lemma \ref{extensionlemma}\eqref{extensionlemma2} implies that $\bfM(\gamma)_Y$ is also natural for a general object $Y$ in $\bfM(\ti)$.

Next, we have to show $2$-functoriality of $\overline{\bfM}$. Given $\rX$ in $\overline{\cC}(\ti,\tj)$ and $\rX'$ in $\overline{\cC}(\tj,\tk)$, we compare $\overline{\bfM}(\rX')\overline{\bfM}(\rX)$ and $\overline{\bfM}(\rX'\rX)$. It suffices to evaluate these compositions on an object $M$ of $\C_\ti$. We compare the lists of objects
\begin{align*}
\overline{\bfM}(\rX')\left(\overline{\bfM}(\rX)(M)\right)&=\bigoplus_{m'}\bfM(\rF_{m'})\left(\bigoplus_{m}\bfM(\rF_m)(M)\right)\\
&=\bigoplus_{(m',m)}\bfM(\rF_{m'})\left(\bfM(\rF_m)(M)\right)\\
&=\bigoplus_{(m',m)}\bfM(\rF_{m'}\rF_m)(M)\\
&=\overline{\bfM}(\rX'\rX)(M),
\end{align*} 
using that $\bfM$ is a $2$-functor and the convention about ordering pairs lexicographically. Further, we compare the upper triangular matrices giving the differential. We see that the differential on $\overline{\bfM}(\rX)(M)$ is the matrix $\bfM(\alpha_{m,m'})$, for $\alpha$ being the upper triangular matrix of $\rX$ encoding the differential. Hence we compute that the entry at the pair $(n,m),(n',m')$ of the differential on $\overline{\bfM}(\rX')\left(\overline{\bfM}(\rX)(M)\right)$ is given by
\begin{align*}
\left(\delta_{n,n'}\bfM(\rF'_n)(\bfM(\alpha_{m,m'}))+\delta_{m,m'}\bfM(\alpha'_{n,n'})\right)_{\bfM(\rF_m)(M)},
\end{align*}
where $\alpha'$ is the differential matrix of $\rX'$. On the other hand, considering the differential matrix for $\bfM(\rF'\rF)(M)$ we see that the $(n,m),(n',m')$-entry is
$\bfM(\beta_{(n,m),(n',m')})_M$ where $\beta$ denotes the differential on $\rX'\rX$ as defined in \eqref{orderingconvention} in Section \ref{stable2cat}. This gives 
\begin{align*}
\bfM ( \beta_{(n,m),(n',m')})_M&=\bfM\left(\delta_{m,m'}\alpha'_{n,n'}\circ_0\id_{\rF_m}+\delta_{n,n'}\id_{\rF'_n}\circ_0\alpha_{m,m'}\right)_M\\
&=\left(\delta_{n,n'}\bfM(\rF'_n)(\bfM(\alpha_{m,m'}))+\delta_{m,m'}\bfM(\alpha'_{n,n'})_{\bfM(\rF_m)(M)}\right).
\end{align*}
Hence the matrices encoding the differential are the same.

Now assume given a morphism $\Psi\colon \bfM\to \bfN$ of representations in $\cC\pamod$. We want to define a morphism of $p$-dg $2$-representations $\overline{\Psi}\colon \overline{\bfM}\to \overline{\bfN}$. For any object $\ti$ of $\cC$ we require a $p$-dg functor $\overline{\Psi}_\ti\colon \bfM(\ti)\to \bfN(\ti)$. However, we can simply use the original $p$-dg functor $\Psi_\ti$ as $\overline{\cC}$ acts on the same $p$-dg categories. 
We have to extend the definition of $\eta_\rF\colon \Psi_\tj\circ \bfM(\rF)\to \bfN(\rF)\circ\Psi_\ti$ to a general object $\rX$ of $\overline{\cC}$. It suffices to construct $\left(\eta_\rX\right)_M$ for an object $M$ of $\bfM(\ti)$ and extend to a natural transformation of functors defined on $\overline{\C_\ti}$ by Lemma \ref{extensionlemma}. For this, we set $\left(\eta_\rX\right)_M$ to be the diagonal matrix with diagonal entries $\eta_{\rF_m}(M)$. As $\overline{(-)}$ is a $p$-dg $2$-functor, $\eta_\rX$ will be a $p$-dg isomorphism. The compatibility condition \eqref{repmorcomp} follows from the one for $\eta_{\rF_m}$ as $\eta_\rX$ is a diagonal matrix, and remains valid after extending to all of $\overline{\C_\ti}$. It further follows that $\overline{\Psi'\Psi}=\overline{\Psi'}~\overline{\Psi}$ for compatible morphisms of $p$-dg $2$-representations.

Next, given a modification $\theta\colon \Psi\to \Psi'$, we want to induce a modification $\overline{\theta}\colon \overline{\Psi}\to \overline{\Psi'}$. As the morphism of $p$-dg functors $\overline{\theta}_\ti$ we can use $\theta_\ti$ since $\overline{\Psi}_\ti=\Psi_\ti$ as $p$-dg functors. It remains to check a more general version of \eqref{modificationcond} for objects $\rX$, $\rX'$ and a morphism $\gamma\colon \rX\to \rX'$ in $\overline{\cC}(\ti,\tj)$. But as $\eta_{\rX}$ is a diagonal matrix with $\eta_{\rF_m}$ appearing on the diagonal, the condition generalizes immediately, and the assignment commutes with the differential.

It is clear that $\bfInd$ is functorial with respect to horizontal and vertical composition of $2$-morphisms since the assignment is the identity on modifications. This also proves that the $p$-dg $2$-functor $\bfInd$ is $2$-faithful.
\end{proof}

\begin{proposition}
Let $\cC$ be a locally finite $2$-category and $\bfM$ an additive $p$-dg representation over $\cC$. Then $\overline{\bfM}$ descends to $2$-representation $\bfD\overline{\bfM}$ of $\cSt\overline{\cC}$. Further, the $2$-category $\cSt\overline\cC$ equals the category $\cD\overline{\cC}$, and we obtain a $2$-functor
\begin{align*}
\bfD \colon \cZ(\cC\pamod)&\longrightarrow \overline{\cC}\dmod, &\bfM\longmapsto\bfD\overline{\bfM}.
\end{align*}
\end{proposition}

\begin{proof}
First, since $\overline{\bfM}$ is a $p$-dg $2$-functor (i.e. it commutes with differentials on $2$-hom spaces), we obtain a $2$-functor $\bfZ\overline{\bfM}$ defined on $\cZ\overline{\cC}$, which associates to an object $\ti$ the category $\Z(\bfM(\ti))$.
It follows from Lemma \ref{derivedfunctor} that the $p$-dg functor $\overline{\bfM}(\rX)$ preserves null-homotopic morphisms and hence descends to a functor $\St(\bfM(\ti))\to \St(\bfM(\tj))$. Similarly to Lemma \ref{nullhomotopiccomposition}, we see that $\bfZ\overline{\bfM}$ maps null-homotopic $2$-morphisms in $\overline{\cC}$ to zero. Hence we obtain an induced $2$-functor $\bfD\overline{\bfM}$ defined on $\cSt\overline{\cC}$.

Let $\bfM=\bfP_\ti$ be a principal $2$-representation. We claim that the $2$-ideal $\cN$ of all $2$-morphisms in $\cZ\overline{\cC}$ which are mapped to zero under the $2$-functor $\bfD\overline{\bfP_\ti}\colon \cZ\cC\to \preaddcat$ (for any $\ti$) is precisely the $2$-ideal of null homotopic $2$-morphisms, which we denote by $\cN\,'$. This implies that $\cSt\overline{\cC}=\cD\overline{\cC}$.

Indeed, we saw above that if $\alpha$ is null homotopic, then it acts by zero in each $2$-representation $\bfD\overline{\bfM}$. This shows that $\cN\,'\subset \cN$. Now consider $\beta\in \overline{\cC(\ti,\tj)}$ such that $\St(\overline{\bfP_\ti}(\beta))=0$. This implies, in particular, that $\St(\overline{\bfP_\ti}(\beta))_{\one_\ti}=0$. But this means that $\beta\circ_0 \id_{\one_\ti}=\beta$ is null-homotopic. Hence $\beta\in \cN\,'$.
\end{proof}

Note that the above Proposition justifies to denote the $2$-representations induced on $\cSt(\overline{\cC})$ by $\bfD\overline{\bfM}$, which is the same notation used before for the $2$-representation induced on $\cD\overline{\cC}$.

\section{Simple transitive and \texorpdfstring{$p$}{p}-dg cell \texorpdfstring{$2$}{2}-representations}\label{cellchapter}

\subsection{Cell combinatorics}

For a strongly finitary $p$-dg $2$-category $\cC$, 
we write $\mathcal{S}(\cC)$ for the set of $p$-dg isomorphism classes of $\Bbbk$-indecomposable $1$-morphisms in $\cC$ up to shift. This set forms a multi-semigroup and can be equipped with several natural preorders as in \cite[Section~3]{MM2}. Namely, given two $\Bbbk$-indecomposable $1$-morphisms $\mathrm{F}$ and $\mathrm{G}$, we say 
$\mathrm{G}\geq_L\mathrm{F}$ in the {\em left preorder} if there is a $1$-morphism $\mathrm{H}$ such that 
$[\mathrm{G}]$ appears as a direct summand of $[\mathrm{H}]\circ [\mathrm{F}]$ in $[\cC]$. 
A {\bf left cell} is an equivalence class for this preorder. Similarly one defines the {\em right} and 
{\em two-sided} preorders $\geq_R$ and $\geq_J$ and the corresponding {\bf right} and {\bf two-sided} 
cells, respectively.  Note, in particular, that all $\Bbbk$-isomorphic indecomposable $1$-morphisms belong to the same left and right cell.

Observe that $\geq_L$ defines a genuine partial order on the set of left cells, and similarly  for $\geq_R$ and right cells, and for $\geq_J$ and two-sided cells.

We call a two-sided cell $\J$ \textbf{strongly regular}, provided that no two of its left (right) cells are comparable with respect to the left (right) order, and the intersection of any left and any right cell contains precisely one element of $\mathcal{S}(\cC)$.

For a fixed left cell $\L$ in $\cC$, notice that there exists a unique object $\ti_\L\in \cC$ such that each $\rF\in \L$ is in $\cC(\ti_\L,\tj)$ for some $\tj\in \cC$.

Note that by defining the $2$-cell structure using $\Bbbk$-indecomposable $1$-morphisms we use the same cell-combinatorics as in the underlying finitary $2$-category $[\cC]$.

\subsection{Construction of \texorpdfstring{$p$}{p}-dg cell \texorpdfstring{$2$}{2}-representations}\label{seccellrep}

 We assume that $\cC$ is strongly finitary for the remainder of this section.
Fix a left cell $\L$ in $\cC$ and set $\ti = \ti_\L$. Recall the $\ti$-th principal $2$-representation $\bfP_\ti\in \cC\pamod^{\mathrm{sf}}$ from Definition/Example \ref{principal}.

For each $\tj\in\cC$, let $\bfR_\L(\tj)$ be the full $p$-dg subcategory of $\overline{\cC(\ti,\tj)}$ given by the bar closure (cf. Definition \ref{barclosure}) of the set $\{\rF\rX | \rX\in\L,\rF\in\cC\}$. That is, objects in $\bfR_\L(\tj)$ are of the form $(\bigoplus_{m=1}^s \rG_m, \alpha)$ such that $[\rG_m]$ appears as a direct summand of $[\rF\rX]$, for some $\rX\in\L,\rF\in\cC$.

\begin{lemma}
This choice of $\bfR_\L(\tj)$ defines a $2$-subrepresentation $\bfR_\L$ of $\bfP_\ti$ in $\cC\pamod^{\mathrm{sf}}$.                                                                                                                                                                                                                                                                                                                                                                                                                                                                                                                                                                                                                                                                                                                                                  \end{lemma}

\proof
Note that $\bfR_\L(\tj)$ is strongly finitary since it is defined using the bar closure.
Now, the only thing needing to be checked is that $\coprod_{\tj \in \ccC}\bfR_\L(\tj)$ is closed under the action of $\cC$. This follows from the following transitivity observation. A $\Bbbk$-indecomposable $\rG$ is in the bar closure of $\{\rF\rX | \rX\in\L,\rF\in\cC\}$ if and only if $\rG \geq_L \rX$ for one, and hence all $\rX$ in the cell $\L$. But in this case, every $\Bbbk$-indecomposable $\rG'$ appearing in the bar closure of $\{\rH\rG|\rH\in\cC\}$ satisfies $\rG' \geq_L \rG\geq_L \rX$, so it also lies in $\bfR_\L$, and thanks to strong finitarity, this shows $\bfR_\L$ is $\cC$-stable.
\endproof

\begin{lemma}\label{Rmaxideal}
The set of $p$-dg ideals $\bfJ$ of $\bfR_\L$ such that 
$\bfJ$ does not contain $\id_\rF$ for any $\rF\in \L$
has a unique maximal element $\bfI_\L$.                                                                                                                                                                                                                                                                                                                                                                                                                                                                                                                                                                                                                                                                                                                                               \end{lemma}

\proof
Assume $\rF\in\L\cap\cC(\ti,\tj)$.
Since $\rF$ is $\Bbbk$-indecomposable, its endomorphism algebra is local, so for any $p$-dg ideal $\bfJ$ not containing $\id_\rF$, $\End_{\bfJ(\tj)}(\rF)$ is contained in the radical of $\End_{\bfR_\L(\tj)}(\rF)$ and closed under $\del$. Hence the sum of any two ideals not containing $\id_\rF$ for  any $\rF\in \L$ again satisfies these condition. Thus the sum of all $p$-dg ideals not containing $\id_\rF$ for  any $\rF\in \L$ is the unique maximal such $p$-dg ideal as desired.
\endproof

\begin{definition}\label{cell2rep}
We define $\bfC_\L:=\bfR_\L/\bfI_\L \in \cC\pamod^{\mathrm{sf}}$ to be the {\bf $p$-dg cell $2$-re\-pre\-sen\-tation} corresponding to $\L$.
\end{definition}

\begin{lemma}\label{biggercellskill}
Any $2$-sided cell not less than or equal to $\J$ annihilates $\bfC_\L$.
\end{lemma}
\proof Let $\rF\in \L\cap \cC(\ti,\tj)$ and consider the $[\cC]$-stable ideal in $[\bfR_\L]$ generated by $\id_\rG$ for $\rG$ from a cell not less than or equal to $\L$. 
Similarly to Lemma \ref{sameideal}, one checks that this is equal to $[\bfI]$ for $\bfI$ the $\cC$-stable $p$-dg ideal generated by $\id_\rG$.

Thus the component of $\bfI$ in $\End_{\bfR(\tj)}(\rF)$ is the same as that of the ideal generated in the underlying finitary category, and hence does not contain $\id_\rF$ by definition of the two-sided order. It is therefore contained in the kernel of $\bfC_\L$ . 
\endproof

\begin{proposition}\label{cellcomp} Let $\L$ be a left cell in a two-sided cell $\J$ of $\cC$ and denote by $[\L]$ and $[\J]$ the corresponding cells in $[\cC]$. 
\begin{enumerate}[$($i$)$]
\item\label{cellcomp1} The cell $2$-representation $\bfC^{[\ccC]}_{[\L]}$ of $[\cC]$ is equivalent to a quotient of $[\bfC_\L]$. 
\item\label{cellcomp2} Assume that  $[\cC]$ is weakly fiat and $[\J]$-simple and that $[\J]$ is strongly regular. Then $[\bfC_\L]$ is an inflation of $\bfC^{[\ccC]}_{[\L]}$ by a local algebra (cf. \cite[Section 3.6]{MM6}).
\end{enumerate}
\end{proposition}

\begin{proof}
Start by considering the principal $2$-representation $\bfP^{[\ccC]}_\ti$ of $[\cC]$ and the $2$-representation $[\bfP_\ti]$ (which is a $2$-functor as applying $[~]$ is $2$-functorial). Using the Yoneda Lemma from \cite[Lemma~9]{MM2}
we can construct a morphism of additive $2$-representations $\eta\colon \bfP^{[\ccC]}_\ti\to [\bfP_\ti]$ by sending $\one_\ti$ to $[\one_\ti]$. Then, for each object $\tj$, we obtain an additive functor 
$$\eta_{\ti}\colon \bfP^{[\ccC]}_\ti(\tj)=[\cC(\ti,\tj)]\longrightarrow [\bfP_\ti(\tj)]=[\overline{\cC(\ti,\tj)}].$$
This functor is an equivalence (cf. Section \ref{fincat}). Hence $\eta$ is an equivalence of additive $2$-representations.

Note that two $\Bbbk$-indecomposable $1$-morphisms $\rF$, $\rG$ of $\cC$ are $\Bbbk$-isomorphic if and only if $[\rF]$ and $[\rG]$ are isomorphic. Hence $\rF\geq_L\rG$ if and only if $[\rF]\geq_L[\rG]$. This shows that the indecomposable objects in $[\L]$ correspond to the set of indecomposable objects in a cell in $[\cC]$.
We can use these observation to see that an object $\rX=(\bigoplus_{m=1}^s\rF_m, \alpha)$ is in $\bfR_{\L}(\tj)$ if and only if $[\rX]$ is in $\bfR^{[\ccC]}_{[\L]}(\tj)$. This yields that the equivalence of $2$-representations of $\bfP^{[\ccC]}_\ti$ and $[\bfP_\ti]$ restricts to an equivalence of $\bfR^{[\ccC]}_{[\L]}$ and $[\bfR_\L]$.

Let $\bfI_\L$ be the maximal $p$-dg ideal from Lemma \ref{Rmaxideal}. We immediately see that $[\bfI_\L]\subseteq \bfI^{[\ccC]}_{[\L]}$. This proves \eqref{cellcomp1}.

For \eqref{cellcomp2}, notice that $ [\bfR_\L]$ is a transitive finitary $2$-representation by construction, and, thanks to the equivalence between $ [\bfR_\L]$ and $\bfR^{[\ccC]}_{[\L]}$, the simple transitive quotient of the latter is $\bfC^{[\ccC]}_{[\L]}$. Thus, $[\bfC_\L]$ is a transitive finitary $2$-representation of $[\cC]$ with simple transitive quotient $\bfC^{[\ccC]}_{[\L]}$.
Under the conditions on $[\cC]$ assumed in \eqref{cellcomp2}, \cite[Theorem~4]{MM6} asserts that $[\bfC_\L]$ is an inflation of $\bfC^{[\ccC]}_{[\L]}$.
\end{proof}

\begin{example}
We include an example where $[\bfC_\L]$ is not a cell $2$-representation.
Let $p=3$ and $K=\Bbbk[x]/(x^3)$, with differential determined by $\del(1)=0$ and $\del(x)=1$. Note that this requires that $\deg x=-2$. We may consider the $p$-dg $2$-category $\cC$ generated (under composition, addition, grading shifts and $p$-dg isomorphism), by one non-identity indecomposable $2$-morphism $\rF$ corresponding to the $K$-bimodule $K\otimes K$. The composition $\rF^2$ is given by $\left(\rF\oplus\rF\langle 2\rangle\oplus\rF\langle 4\rangle, \begin{pmatrix}
0&\id\langle -2\rangle&0\\0&0&2\id\langle -2\rangle\\0&0&0
\end{pmatrix}\right)$.
Note that $\cC$ has one non-identity left cell $\L$, which contains the $\Bbbk$-indecomposable $\rF$. To find the cell $2$-representation, we need to determine the ideal $\bfI_\L$ as in Lemma \ref{Rmaxideal}. However, $\End_{\ccC}(\rF)^\op\cong K\otimes K$ as a $p$-dg algebra. In $K\otimes K$, $\id_{\rF}=1\otimes 1$ lies in the image of the differential since $\del(1\otimes x)=\id_{\rF}$. Hence, $1\otimes x$ cannot be contained in $\bfI_\L$. However, in $[\cC]$, the ideal $\bfI_{[\L]}$ is given by matrices with entries of the form $K\otimes \rad K$, which contains $1\otimes x$. Hence $[\bfI_{\L}]$ is strictly contained in $\bfI_{[\L]}$, and $[\bfC_\L]$ is \emph{not} a cell $2$-representation (and not simple transitive) over $[\cC]$.
\end{example}

\subsection{Simple transitive \texorpdfstring{$p$}{p}-dg \texorpdfstring{$2$}{2}-representations}

We can now adapt the definition of a \emph{simple transitive} $2$-representation from \cite{MM5} to the $p$-dg setting.

\begin{definition} 
Let $\cC$ be a locally finite $2$-category and $\bfM \in \cC\pamod^{\mathrm{sf}}$. Fix $\ti\in\cC$ and $X\in \bfM(\ti)$. We denote by $\bfG_\bfM(X)$ the smallest strongly finitary $2$-subrepresentation of $\bfM$ containing $X$, that is, $\coprod_{\tj\in\ccC}\bfG_\bfM(X)(\tj)$ is the bar closure of $\{\rF X| \rF \in \cC\}$.
We say that $\bfM $ is {\bf transitive} if, for every  $\ti\in\cC$ and $X\in \bfM(\ti)$, 
we have $\bfM = \bfG_\bfM(X)$.
\end{definition}

\begin{lemma}
Let $\cC$ be a locally finite $2$-category and $\bfM \in \cC\pamod^{\mathrm{sf}}$. 
Then $\bfM$ is transitive if and only if  $[\bfM]$ is transitive as a $2$-representation of $[\cC]$ in the sense of \cite[3.1]{MM5}.
\end{lemma}

\proof
Assume $\bfM$ is transitive. Then for any $\ti\in\cC$ and $X\in \bfM(\ti)$, 
we have $\coprod_{\tj\in\ccC}\bfM(\tj) = \coprod_{\tj\in\ccC}\bfG_\bfM(X)(\tj)=\widehat{\S}$ for $\S$ the full $p$-dg subcategory on objects $\{\rF X| \rF \in \cC\}$. In particular 
$\left[\coprod_{\tj\in\ccC}\bfM(\tj)\right]=\left[\coprod_{\tj\in\ccC}\bfG_\bfM(X)(\tj)\right]=\left[\widehat{\S}\right]=\add\left([\S]\right)$ and $[\bfM]$ is transitive as a $2$-representation of $[\cC]$.

Conversely, assume $[\bfM]$ is transitive as a $2$-representation of $[\cC]$, that is $\add\left([\S]\right) = \left[\coprod_{\tj\in\ccC}\bfM(\tj)\right]$ for $\S=\{\rF X| \rF \in \cC\}$. Then $\widehat{\S}$ in particular contains a representative of each $p$-dg isomorphism class of $\Bbbk$-indecomposable objects in $\coprod_{\tj\in\ccC}\bfM(\tj)$, and thus, since every object in $\coprod_{\tj\in\ccC}\bfM(\tj)$ has a fantastic filtration by $\Bbbk$-indecomposables, $\widehat{\S} =\coprod_{\tj\in\ccC}\bfM(\tj)$ and $\bfM$ is transitive.
\endproof

\begin{lemma}\label{Mmaximalideal} Let $\cC$ be a locally finite $2$-category and assume that $\bfM \in \cC\pamod^{\mathrm{sf}}$ is transitive.
The set of $p$-dg ideals $\bfJ$ of $\bfM$ such that 
$\bfJ$ does not contain $\id_X$ for any $X\in \coprod_{\ti\in \ccC}\bfM(\ti)$
has a unique maximal element $\bfI_\bfM$.            
\end{lemma}

\proof
Since $\bfM(\ti)$ is strongly finitary and, in particular, every object has a fantastic filtration by $\Bbbk$-indecomposable objects,
the condition of not containing $\id_X$ for any $X\in \coprod_{\ti\in \ccC}\bfM(\ti)$ is equivalent to not containing $\id_Y$ for any $\Bbbk$-indecomposable $Y\in \coprod_{\ti\in \ccC}\bfM(\ti)$. Then the lemma follows using the same argument as in the proof of Lemma \ref{Rmaxideal}.
\endproof

\begin{remark}
Observe that in Lemmas \ref{Rmaxideal} and \ref{Mmaximalideal} we have crucially used the assumption that $\bfM$ is a strongly finitary $2$-representation. Otherwise, we might not have any object with local endomorphism ring in any of the $\bfM(\ti)$.
\end{remark}

\begin{definition} For a strongly finitary $2$-representation $\bfM$, we call $\bfM/\bfI_\bfM$ the {\bf simple transitive quotient} of $\bfM$ and say that $\bfM$ is {\bf simple transitive} if $\bfI_\bfM = 0$.
\end{definition}

We immediately have the following lemma.

\begin{lemma}
The cell $2$-representations of strongly finitary $p$-dg $2$-categories are simple transitive.
\end{lemma}

\subsection{Reduction to one cell}

In this section, we prove that in order to find the cell $2$-representations of a strongly finitary $p$-dg category $\cC$, under some additional conditions, it suffices to consider a certain restriction of $\cC$ that only has one $2$-sided cell apart from the identity cells. We therefore assume for this section that $[\cC]$ is (weakly) fiat and that all its $2$-sided cells are strongly regular.

Let $\J$ be a $2$-sided cell in $\cC$, and $\L$ a left cell in $\J$. Consider the cell $2$-representation $\bfC_\L$. 
Thanks to Lemma \ref{biggercellskill} we may, without loss of generality (if necessary passing to the quotient modulo the kernel of  $\bfC_\L$), assume that  $\J$  is the unique maximal $2$-sided cell of $\cC$. 

Denote by $\cC_\J$ the $2$-full $p$-dg $2$-subcategory of $\cC$ generated by all $1$-morphisms all of whose $\Bbbk$-indecomposable components are identity $1$-morphisms or in $\J$.

It then follows from \cite[Proposition 32]{MM1}, and the fact that cells are defined on the underlying finitary $2$-categories, that $\J$ is again a $2$-sided cell in $\cC_\J$ with the same left and right cells as $\J$ has in $\cC$. By restriction, the cell $2$-representation $\bfC_\L$ of $\cC$ becomes a strongly finitary $p$-dg $2$-representation $\bfC_{\L}$ of $\cC_\J$.

\begin{proposition}\label{redtoonecell}
The restriction of $\bfC_{\L}$ to $\cC_\J$ is the cell $2$-representation of  $\cC_\J$ for the left cell $\L$.
\end{proposition}

\proof 
In order to compute the cell $2$-representation $\bfC_{\L,\J}$ of  $\cC_\J$, we need to consider the $p$-dg $2$-representation $\bfR_{\L,\J}$ of $\cC_J$ as constructed in Section \ref{seccellrep} and take the simple transitive quotient. Notice that $\bfR_{\L,\J}$  is simply the restriction of $\bfR_{\L}$ to $\cC_\J$ by definition and hence ideals in the underlying category which are stable under $\del$ and do not contain the identity on any $2$-morphism in $\L$ coincide. Furthermore, any  ideal that is stable under $\cC$ is also stable under $\cC_J$. Therefore $\bfI_\L$ is contained in $\bfI_{\L,\J}$. 

Now assume there exists an ideal $\coprod_{\ti\in\ccC}\bfI(\ti)$ in $\coprod_{\ti\in\ccC}\bfR_{\L}(\ti)$ that does not contain  $\id_\rF$ for any $\rF\in \L$, is stable under $\del$ and under $\cC_J$, but not under $\cC$. Notice that in particular $[\coprod_{\ti\in\ccC}\bfI(\ti)]$ is stable under $[\cC_\J]$ and hence is contained in the maximal ideal of the fiat $2$-representation $[\bfR_\L]$ of $[\cC_\J]$, which needs to be factored out to obtain the (finitary) cell $2$-representation of $[\cC_\J]$. We know that the latter is equivalent to the restriction of the (finitary) cell $2$-representation of $[\cC]$ by \cite[Corollary 33]{MM1}, and hence $[\coprod_{\ti\in\ccC}\bfI(\ti)]$ is stable under $[\cC]$.
Now being stable under $[\cC]$ and under $\del$ implies that this ideal is stable under $\cC$.
\endproof

\subsection{Endomorphisms of \texorpdfstring{$p$-dg $2$-representations}{p-dg 2-representations}}\label{pdgendocats}

In this section, we record how results of \cite{MM3} about the endomorphism categories of cell $2$-representations can partially be extended to $p$-dg $2$-categories under particularly nice assumptions.
For this, let $\cC$ be a strongly finitary $p$-dg $2$-category such that $[\cC]$ is fiat and all $2$-sided cells are strongly regular. Let $\L$ be a left cell in $\cC$ and $\bfC_\L$ the associated cell $2$-representation. Let $\bfC_\L^{[\ccC]}$ be the cell $2$-representation of $[\cC]$ associated to $\L$.

\begin{proposition}\label{kaction}
For any $p$-dg $2$-representation $\bfM$ of $\cC$, we have a fully faithful functor $\Bbbk\plmod$ to $\End_{\ccC}(\bfM)$.
\end{proposition}

\proof
For $V\in \Bbbk\plmod$, define an endomorphism $\Psi^V$ of $\bfM$ as follows. The functor $\Psi^V_\ti$ is the $p$-dg endofunctor of $\bfM(\ti)$ given by tensoring with $V$ as in Section \ref{tensorproducts}. Choosing identities as coherences, this gives rise to an endomorphism of $\bfM$ and this assignment is obviously functorial and fully faithful.
\endproof

\begin{theorem}\label{nicecell}
If $\bfM$ is a $p$-dg $2$-representation of $\cC$ such that $[\bfM]$ is equivalent to $\bfC_\L^{[\ccC]}$, then the category $\End_{\ccC}(\bfM)$ is $p$-dg equivalent to $\Bbbk\plmod$.
\end{theorem}

Before proving the theorem, we state a technical lemma.

\begin{lemma}\label{kmod}
Assume $\C$ is a locally finite $p$-dg category such that $[\C]\simeq \Bbbk\lmod$, then $\C$ is $p$-dg equivalent to a full subcategory of $\Bbbk\plmod$.
\end{lemma}

\proof
We first claim that $\C$ is indeed strongly finitary and only has one $\Bbbk$-indecomposable object up to $p$-dg isomorphism. Indeed, thanks to the equivalence $[\C]\simeq \Bbbk\lmod$, we have at least one $\Bbbk$-indecomposable object $X$ with endomorphism ring $\Bbbk$. Consider the subquotient idempotent completion $\C^{\bullet}$ of $\C$ from Section \ref{fincat}, and take another $\Bbbk$-indecomposable object $Y$ in $\C^{\bullet}$. Since it is $\Bbbk$-isomorphic to $X$, we have an isomorphism $X\to Y$ in $\Hom_\C(X,Y)\cong \Bbbk$ and this is necessarily annihilated by $\del$ and thus a $p$-dg isomorphism. This proves that $\C^{\bullet}$  only has one $\Bbbk$-indecomposable object.

Under the Yoneda embedding $\C \to \C^\op\cof$, the image of an object $Z$ in $\C$ is $p$-dg isomorphic to a direct sum of copies of $P_X^\op$, together with a differential $\del \tI+\left((\alpha_{rs})_*\right))_{r,s}=\left((\alpha_{rs})_*\right))_{r,s}$ where all $\alpha_{rs}$ are necessarily scalars. Given nilpotency of the matrix $(\alpha_{rs})_{r,s}$, we can find a basis in our direct sum of copies of $P_X^\op$ such that it is upper triangular, and thus $\C$ is strongly finitary.

Let $\B$ be the additive subcategory of $\C$ consisting of direct sums of $\Bbbk$-indecomposable objects, thus $\B$ is $p$-dg equivalent to the additive closure of $X$ and equivalent to $\Bbbk\lmod$ (understood as a $p$-dg category with trivial differential). By the previous paragraph and Lemma \ref{ABsubs}, this implies $\overline{C}$ is $p$-dg equivalent to $\overline{\B}$, which is clearly $p$-dg equivalent to $\Bbbk\plmod$. As $\C$ is a full subcategory of $\overline{\C}$, the lemma follows.
\endproof

We now prove the theorem.

\begin{proof}[Proof of Theorem \ref{nicecell}]
First note that we have a fully faithful functor $[\End_{\ccC}(\bfM)]\to \End_{[\ccC]}([\bfM])$. Under the assumption of an equivalence between $[\bfM]$ and $\bfC_\L^{[\ccC]}$, we thus have a fully faithful functor $[\End_{\ccC}(\bfM)]\to\End_{[\ccC]}(\bfC_\L^{[\ccC]})$. The latter is equivalent to $\Bbbk\lmod$ by \cite[Theorem 16]{MM3}. By Lemma \ref{kmod}, $\End_{\ccC}(\bfM)$ is $p$-dg equivalent to a full subcategory of $\Bbbk\plmod$.
By Proposition \ref{kaction}, $\Bbbk\plmod$ is also a full subcategory of $\End_{\ccC}(\bfM)$. The composition $\Bbbk\plmod\hookrightarrow \End_{\ccC}(\bfM)\hookrightarrow \Bbbk\plmod$, sends $\Bbbk$ to the identity and back to $\Bbbk$, and the theorem follows.
\end{proof}

\section{The \texorpdfstring{$p$}{p}-dg \texorpdfstring{$2$}{2}-category \texorpdfstring{$\cC_\A$}{CA}}\label{CAsec}

\subsection{Recollections from the finitary world}

\begin{definition}\label{AsubbC}
Let $n\in\mathbb{N}$ and $A:=(A_1,A_2,\dots,A_n)$ be a collection of  connected, 
finite dimensional associative $\Bbbk$-algebras. For $i\in\{1,2,\dots,n\}$, choose some small category $\A_i$ equivalent to $A_i\text{-}\mathrm{proj}$.
Set $\A=(\A_1,\A_2,\dots,\A_n)$.
Denote by  $\AsubC$ the $2$-full finitary $2$-subcategory of $\mathfrak{A}_{\Bbbk}$ with objects $\A_i$, whose $1$-morphisms consist of functors isomorphic to direct sums of identity functors and  functors given tensoring with projective $A_i\text{-}A_j$-bimodules. The $2$-morphisms of $\AsubC$ are given by all natural transformations of such functors.
\end{definition}

\begin{remark}
In \cite{MM3} the algebras $A_i$ are assumed to be basic. However, it is easy to see that any different, but Morita equivalent, choice of algebras leads to a biequivalent $2$-category.
\end{remark}

By \cite[Theorem 13]{MM3}, a fiat $2$-category $\cC$, such that $\cC$ has only one two-sided cell $\J$ apart from the identities, this two-sided cell is strongly regular, and such that $\cC$ has no nonzero $2$-ideals not containing the identity on some $1$-morphisms (i.e. it is $\J$-simple),
 is biequivalent to $\AsubC$ for a suitable choice of self-injective $A$ (apart from possibly having slightly smaller endomorphism rings of the identities). 
Further, the natural (defining) representation $\bfN\colon \AsubC\to \cEnd(A\proj)$ (denoted by $\bfD$ in \cite{MM5}) is isomorphic to the cell representation $\bfC_\L$ for any left cell $\L$ in $\J$ by \cite[Proposition~9]{MM5}.

In this section, we generalize these definitions and results to the $p$-dg enriched setup.

\subsection{Definition of \texorpdfstring{$\cC_{\A}$}{CA}}

Motivated by \cite[Theorem 13]{MM3}, we will now define a particularly nice class of $p$-dg $2$-categories.
For this, suppose we are given a list  $\A_1,\dots \A_n$ of locally finite $p$-dg categories, each of the form considered in Section \ref{algcat}.  That is, we assume for any $i=1,\ldots, n$ the existence of  finite set of objects $\mathtt{X}^i$ such that every object in $\C$ has a fantastic filtration by shifts of objects in $\mathtt{X}^i$. Recall that in this case we can associate to $\A_i$ a finite-dimensional $p$-dg algebra $A_i$ as in Example \ref{algebraex}.
We set $\A:=\coprod_{i=1}^n \A_i$, and $\mathtt{X}:=\cup_{i=1}^n \mathtt{X}^i$.

\begin{definition}\label{AsubCgen}
We define the  $p$-dg $2$-category $\Csub{\A}$ whose
\begin{itemize}
\item objects $\ti$ are identified with the categories $\overline{\A_i}$ for $i=1,\dots, n$;
\item $1$-morphisms are the closure under sums and grading shifts of the identity functors and compositions of functors isomorphic to tensoring with objects in $\left\lbrace X\otimes Y\middle| X,Y \in \mathtt{X}\right\rbrace\subset \A\boxtimes \A^\op$, using the tensor action described in \eqref{bimoduletensor} of Section \ref{tensorproducts};
\item $2$-morphisms are $\Bbbk$-linear natural transformations (morphisms) of such functors.
\end{itemize}
\end{definition}
The $2$-category $\Csub{\A}$ comes with its {\bf defining} or {\bf natural representation} $\bfN: \Csub{\A} \to \cEnd(\overline{\A})$. Both $\Csub{\A}$ and its defining representation are strongly finitary if  $\A$ is.

\begin{lemma}
In $\Csub{\A}$, the endomorphism ring of $\one_\ti$ is $p$-dg isomorphic to the center $Z_i$ of the $p$-dg algebra $A_i$ associated to $\A_i$ as in Example \ref{algebraex}.
\end{lemma}
\begin{proof}
Using Lemma \ref{extensionlemma}\eqref{extensionlemma2} we see that a natural transformation $\lambda \colon \id_{\overline{\A_i}}\to \id_{\overline{\A_i}}$ consists of diagonal matrices $\lambda_X=(\lambda_{X_i})_i$ for an object $X=(\bigoplus_i{X_i},\alpha)$. By assumption, any object of $\A_i$ has a fantastic filtration by the objects in $\mathtt{X}^i$. Hence, the natural transformation is determined by the data of the $\lambda_X$ for $X\in \mathtt{X}^i$. The sum of these morphisms gives an element of $A_i$ which has to be in the center by naturality of $\lambda$. Conversely, any element of the center $Z_i$ of $A_i$ has to be a diagonal matrix and induces a natural transformation $\id_{\overline{\A_i}}\to \id_{\overline{\A_i}}$. These assignments give a $p$-dg isomorphism $\End(\one_\ti)^\op\cong Z_i$.
\end{proof}

As in \cite[Section~4.5]{MM3},  let $Z_i'$ be the subalgebra of $Z_i$ generated by the identity and all elements that factor through $1$-morphisms given by tensoring with $p$-dg indecomposable objects in $\A\boxtimes \A^\op$. We can slightly generalize $\Csub{\A}$ to $\Csub{\A, X}$, where $X=(X_1, \dots,X_n)$ is a list of $p$-dg subalgebras of $Z_i$ containing $Z_i'$, by letting $\Csub{\A, X}$ be the $2$-subcategory of $\Csub{\A}$ on the same objects, same $1$-morphisms and same $2$-morphisms except for the endomorphism rings of the $\one_\ti$, which are now given by $X_i$. In the following, we suppress the $X$ from the notation, but will always work in the generality of $\Csub{\A,X}$.

Observe that, up to possible variations in the endomorphism rings of identities,  $\Csub{\A}$ is the smallest $2$-full $p$-dg $2$-subcategory of the $2$-category of endofunctors  of $\overline{\A}$ that contains all functors $p$-dg isomorphic to tensoring with objects in $\A\boxtimes \A^\op$ isomorphic to tensoring with $X\otimes Y$, for $X,Y \in \mathtt{X}$. On the other extreme, we could take the smallest $2$-full $p$-dg $2$-subcategory of the $p$-dg $2$-category of endofunctors  of $\overline{\A}$ that contains all functors $p$-dg isomorphic to tensoring with any objects in $\overline{\A\boxtimes \A^\op}$ (recall the construction of this action from Section \ref{tensorproducts}). We denote this $2$-category by $\Csub{\A}^+$, respectively $\Csub{\A,X}^+$ if we want to specify endomorphism rings of identities. Its defining $p$-dg $2$-representation is the same as that of $\Csub{\A}$ and the same remarks about  local finiteness and strong finitarity apply.

Recall, from Section \ref{idempfant}, the subquotient idempotent completion $\A^{\bullet}$ of $\A$, and the category $\B$, which is the additive closure of $\Bbbk$-indecomposable objects in $\A^{\bullet}$.  
From Lemma \ref{ABsubs} and Corollary \ref{ABequiv}, we immediately have the following conclusions.

\begin{lemma}\label{cattoalg}$~$
\begin{enumerate}[(i)]
\item There  are natural $2$-full $p$-dg embeddings of $\Csub{\A}$ and $\Csub{\B}$  into $\Csub{\A^{\bullet}}$. Similarly, there  are natural $2$-full $p$-dg embeddings of $\Csub{\A}^+$ and $\Csub{\B}^+$  into $\Csub{\A^{\bullet}}^+$.
\item 
Viewed as a $2$-full $p$-dg $2$-subcategories of $\Csub{\A^{\bullet}}^+$, the $2$-category
$\Csub{\A}^+$  is (up to $p$-dg isomorphism) contained in $\Csub{\B}^+$  if and only if every $X\in \A$ has a fantastic filtration by objects in $\B$.
\item Viewed as $2$-full $p$-dg $2$-subcategories of $\Csub{\A^{\bullet}}^+$, the $2$-category
$\Csub{\B}^+$ is (up to $p$-dg isomorphism) contained in  $\Csub{\A}^+$   if and only if  for every $Y\in \B$, $P_Y^{\op}$ is semi-free over $\A$.
\item If every $X\in \A$ has a fantastic filtration by objects in $\B$ and, for every $Y\in \B$, $P_Y^{\op}$ is semi-free over $\A$, then the $2$-categories $\Csub{\A}^+$ 
 and $\Csub{\B}^+$ are $p$-dg biequivalent. 
\end{enumerate}
\end{lemma}

\begin{remark}\label{annihilatedidempotents}
For $\A=\coprod_{i=1}^n\A_i$ as above, such that all $\A_i$ are strongly finitary, consider the $p$-dg subcategories $\B_i$ of $\A_i$.
Note that in the $p$-dg algebra $B_i$ associated to $\B_i$ (using Example \ref{algebraex}), we can find a decomposition
\[
1=e_1+\ldots+e_s
\]
where $e_1,\ldots, e_s$ are pairwise orthogonal primitive idempotents annihilated by $\del$. They correspond to a choice of representatives of $p$-dg isomorphism classes of $\Bbbk$-indecomposable objects. In the following, we consider $B=\prod_{i=1}^nB_i$ and $\B=\coprod_{i=1}^n\B_i$.
\end{remark}

\begin{proposition}\label{AsubBprop}
In the setup of Remark \ref{annihilatedidempotents}, the $2$-categories $\cC_{B}$ and $[\cC_{\B}]$ are biequivalent.
\end{proposition}
\begin{proof}
We have that $[\overline{\B}]\simeq [\B]\simeq B\text{-proj}$, cf. Remark \ref{endostrongfin}. Hence, the defining $p$-dg $2$-representation $\bfN$ of $\cC_\B$ descends to a bifunctor $[\bfN^{\ccC_\B}]\colon [\cC_\B]\to \cEnd(B\proj)$ using Lemma \ref{functorbieq}. We further have a $2$-functor $\bfN^{\ccC_{B}}\colon \cC_{B}\to \cEnd(B\proj)$ using the defining $2$-representation of $\cC_{B}$. For a $\Bbbk$-indecomposable $1$-morphism $\rG\colon \ti\to \tj$ in $\cC_\B$, the functor $[\rG]$ corresponds, under the equivalences of $[\overline{\B_i}]$ and $B_i\proj$, to tensoring with a $B_j$-$B_i$-bimodule. This shows that the image of $[\bfN^{\ccC_\B}]$ in $\cEnd(B\proj)$ is contained in the image of $\bfN^{\ccC_{B}}$. The latter is a $2$-fully faithful $2$-representation, and hence we obtain a bifunctor from $[\cC_{\B}]$ to $\cC_{B}$. This bifunctor is a bijection on objects and restricts to an equivalence of categories between $[\cC_\B(\ti,\tj)]$ and $\cC_{B}(\ti,\tj)$. 
Indeed, the functor given by restricting to this $1$-hom category is fully faithful since for two $\Bbbk$-indecomposable $1$-morphisms $\rG_1$ and  $\rG_2$,
\begin{align*}
\Hom_{[\ccC_\B]}(\rG_1,\rG_2)\cong \Hom_{\ccC_{B}}(P_1,P_2),
\end{align*}
where $P_m$ is the projective $B\otimes B^\op$-module 
such that tensoring with $P_m$ is isomorphic to the image of $\rG_m$ under the bifunctor, $m=1,2$.
Further, this functor is dense by construction of $B$.
Thus, the bifunctor provides the desired biequivalence between $\cC_{B}$ and $[\cC_{\B}]$.
\end{proof}

\begin{proposition}\label{CAoverlineplus}
Let $\A$ be a strongly finitary $p$-dg category. Then $\overline{\cC_\A}$ is $p$-dg biequivalent to $\cC_\A^+$.
\end{proposition}

\begin{proof}
Consider the defining $p$-dg $2$-representation $\bfN\colon \cC_{\A}\to \cEnd(\overline{\A})$. Using Proposition \ref{overline2reps}, we obtain a $p$-dg $2$-functor $\overline{\bfN} \colon \overline{\cC_{\A}}\to\cEnd(\overline{\A})$. Note that $\overline{\bfN}(\ti)=\overline{\A_i}$, so the $2$-functor $\overline{\bfN}$ is just the identity on objects. For any pair of objects $\ti$, $\tj$ we obtain a commutative diagram of $p$-dg functors
$$
\xymatrix{
&\overline{\cC_{\A}}(\ti,\tj)\ar[rd]^{\overline{\bfN}}& \\
\cC_{\A}(\ti,\tj)\ar@{^{(}->}[ru] \ar@{^{(}->}[rd]\ar[rr]^{\bfN}&&\Hom(\overline{\A_\ti},\overline{\A_\tj}).\\
&\cC^+_{\A}(\ti,\tj)\ar@{^{(}->}[ru]_{\bfN^+}&
}
$$
By definition, $\bfN^+$, the natural $p$-dg $2$-representation of $\cC_\A^+$, is just the inclusion of $p$-dg functors obtained by tensoring with objects in $\overline{\A_\tj\boxtimes \A_\ti^\op}$. First note that $\overline{\bfN}$ is fully faithful. This follows as $\bfN$ is fully faithful by Lemma \ref{extensionlemma}\eqref{extensionlemma2}. We show that its image coincides with the essential image of the $p$-dg functor obtained by the restriction of $\bfN^+$ to prove that $\overline{\cC_{\A}}(\ti,\tj)$ is $p$-dg equivalent to $\cC_\A^+(\ti,\tj)$. Indeed, via Lemma \ref{matcatequiv} and the construction in Proposition \ref{overline2reps}, any functor $F$ in the image of $\overline{\bfN}$ has the form $\left(\bigoplus_m(X_m\otimes Y_m)\otimes_{\A_i}(-), \lambda\right)$, for $X_m\otimes Y_m$ objects in $\A_\tj\boxtimes \A_\ti^\op$, and $\lambda=(\lambda_{m',m})_{m',m}$ an upper triangular matrix consisting of natural transformations between these functors. The natural transformations $\lambda_{m',m}\colon (X_m\otimes Y_m)\otimes_{\A_i}(-)\to (X_{m'}\otimes Y_{m'})\otimes_{\A_i}(-)$ are induced by morphisms $\alpha_{m',m}\colon X_m\otimes Y_m\to X_{m'}\otimes Y_{m'}$ in $\A_\tj\boxtimes \A_{\ti}^\op$. This follows by identifying $\lambda_{m',m}$ with a morphism $P_X^\op\otimes P_Y\to P_{X'}^\op\otimes P_{Y'}$ in $\A_\tj^\op\boxtimes \A_\ti\cof$    and applying the enriched Yoneda Lemma from \cite{K}.
We can thus construct an object $\left(\bigoplus_m {X_m\otimes Y_m},(\alpha_{m',m})_{m',m}\right)$ in $\overline{\A_\tj\boxtimes \A_{\ti}^\op}$, and the $p$-dg functor obtained from its action via $\boxtimes_{\A_i}$ recovers $F$. Conversely, using Lemma \ref{ABsubs}\eqref{ABsubs1}, we see that every $p$-dg functor in $\cC^+_\A(\ti,\tj)$ has a fantastic filtration by functors in $\cC_\A(\ti,\tj)$ and hence the essential image of $\overline{\bfN}$ is contained in $\overline{\cC_\A}(\ti,\tj)$. Hence, by Lemma \ref{pdgbieqchar}, we obtain  a $p$-dg biequivalence as claimed.
\end{proof}

\subsection{\texorpdfstring{$p$}{p}-dg cell \texorpdfstring{$2$}{2}-representations of \texorpdfstring{$\cC_{\A}$}{CA}}

Throughout this subsection, we assume that $\A$ is strongly finitary and $\del(\rad \A)\subset \rad \A$. Note that in a more general situation, one can sometimes pass to $\B\subset \A^{\bullet}$ as in the previous section to ensure strong finitarity. An example where this is possible is discussed in Section \ref{sl2cat}.


\begin{remark}\label{cellidemp}
Notice that, by Proposition \ref{AsubBprop}, the cell structure in $\cC_{\A}$ is precisely the same as that in $\Csub{A}$, where $A=\prod_{i=1}^n A_i$ is the $p$-dg algebra associated to $\A=\coprod_{i=1}^n \A_i$ as in Example \ref{algebraex}, described in \cite[Section 5.1]{MM5}. Also notice that the structure described there does not need self-injectivity of $A$.
In the setup of the present paper we do not impose that $A$ is a basic algebra as  the same underlying indecomposable projective $A$-module can have different $p$-differentials so that the resulting modules are non $p$-dg isomorphic (but $\Bbbk$-isomorphic).
\end{remark}

We now fix a non-identity cell $\L$ in  $\cC_{\A}$.
By Remark \ref{cellidemp} we can find $\ti:=\ti_\L$ and a $\Bbbk$-isomorphism class $\S$ of $\Bbbk$-indecomposable objects in $\A_i$, such that  all $\Bbbk$-indecomposable $1$-morphisms in $\L$ are given by functors $\Bbbk$-isomorphic to tensoring with $X_s \otimes Y \in \A\boxtimes \A^\op$ for $Y\in \S$ and  any $\Bbbk$-indecomposable object $X_s$ of $\A_j$. Fix a set of representatives $X_{\L,1}, \ldots, X_{\L,d} $ of $p$-dg isomorphism classes in $\S$.
We denote by $e_{\L,1},\dots, e_ {\L,d}$ the idempotents corresponding to, respectively,  $\id_{X_{\L, 1}}, \dots \id_{X_{\L, d}}$ in the algebra $A_\ti$ associated to $\A_i$ via the construction of Example \ref{algebraex}. We further set $e_{\L}= e_{\L,1}+\cdots +e_ {\L,d}$.
Similarly for any $\Bbbk$-indecomposable object $X_t\in \A_j$, we write $e_t$ for the idempotent in $A_j$ corresponding to $\id_{X_t}$.

Furthermore, let $\rF$ denote the direct sum of a complete set of $p$-dg isomorphism classes of $\Bbbk$-indecomposable $1$-morphisms in $\J$, so that there is a $p$-dg isomorphism of $p$-dg algebras $\phi: \End_{\ccC_\A}(\rF) \to A^\op\otimes A (\cong \End_{A\text{-}A\text{-bimod}}(A\otimes A))$. 

Consider the $p$-dg $2$-representation $\bfR_\L$ of $\cC_\A$. An object in $\bfR_\L(\tj)$ is of the form  $(\bigoplus_{m=1}^a\rG_m, \alpha)$ for $\rG_m \in \L\cap \cC_\A(\ti,\tj)$ and each component of the matrix $\alpha$ corresponds, under the $p$-dg isomorphism $\phi$, to an element in $A_j\otimes e_\L A_ie_\L$. Similarly, morphisms $\gamma$ between such objects have entries which correspond to elements in $A_j\otimes e_\L A_ie_\L$ under $\phi$.
The ideal $\bfJ(\tj)$ in $\coprod_{\tj\in{\ccC_{\A}}}\bfR_\L(\tj)$ generated by those morphisms $\gamma$ whose components all lie in $A_j\otimes \rad e_\L A_ie_\L$ is $\del$-stable thanks to $\rad \A$ being closed under $\del$, and it is straightforward to check that it is
$\cC_\A$-stable and proper, using the fact that horizontal composition with $2$-morphisms is induced by tensoring on the left with elements in $A^\op\otimes A$ under $\phi$.
The ideals $\bfJ(\tj)$ therefore form a $p$-dg ideal $\bfJ$ of the $2$-representation $\bfR_\L$. 
We consider the quotient $2$-representation $\bfR_\L/\bfJ$.

\begin{lemma}\label{cellrepdes}The quotient $\bfR_\L/\bfJ$ is the $p$-dg cell $2$-representation $\bfC_\L$.
\end{lemma}

\proof
It suffices to prove that $\bfJ = \bfI_\L$ for the maximal $p$-dg ideal defined in Lemma \ref{Rmaxideal}. From the above considerations, we deduce that $\bfJ\subset \bfI_\L$. It remains to check that $\bfI_\L\subset \bfJ$. 
Suppose $\gamma
\colon \rX\to \rY$ is a morphism in $\bfR^{\AsubcC}_\L/\bfJ(\tj)$ between objects $\rX$ 
 and $\rY$ 
  and assume there is a matrix entry in $\gamma$ giving a morphism between two (without loss of generality $\Bbbk$-indecomposable) summands $\rF_t$ and $\rF_s$ of $\rX$ respectively $\rY$, 
which corresponds under $\phi$ from Remark \ref{cellidemp} to some $e_sae_t\otimes e_{\L,q} be_{\L,r} \in e_sA_je_t\otimes e_\L A_ie_\L $ (for some idempotents $e_s,e_t$ in $A_j$)  such that $e_{\L,q} be_{\L,r}$ is not in $e_\L \,\rad A_ie_\L $. 
Then the ideal generated by $\gamma$, in particular, contains 
this matrix component via composition with inclusions of and projections onto $\Bbbk$-indecomposable $1$-morphisms.
Tensoring with $\rX_u\otimes \rX_s$ (which corresponds to tensoring with $A_je_u\otimes e_sA_j$) for any $\Bbbk$-indecomposable $\rX_u\in \A$, and again pre- and post-composing with appropriate inclusions and projections, produces a $\Bbbk$-isomorphism of the form $e_{X_u}\otimes  e_{\L,q} b e_{\L,r} \colon X_u\boxtimes X_{\L,r} \to  X_u\boxtimes X_{\L,q}$ in the left cell $\L$. Composing this with $e_{X_u}\otimes  e_{\L,r} b'e_{\L,q}$, where $b'$ is an inverse $\Bbbk$-isomorphism to $b$, we obtain the identity on $X_u\boxtimes X_{\L,r}$ in our ideal.
This provides a contradiction to $\bfI_\L$ containing a morphism not contained in $\bfJ$.
Thus $\bfC_\L = \bfR^{\cCsub{\A}}_\L/\bfJ$, as claimed. \endproof

\begin{remark}
We would like to point out that in the cell $2$-representation, $ X_u\boxtimes X_{\L,r}$ and $ X_u\boxtimes X_{\L,q}$ in fact become $p$-dg isomorphic. The reason is that for a $\Bbbk$-isomorphism $e_{\L,q} be_{\L,e}$ between them $\del(e_{\L,q} be_{\L,e})$  is necessarily contained in the radical of $e_\L A_ie_\L$ by degree reasons, and is hence zero in the quotient.
\end{remark}

\begin{theorem}\label{cellthm}
The cell $2$-representation $\bfC_\L$ and $\bfN^{\ccC_\A}$ are $p$-dg equivalent.
\end{theorem}

\proof
We first consider the $p$-dg quotient completed $2$-representations $\vv{\bfP_\ti}$ and $\vv{\bfN^{\ccC_\A}}$, and define the morphism $$\vv{\Psi}\colon \vv{\bfP_\ti}\to \vv{\bfN^{\ccC_\A}}$$ in $\cC_\A\pcmod$ induced, using Lemma \ref{quotyoneda}, by sending the object $0 \to \one_i$ to a diagram $X \overset{\gamma}{\to} Y$ whose cokernel is isomorphic to the object corresponding to $\rS:=A_ie_{\L,q}/\rad A_ie_{\L,q}$ (for some fixed $q$) under the $p$-dg equivalence of Lemma \ref{vecmod}. Under the identification in Lemma \ref{vecmod}, applying a $\Bbbk$-indecomposable $1$-morphism $\rG$ in $\J$ to S produces either zero or a projective and, in particular, semi-free module, hence an object in $\A^{\op}\cof$. Using Lemma \ref{matcatequiv} and \eqref{equivcommute}, we see that  $\rG(X  \overset{\gamma}{\to} Y)$ is $p$-dg isomorphic to an object in $\overline{\A}$.
Hence $\vv{\Psi}$ restricts to a morphism $\Psi$ from $\bfR^{\ccC_\A}_\L$ to $\bfN^{\ccC_\A}$ in $\cC\pamod$. Note that, using the correspondences from \eqref{equivcommute}, if  the $\Bbbk$-indecomposable $\rG=X_s\otimes X_{\L,r}$ corresponds to the $A$-$A$-bimodule $A_je_s\otimes e_{\L,r} A_i$, we have $\rG(X  \overset{\gamma}{\to} Y)\cong X_s$.

The morphism $\Psi$ is $p$-dg dense at each object $\ti$, as any $\Bbbk$-indecomposable of $\bfN^{\ccC_\A}(\ti)=\overline{\A_\ti}$ is $p$-dg isomorphic to an object in the image by construction, and both the source and the target are complete under compact semi-free extensions. Similarly, it is $2$-full by construction.

We claim that he kernel of $\Psi$ is given by the ideal $\bfJ$ consisting of all morphisms $\gamma$ corresponding to matrices with entries in $A\otimes e_\L \,\rad A_i e_\L$ under $\phi$ from Remark \ref{cellidemp} and induces an equivalence between $\bfR^{\ccC_\A}_\L/\bfJ$ and $\bfN^{\ccC_\A}$. Indeed, notice that using Proposition \ref{AsubBprop}, $[\Psi]$ is precisely the morphism constructed in  \cite[Proposition~9]{MM5}, and that the proof there does not use self-injectivity of $A$ (or equivalently, the fact, that the $2$-category considered there is weakly fiat). Then the claim follows immediately from the same statement in the proof of \cite[Proposition~9]{MM5} since by construction our morphisms between two objects are the same as those between the direct sums of their $\Bbbk$-indecomposable components. This shows that the induced morphism in $\cC\pamod$ from $\bfR_\L/\bfJ$ to $\bfN^{\ccC_\A}$ is $2$-faithful and hence an equivalence. 
Applying Lemma \ref{cellrepdes}, we deduce the theorem.\endproof

\section{Application to \texorpdfstring{$\mathfrak{sl}_2$}{sl2}-categorification at roots of unity}\label{sl2cat}

In this section, we apply some of our constructions to Elias--Qi's categorification of Lusztig's idempotented form of the small quantum group $\dot{u}_q(\mathfrak{sl}_2)$ of $\mathfrak{sl}_2$ at a $p$-th root of unity, for $p$ prime \cite{EQ}. In particular, we can show that the cell $2$-representations of the  cyclotomic quotient of this categorification are given by the natural defining $2$-representation obtained from nil-Hecke algebras.

In \cite[Chapter~4]{EQ}, the categorification of quantum $\mathfrak{sl}_2$ of \cite{La} is equipped with $p$-differentials to give a $p$-dg $2$-category $\cU$. We fix a dominant integral weight $0\leq \lambda\leq p-1$ of $\mathfrak{sl}_2$. For finiteness reasons, we consider a cyclotomic quotient $p$-dg $2$-category $\cU_\lambda$ of $\cU$. This quotient is obtained by considering the $2$-representation $\mathbf{V}^\lambda$ which is used in \cite[Section~6.3]{EQ} (denoted there by $\mathcal{V}^\lambda$) to categorify the finite-dimensional simple graded $\dot{u}_q(\mathfrak{sl}_2)$-modules $V^\lambda$. More precisely, $\cU_\lambda$ is the $p$-dg $2$-category $\cU/\cI$, where $\cI$ is the kernel of the $p$-dg $2$-functor $\cU\to \cEnd(\mathbf{V}^\lambda)$ obtained from \cite{EQ}.

Thus, the $p$-dg $2$-category $\cU_\lambda$ has objects $\mu\in \mathbb{Z}$ such that $\mu\leq \lambda$, and $1$-morphisms are generated by $\one_{\mu+2} \rE\one_\mu$, $\one_{\mu-2}\rF\one_{\mu}$ (in \cite{EQ}, the symbols $\E$, respectively $\F$ are used). This $2$-category is \emph{not}  strongly finitary, but we can consider the cell structure of its  subquotient idempotent completion $\cU_\lambda^{\,\bullet}$ from Proposition \ref{2completion}, which is strongly finitary as by \cite[Section~5.2]{EQ} every indecomposable projective module in $\cU$ (and hence in $\cU_\lambda$) has a fantastic filtration by a (locally, for fixed objects) finite set of $\Bbbk$-indecomposables (up to grading shift and $p$-dg isomorphism). A full list of $p$-dg isomorphism classes of non-identity $\Bbbk$-indecomposable $1$-morphisms is given by $\one_{\mu+2r-2s} \rE^{(r)}\rF^{(s)}\one_\mu$, $0\leq r,s\leq p-1$ (see \cite[Equation~(6.8)]{EQ}).

The $2$-category $[\cU_\lambda^{\,\bullet}]$ can be identified with the $\cU_\lambda$ studied in \cite[Section~7.2]{MM5}. Note that all cells in $[\cU_\lambda^{\bullet}]$ are strongly regular (see e.g. \cite[Section 7.2]{MM5}). In \cite{MM5} it is also shown that the finitary additive $2$-category $[\cU_\lambda^{\,\bullet}]$ is fiat. By Proposition \ref{redtoonecell}, in order to understand cell $2$-representations of $\cU_\lambda^{\,\bullet}$, it therefore suffices to consider the restriction $\cU_{\lambda,\J}^{\,\,\bullet}$ of $\cU_\lambda^{\,\bullet}$ to one particular two-sided cell.

Let $\J$ be the lowest two-sided cell of $\cU_\lambda^{\,\bullet}$.
The $\Bbbk$-indecomposable $1$-morphisms in  $\J$ within the subquotient idempotent completion $\cU_\lambda^{\,\bullet}$ of 
$\cU_\lambda$ are given by $\rF^{(r)}\one_\lambda\rE^{(s)}$ for $0 \leq r,s \leq \lambda$ (denoted in \cite{EQ} by $\F^{(r)}\one_\lambda\E^{(s)}$) which --- identifying their action on representables with a tensor action of nil-Hecke algebra bimodules on $p$-dg module categories for nil-Hecke algebras --- correspond to the functors given by tensoring with $NH_r^\lambda\epsilon_r\otimes_\Bbbk \epsilon_s^*NH_s^\lambda$ as a $NH_r^\lambda$-$NH_s^\lambda$-bimodules  (see \cite[(59)]{KQ} for the definition of the idempotents $\epsilon_r$, $\epsilon_s^*$).

Since the $\Bbbk$-indecomposable $1$-morphisms appearing as factors in a $p$-dg indecomposable $1$-morphism in $\cU_\lambda$ all belong to the same left/right/two-sided cells, it makes sense to talk about cells in $\cU_\lambda$ (despite the fact that we previously only defined this for strongly finitary $2$-categories).
The lowest cell in $\cU_\lambda$ is then generated by the $1$-morphisms $\rF^{r}\one_\lambda\rE^{s}$ for $0 \leq r,s \leq \lambda$ which --- identifying their action on representables with a tensor action of nil-Hecke algebra bimodules on $p$-dg module categories for nil-Hecke algebras --- corresponds to the functor given by tensoring with $NH_r^\lambda\otimes_\Bbbk NH_s^\lambda$ as a $NH_r^\lambda$-$NH_s^\lambda$-bimodule.

We set $\A:=\A_\lambda$ to be the $p$-dg category with objects $\mathtt{1}, \dots, \mathtt{\lambda}$  where $\End_\A(\ti)^\op = NH_i^\lambda$ and $\Hom_\A(\ti,\tj)=0$ for $\ti\neq \tj$, then $\cU_{\lambda,\J}$ is $p$-dg biequivalent to $ \Csub{\A}$.

As before, let $\A^{\bullet}$ denote the subquotient idempotent completion of $\A$ and $\B$ the additive closure of the category of $\Bbbk$-indecomposable objects in $\A^{\bullet}$.  Then $\cU_{\lambda,\J}^{\,\,\bullet}$ is $p$-dg biequivalent to  $\Csub{\A^{\bullet}}$ .

Since, by \cite[Corollary~5.17]{EQ}, any $\rF^{r}$ and $\rE^{s}$ have fantastic filtrations by $\rF^{(r)}$ and $\rE^{(s)}$, respectively, we have $p$-dg biequivalences between $\Csub{\A}^+$ 
and  $ \Csub{\A^{\bullet}}^+$ 
and $\Csub{\B}^+$ (cf. also Lemma \ref{cattoalg}).

To summarize, we formulate a proposition.

\begin{proposition}\label{uqprop}
We have the following $p$-dg biequivalences of $p$-dg $2$-categories:
\begin{enumerate}
\item $\cU_{\lambda,\J}\approx \Csub{\A}$;
\item $\Csub{\A}^+\approx\Csub{\A^{\bullet}}^+\approx\Csub{\B}^+$;
\item $\cU_{\lambda,\J}^{\,\,\bullet}\approx \Csub{\A^{\bullet}}$.
\end{enumerate}
\end{proposition}

Moreover, by Proposition \ref{AsubBprop}, $[\Csub{\B}]$ is biequivalent to $\Csub{B}$  for $B$ the endomorphism algebra of a complete set of representatives of isomorphism classes of indecomposable $1$-morphisms in $B$, which by \cite[(6.14)]{EQ} is given by the coinvariant algebra $B:=\prod_{j=0}^\lambda H_{j, \lambda}$ (in the notation from \cite[Section 6.3]{EQ}). This is a strongly positive $p$-dg algebra, in particular, its idempotents are annihilated by $\del$ and its radical is stable under the differential (cf. Remark \ref{annihilatedidempotents}).  
Thus we know by Theorem \ref{cellrepdes} that the cell $2$-representation of $\Csub{\B}$ (and of $\Csub{\B}^+$) is just the natural $2$-representation. Thus the $p$-dg $2$-representation $\mathbf{V}^\lambda$ categorifying $V^\lambda$ gives the $p$-dg cell $2$-representation for the lowest two-sided cell of the $2$-category $\cU_{\lambda}^{\,\bullet}$.

\begin{theorem}
The endomorphism category of the cell $2$-representation of $\cU_\lambda^{\,\bullet}$, or equivalently, of the categorification of the finite-dimensional simple graded module $V^\lambda$ of $\dot{u}_q(\mathfrak{sl}_2)$ on cyclotomic $p$-dg NilHecke algebras, is equivalent to $\Bbbk\plmod$.
\end{theorem}
\begin{proof}
The cell $2$-representation for a non-identity left cell of $\cC_{\B}^+$ is the natural one on $\overline{\B}$, and, by Theorem \ref{nicecell}, its endomorphism category is $p$-dg equivalent to $\Bbbk\plmod$. Using the $p$-dg biequivalence from Proposition \ref{uqprop}, we see that the endomorphism category of a cell $2$-representation $\bfC_\L$  for a non-identity left cell  $\L$  of $\overline{\cU_{\lambda,\J}^{\,\,\bullet}}$ is also $p$-dg equivalent to $\Bbbk\plmod$. By Proposition \ref{redtoonecell}, the same statement is true for the restriction of $\bfC_\L$  to $\cU_{\lambda,\J}^{\,\bullet}$. Notice that the latter restriction is the categorification of the finite-dimensional simple graded module $V^\lambda$ of $\dot{u}_q(\mathfrak{sl}_2)$ constructed in \cite[Section 6.3]{EQ}.
\end{proof}



\begin{thebibliography}{99999}
\bibitem[BK]{BK} A. I. Bondal, M. M. Kapranov. Enhanced triangulated categories. Mat. Sb.  {\bf 181} (1990), no. 5, 669--683; transl. in Math. USSR-Sb.  {\bf 70} (1991), no. 1, 93--107.
\bibitem[CF]{CF} L.~Crane, I.~Frenkel. Four-dimensional topological 
quantum field theory, Hopf categories, and the canonical bases. 
Topology and physics. J. Math. Phys. {\bf 35} (1994), no. 10, 5136--5154.
\bibitem[EQ1]{EQ} B.~ Elias, Y.~Qi. An approach to categorification of some small quantum groups II. Advances in Mathematics {\bf 288} (2016),  81--151. 
\bibitem[EQ2]{EQ2} B.~Elias,  Y.~Qi. A categorification of quantum sl(2) at prime roots of unity. Advances in Mathematics {\bf 299} (2016), 863--930.
\bibitem[Fr]{Fr} P.~Freyd. Representations in abelian categories. 
 Proc. Conf. Categorical Algebra (1966), 95--120.
\bibitem[Ha]{Ha} D.~Happel. Triangulated categories in the representation theory of finite dimensional algebras. London Math. Soc. Lect. Note Ser. 119. Cambridge University Press, 1988.
\bibitem[Ka]{Ka} M.~Kapranov. On the q-analog of homological algebra. Preprint, arXiv:q-alg/9611005 (1996).
\bibitem[K]{K} G.M.~Kelly. Basic Concepts of Enriched Category Theory. Reprints in Theory and Applications of Categories, No. 10, 2005.  
\bibitem[Kh1]{Kh2} M.~Khovanov. A categorification of the Jones polynomial. Duke Math. J. {\bf 101} (2000), no. 3, 359--426.
\bibitem[Kh2]{Kh} M.~Khovanov. Hopfological algebra and categorification at a root of unity: the first steps. J. Knot Theory Ramifications {\bf 25} (2016), no. 3, 26 pp. 
\bibitem[KL1]{KL09} M.~Khovanov, A. Lauda. A diagrammatic approach to categorification of quantum groups I. Represent. Theory {\bf 13} (2009), 309--347. 
\bibitem[KL2]{KL} M.~Khovanov, A.~Lauda. A categorification of quantum $\mathfrak{sl}_n$. Quantum Topol. {\bf 1} (2010), 1--92.
\bibitem[KL3]{KL11} M. Khovanov, A. Lauda. A diagrammatic approach to categorification of quantum groups II. Trans. Amer. Math. Soc. {\bf 363} (2011), no. 5, 2685--2700.
\bibitem[KQ]{KQ} M.~Khovanov, Y.~Qi. An approach to categorification of some small quantum groups. Quantum Topol. {\bf 6} (2015), no. 2, 185--311.
\bibitem[KMMZ]{KMMZ} T. Kildetoft, M. Mackaay, V. Mazorchuk, J. Zimmermann. Simple transitive $2$-representations of small quotients of Soergel bimodules. Trans. Amer. Math. Soc. {\bf 371} (2019), no. 8, 5551--5590.
\bibitem[La]{La} A.~Lauda. A categorification of quantum ${\mathfrak sl}(2)$. Adv. Math. {\bf 225} (2010), no. 6, 
3327--3424.
\bibitem[Le]{Le}  T.~Leinster. Basic bicategories. Preprint, arXiv:math/9810017 (1998).
\bibitem[Lu]{Lu} G.~Lusztig. Introduction to Quantum Groups. Progress in Mathematics, vol. 110. Birkhäuser, Boston, MA, 1993.
\bibitem[MaMa]{MaMa} M. Mackaay, V. Mazorchuk. Simple transitive 2-representations for some 2-subcategories of Soergel bimodules.  J. Pure Appl. Algebra {\bf 221} (2017), no. 3, 565--587. 
\bibitem[Ma]{Ma} J.~P.~May. The Additivity of Traces in Triangulated Categories. Advances in Mathematics, {\bf 163} (2001), 34--73.
\bibitem[M1]{M1} W. Mayer. A new homology theory. Ann. of Math. (2) {\bf 43} (1942), no. 2, 370--380.
\bibitem[M2]{M2} W. Mayer. A new homology theory II, Ann. of Math. (2) {\bf 43} (1942), no. 3, 594--605.
\bibitem[Maz]{Maz} V. Mazorchuk. Classification problems in $2$-representation theory. S\~ao Paulo J. Math. Sci. {\bf 11} (2017), no. 1, 1--22.
 \bibitem[MM1]{MM1} V.~Mazorchuk, V.~Miemietz. Cell $2$-representations of finitary
 $2$-categories. Compositio Math. {\bf 147} (2011), 1519--1545.
 \bibitem[MM2]{MM2} V.~Mazorchuk, V.~Miemietz. Additive versus abelian $2$-representations of 
fiat $2$-ca\-te\-go\-ri\-es. Moscow Math. J. {\bf 14} (2014), no. 3, 595--615.
\bibitem[MM3]{MM3} V.~Mazorchuk, V.~Miemietz. Endomorphisms of cell $2$-representations.  IMRN (2016), no. 24, 7471--7498.
\bibitem[MM5]{MM5} V.~Mazorchuk, V.~Miemietz. Transitive representations of finitary $2$-categories. Trans. Amer. Math. Soc.  \textbf{368} (2016), no. 11, 7623--7644. 
\bibitem[MM6]{MM6} V.~Mazorchuk, V.~Miemietz. Isotypic faithful $2$-representations of $\J$-simple fiat $2$-categories. Math. Z. \textbf{282} (2016), no. 1, 411--434.
\bibitem[MMZ2]{MMZ2} V.~Mazorchuk, V.~Miemietz, X.~Zhang. Pyramids and $2$-representations. To appear in Rev. Mat. Iberoam. Available at \url{http://dx.doi.org/10.4171/rmi/1133} (2019).


\bibitem[MP]{MP} M.~Makkai, R.~Par\'e. Accessible categories: the foundations of categorical model theory. Contemporary Mathematics, {\bf 104} (1989). American Mathematical Society, Providence, RI.
\bibitem[Or]{Or} D.~Orlov. Smooth and proper noncommutative schemes and gluing of DG categories. Adv. Math. {\bf 302} (2016), 59--105.
\bibitem[Qi]{Qi} Y.~Qi. Hopfological algebra. Compos. Math. {\bf 150} (2014), no. 1, 1--45.
\bibitem[QS]{QS} Y.~Qi, J.~Sussan. A categorification of the Burau representation at prime roots of unity. Selecta Math. (N.S.) {\bf 22} (2016), no. 3, 1157--1193.
\bibitem[Ro]{Ro} R.~Rouquier. $2$-Kac-Moody algebras. Preprint, arXiv:0812.5023 (2008). 
\bibitem[RT]{RT} N.Y. Reshetikhin, V.G. Turaev. Ribbon graphs and their invariants derived from quantum groups. Comm. Math. Phys. {\bf 127} (1990), no. 1, 1--26.
\bibitem[Se]{Se} P. Seidel. Fukaya categories and Picard--Lefschetz theory. Zurich Lectures in Advanced Mathematics. European Mathematical Society (EMS), Z\"urich, 2008.
\bibitem[W]{W} E. Witten. Quantum field theory and the Jones polynomial. Comm. Math. Phys. {\bf 121} (1989), no. 3,  351--399.
\bibitem[Zh]{Zh} X. Zhang. Simple transitive $2$-representations and Drinfeld center for some finitary $2$-categories. J. Pure Appl. Algebra {\bf 222}, Issue 1 (2018), 97--130.
\bibitem[Zi]{Zi} J. Zimmermann. Simple transitive $2$-representations of Soergel bimodules in type $B_2$. J. Pure Appl. Algebra {\bf 221}, Issue 3 (2017), 666--690.
\end{thebibliography}
\end{document}